\documentclass[a4paper,11pt]{amsart}
\usepackage[english]{babel}

\usepackage[T1]{fontenc}

\usepackage{graphicx}
\usepackage{amssymb}

\usepackage{amsthm}
\usepackage{amsmath}
\usepackage[hyphens]{url}
\usepackage{hyperref}
\usepackage{enumerate}

\usepackage[noadjust]{cite}
 
\usepackage{geometry}
\usepackage[all]{xypic}

\usepackage[colorinlistoftodos]{todonotes}	%todonotes
\usepackage{tikz-cd}

\allowdisplaybreaks

\usepackage{tikz}
\usetikzlibrary{calc}

\newcommand{\m}{\mathfrak}
\newcommand{\mc}{\mathcal}
\newcommand{\bb}{\mathbb}
\newcommand{\what}{\widehat}

\newcommand\radice[2]{\hspace{-1.5pt}\sqrt[\uproot{2}#1]{#2}}
\newcommand{\infroot}[1]{\radice{\infty}{#1}}

\DeclareMathOperator{\lcm}{lcm} 			%lcm
\DeclareMathOperator{\spec}{Spec} 		%spec of a ring
\DeclareMathOperator{\Div}{{Div}} 		%fibered category of generalized divisors
\DeclareMathOperator{\Par}{Par} 			%category of parabolic sheaves
\DeclareMathOperator{\Coh}{Coh} 			%category of coherent sheaves
\DeclareMathOperator{\Sch}{Sch} 			%category of schemes
\DeclareMathOperator{\wt}{wt} 				%weights
\DeclareMathOperator{\gp}{gp}				 %associated group
\DeclareMathOperator{\Qcoh}{QCoh}	 	%category of quasi-coherent sheaves
\DeclareMathOperator{\Hom}{Hom} 			%morphisms of a category
\DeclareMathOperator{\D}{D} 				%D?????
\DeclareMathOperator{\Stab}{Stab} 			%stabilizier
\DeclareMathOperator{\id}{id}				%identity
\DeclareMathOperator{\Pic}{Pic} 			%picard group
\DeclareMathOperator{\cha}{char}			%characteristic of a field
\DeclareMathOperator{\rk}{rk}				%rank of something
\DeclareMathOperator{\FP}{FP}
\DeclareMathOperator{\Aff}{Aff}
\DeclareMathOperator{\gra}{gr}

\DeclareMathOperator{\End}{End}

\DeclareMathOperator{\et}{\acute{e}t} 		%small Žtale site 

\newcommand{\fPar}{\mathfrak{Par}}

\newcommand{\uqcoh}{\underline{\Qcoh}}	%fibered category of quasi coherent sheaves
\newcommand{\upar}{\underline{\Par}}		%fibered category of parabolic sheaves
\newcommand{\aff}{(\Aff)^{op}}

\newtheorem{thm}{Theorem}[subsection]
\newtheorem*{thm*}{Theorem}
\newtheorem*{ithm}{Ideal Theorem}

\newtheorem{lemma}[thm]{Lemma}
\newtheorem{prop}[thm]{Proposition}
\newtheorem{cor}[thm]{Corollary}

\theoremstyle{definition}
\newtheorem{definition}[thm]{Definition}
\newtheorem{question}[thm]{Question}
\newtheorem{notation}[thm]{Notation}
\newtheorem{assumptions}[thm]{Assumptions}

\newtheorem*{example*}{Example}
\newtheorem{example}[thm]{Example}

\newtheorem{rmk}[thm]{Remark}

\numberwithin{equation}{section}



\title[Moduli of parabolic sheaves]{Moduli of parabolic sheaves on a polarized\\ logarithmic scheme}
\author{Mattia Talpo}
\begin{document}

\begin{abstract}
We generalize the construction of moduli spaces of parabolic sheaves given by Maruyama and Yokogawa in \cite{maruyama-yokogawa} to the case of a projective fine saturated log scheme with a fixed global chart. Furthermore we construct moduli spaces of parabolic sheaves without fixing the weights.
\end{abstract}

\maketitle

%\listoftodos

\tableofcontents

\section{Introduction}

The aim of this paper is to generalize results of Maruyama and Yokogawa (\cite{maruyama-yokogawa}) about moduli spaces of parabolic sheaves.

Parabolic bundles were first defined by Mehta and Seshadri (\cite{metha-seshadri}) on a projective curve $X$ with a finite number of marked points $x_{1},\hdots,x_{r}\in X$ as vector bundles $E$ on $X$ with the additional data of a filtration 
$$
0=F_{i,k_{i}+1}\subset F_{i,k_{i}}\subset \cdots \subset F_{i,1}=E_{p_{i}}
$$
of the fiber of $E$ over each marked point $x_{i}$, together with a set of weights $0\leq a_{i,1}<\cdots<a_{i,k_{i}}<1$ (which are real numbers) associated to the pieces of the filtration. These objects came up in their work in order to generalize to open curves (the complement of the marked points) the classical correspondence between stable sheaves of rank $0$ on a projective complex curve and irreducible unitary representations of the topological fundamental group. The weights in their definition are real numbers, but for moduli theory one usually assumes that they are rational.

Their definition was later increasingly generalized by many people (\cite{maruyama-yokogawa, biswas, mochizuki, iyer-simpson, borne}) first to a projective variety together with a (simple normal crossing) divisor, and eventually to arbitrary logarithmic schemes (\cite{borne-vistoli}).

Maruyama and Yokogawa in \cite{maruyama-yokogawa} develop a moduli theory for parabolic sheaves on a smooth projective variety with an effective Cartier divisor (extending Mehta and Seshadri's results) by defining an appropriate notion of parabolic (semi-)stability and generalizing the classical GIT theory for moduli spaces of coherent sheaves on a projective variety (\cite{huybrechts-lehn}). Their final result is the construction of the moduli space of torsion-free (semi-)stable parabolic sheaves with fixed rational weights and fixed parabolic Hilbert polynomial.

Nironi (\cite[Section 7.2]{nironi}) recovers the results of Maruyama and Yokogawa by exploiting the correspondence between parabolic sheaves and quasi-coherent sheaves on a root stack of \cite{borne-vistoli}, and applying his results about moduli of coherent sheaves on a projective Deligne--Mumford stack.

We generalize the result of \cite{maruyama-yokogawa} in two ways: first we replace the log scheme obtained from an effective Cartier divisor with a general projective fine and saturated log scheme $X$ over a field $k$ with a global chart $P\to \Div(X)$ for the log structure and a chosen polarization (what we call a \emph{polarized log scheme}), and second, we consider moduli spaces for parabolic sheaves without fixing the weights.

For the first part, we generalize Nironi's approach and apply the correspondence between parabolic sheaves with respect to a Kummer extension $P\to Q$ and quasi-coherent sheaves on the root stack $\radice{Q}{X}$ of \cite{borne-vistoli} (the root stack is denoted by $X_{Q/P}$ there) to reduce the moduli problem to the one of coherent sheaves on the root stack. The prototypical example of Kummer extension $P\to Q$ is the inclusion $P\subseteq \frac{1}{n}P$, that corresponds to taking ``rational weights with denominators dividing $n$''. After this identification, we apply the theory of \cite{nironi} for moduli of coherent sheaves on a tame projective DM stack to the root stack $\radice{Q}{X}$. We remark that we have to assume that $\radice{Q}{X}$ is Deligne--Mumford; this happens if the order of the  cokernel of $P^{\gp}\to Q^{\gp}$ is not divisible by $\cha(k)$, so, in particular, if the characteristic of $k$ is $0$. It seems very plausible that Nironi's theory also applies to tame Artin stacks, so this assumption should be superfluous.

Following Nironi, we define a notion of (semi-)stability for parabolic sheaves with respect to $P\to Q$ by means of a generating sheaf on the root stack $\radice{Q}{X}$, and obtain a moduli space of (semi-)stable sheaves. The final product is the following theorem.

\begin{thm*}[Theorem \ref{thm.fixed.weights}]
Let $X$ be a (fine saturated) polarized log scheme over a field $k$, with global chart $P\to \Div(X)$, and with a Kummer morphism of fine saturated sharp monoids $P\to Q$. Moreover assume that the root stack $\radice{Q} X$ is Deligne--Mumford.

Then the stack  $\mc{M}^{ss}_{H}$ of semi-stable parabolic sheaves with Hilbert polynomial $H \in \bb{Q}[x]$ is an Artin stack of finite type over $k$, and it has a good (resp. adequate, in positive characteristic) moduli space in the sense of Alper (\cite{alper,alperadeq}), that we denote by $M^{ss}_{H}$. This moduli space is a projective scheme.

The open substack $\mc{M}^{s}_{H}\subseteq \mc{M}^{ss}_{H}$ of stable sheaves has a coarse moduli space $M^{s}_{H}$, which is an open subscheme of $M^{ss}_{H}$, and the map $\mc{M}^{s}_{H}\to M^{s}_{H}$ is a $\bb{G}_{m}$-gerbe.
\end{thm*}

For the second part, we have to analyze the behavior of (semi-)stability under extension of denominators of the weights. The idea is that parabolic sheaves with arbitrary rational weights correspond to quasi-coherent sheaves on the \emph{infinite root stack} $\infroot{X}$ (introduced in \cite{infinite}) of the log scheme $X$. This is an inverse limit of the finite root stacks $\radice{n}{X}=\radice{\frac{1}{n}P}{X}$ with respect to the natural projections $\radice{m}{X}\to \radice{n}{X}$ for $n\mid m$, so the natural approach is to cook up a moduli theory for (finitely presented) sheaves on $\infroot{X}$ as a ``direct limit'' of the moduli theories we have on each $\radice{n}{X}$, thanks to the preceding discussion. Here we have to assume that the characteristic of $k$ is $0$, in order for all the relevant root stacks to be Deligne--Mumford.

We carry out such a plan. For this we have to assume that the monoid $P$ giving the chart of $X$ is \emph{simplicial}, meaning that the extremal rays of the cone $P_{\bb{Q}}$ are linearly independent in $P^{\gp}_{\bb{Q}}$. This ensures the existence of a Kummer extension of the form $P\subseteq \bb{N}^{r}$ and we show that semi-stability is preserved by pullback along $\radice{\frac{1}{m}\bb{N}^{r}}{X}\to \radice{\frac{1}{n}\bb{N}^{r}}{X}$; the behavior of stability is more subtle. By using the root stacks $\{\radice{\frac{1}{n}\bb{N}^{r}}{X}\}_{n \in \bb{N}}$ instead of $\{\radice{n}{X}\}_{n\in \bb{N}}$ and exploiting the above, we are able to define a ``limit'' notion of (semi-)stability for finitely presented parabolic sheaves with rational weights, and obtain moduli spaces. In good cases they are a direct limit of the moduli spaces at finite level, with the transition maps being open and closed immersions. The following theorem contains our results in this direction.

\begin{thm*}[Theorem \ref{thm.parabolic.rational}]
Let $X$ be a (fine saturated) polarized log scheme over $k$ such that in the global chart $P\to \Div(X)$ the monoid $P$ is simplicial. The stack $\mc{M}^{ss}$ of semi-stable parabolic sheaves with rational weights is an Artin stack, locally of finite type, and it has an open substack $\mc{M}^{s}\subseteq \mc{M}^{ss}$ parametrizing stable sheaves.

Furthermore, If stability is preserved by pullback along the projections $\radice{\frac{1}{m}\bb{N}^{r}}{X}\to \radice{\frac{1}{n}\bb{N}^{r}}{X}$ between the root stacks of $X$ (for example if the log structure of $X$ is generically trivial, and we are considering torsion-free parabolic sheaves), then $\mc{M}^{ss}$ has a good moduli space $M^{ss}$, which is a disjoint union of projective schemes. Moreover there is an open subscheme $M^{s}\subseteq M^{ss}$ that is a coarse moduli space for the substack $\mc{M}^{s}$, and $\mc{M}^{s}\to M^{s}$ is a $\bb{G}_{m}$-gerbe.
\end{thm*}

\addtocontents{toc}{\protect\setcounter{tocdepth}{-1}}
\subsection*{Description of contents}

Let us describe the contents of each section in some more detail.

Section \ref{section.2} contains some preliminaries about logarithmic schemes, root stacks and parabolic sheaves. We will use a somewhat non-standard definition of log schemes, different from (but closely related to) Kato's definition (\cite{kato}), and relying on Deligne--Faltings structures. This was systematically studied in \cite{borne-vistoli}, which is our main reference. In \ref{subsec.rs.ps} we recall the definitions and properties of root stacks of a log scheme and of parabolic sheaves, and the relation between the two (Theorem \ref{BV}). Finally in \ref{irs} we recall the construction and properties of the infinite root stack, introduced in \cite{infinite}, and of parabolic sheaves with arbitrary rational weights.

In Section \ref{section.3} we consider the moduli theory of parabolic sheaves on a projective fine saturated log scheme with a global chart, and with respect to a fixed Kummer extension. To start with, in \ref{sec.3.1} we single out the concept of family of parabolic sheaves that we want to use, by generalizing some of the usual properties of coherent sheaves (mainly finite presentation, flatness and pureness) to the parabolic setting. This links the moduli problem for parabolic sheaves to the one for coherent sheaves on a tame projective DM stack. Next, in \ref{sec.3.3} we start to recall Nironi's strategy (\cite{nironi}) for constructing moduli spaces of coherent sheaves on such a stack, focusing first on the choice of a generating sheaf (\ref{maru.yoko},\ref{global.sec}), and then describing the notion of (semi-)stability and the results that we obtain about moduli spaces (\ref{results}, \ref{thm.fixed.weights}). We conclude this section with an example that shows that (semi-)stability depends on the global chart that we fix for the log structure (\ref{dependence}).

Section \ref{section.4} is devoted to the moduli theory of parabolic sheaves with arbitrary rational weights. We outline our strategy of exploiting the inverse limit of finite root stacks, and argue that in order to do that we have to impose some additional conditions on the log structure, for example to ensure flatness of the projections between the different root stacks. We restrict our attention to simplicial log structures, described in \ref{sec.4.1}, and under these hypotheses we analyze the behavior of semi-stability and stability under pullback along projections between root stacks (\ref{sec.4.2}, \ref{results.4.2}). The behavior of stability in particular is quite subtle, and requires a careful study. In \ref{sec.4.3} we apply the results of the previous discussion to define (semi-)stability for sheaves on the infinite root stack, i.e. parabolic sheaves with arbitrary rational weights, and draw results about the corresponding moduli stacks and spaces (\ref{thm.parabolic.rational}).

\subsection*{Acknowledgements}

This paper is taken from approximately half of my PhD thesis \cite{tesi}. I heartily thank my advisor Angelo Vistoli for his constant support and guidance through my PhD years.

Martin Olsson acted as my local supervisor during a visit to the University of California at Berkeley, in the spring term of 2012. I thank him for his hospitality and for the time he dedicated to me.

I am also grateful to Niels Borne for useful comments.

Finally I would like to thank the anonymous referee for a very careful reading and useful suggestions.

During the final stage of writing this article I was funded by the Max Planck Institute for Mathematics in Bonn, as a postdoctoral fellow.

\subsection*{Notations and conventions}

All monoids will be commutative, and often also fine, saturated and sharp (and hence torsion-free). A log scheme will be a scheme equipped with a Deligne--Faltings structure, in the sense of \cite{borne-vistoli}, or equivalently a quasi-integral log scheme in the sense of Kato (\cite{kato}). We will freely pass from one point of view to the other.

All schemes and stacks will be defined over a field $k$. In \ref{section.4} we will need to assume that the characteristic of $k$ is $0$ (see \ref{characteristic}).

We will denote by $(\Sch)$ and $(\Aff)$ the categories of schemes and affine schemes over $k$ respectively. The small \'{e}tale site of a scheme $X$ will be denoted by $X_{\et}$. If $x\colon \spec(K)\to X$ is a geometric point (often we will also write just $x\to X$) and $A$ is a sheaf on $X_{\et}$, we will denote by $A_{x}$ the stalk of $A$ at the geometric point $x$.

For symmetric monoidal categories we use the conventions described in \cite[Section 2.4]{borne-vistoli}, and for fibered categories we refer to \cite[Chapter 1]{vistoli}.

By ``finitely presented sheaf'' on a ringed site $(\mc{X},\mc{O}_{\mc{X}})$ we mean a sheaf of $\mc{O}_{\mc{X}}$-modules on $\mc{X}$ that has local presentations as a cokernel of a map between two  finite sums of copies of the structure sheaf, but not necessarily a global one (as in \cite[Tag 03DK]{stacks-project}).

A subscript will often be a shorthand for pullback along a morphism of schemes.

\addtocontents{toc}{\protect\setcounter{tocdepth}{2}}

\section{Log schemes, root stacks and parabolic sheaves}\label{section.2}

In this section we recall the definition and results of \cite{borne-vistoli} and \cite{infinite} about root stacks of logarithmic schemes and parabolic sheaves, as well as the somewhat non-standard notion of log scheme that we will use throughout.

We refer the interested reader to the original papers for a more detailed exposition.

\subsection{Log schemes}

Assume that $X$ is a scheme, and denote by $\Div_{X}$ the fibered category over $X_{\et}$ consisting of pairs $(L,s)$ where $L$ is an invertible sheaf and $s$ is a global section.

\begin{definition}
A \emph{Deligne-Faltings structure} on $X$ (abbreviated with DF from now on) is a symmetric monoidal functor $L\colon A\to \Div_{X}$ with trivial kernel, where $A$ is a sheaf of monoids on the small \'{e}tale site $X_{\et}$.

A \emph{logarithmic scheme} is a scheme $X$ equipped with a DF structure.
\end{definition}

The phrasing ``with trivial kernel''  means that for $U\to X$ \'{e}tale, the only section of $A(U)$ that goes to $(\mc{O}_{U},1)$ is the zero section.

We will denote a log scheme by $(X,A,L)$, or just by $X$ if the DF structure is understood. Sometimes we will also denote by $X$ the bare scheme underlying the log scheme, and we will not use the usual underlined notation $\underline{X}$. This should cause no real confusion.

\begin{rmk}
This definition (as well as everything that follows in this section) also makes sense when $X$ is an Artin stack. The only difference is that we have to use the lisse-\'{e}tale site of $X$ in place of the small \'{e}tale site. In the case of schemes or DM stacks, using the lisse-\'{e}tale site or the small \'{e}tale site produces the same theory if we restrict to fine log structures (for a proof see \cite[Proposition $5.3$]{olssonalgstacks}).

For most of this article we will mainly be concerned with log schemes, but from time to time the wording ``log stack'' will come up. For precise definitions, see \cite{olssonalgstacks}.
\end{rmk}

Recall that the standard definition of a log scheme (\cite{kato}) is that of a scheme $X$ with a sheaf of monoids $M$ on $X_{\et}$, with a morphism $\alpha\colon M\to \mc{O}_{X}$ (where $\mc{O}_{X}$ is a sheaf of monoids with the multiplication) such that the restriction of $\alpha$ to $\alpha^{-1}(\mc{O}_{X}^{\times})$ induces an isomorphism $\alpha|_{\alpha^{-1}(\mc{O}_{X}^{\times})}\colon \alpha^{-1}(\mc{O}_{X}^{\times})\to \mc{O}_{X}^{\times}$ with $\mc{O}_{X}^{\times}$. A log scheme in this sense is \emph{quasi-integral} if the natural resulting action of $\mc{O}_{X}^{\times}$ on $M$ is free.

The link between our definition and the standard notion of a quasi-integral log scheme is the following: given a morphism of sheaves of monoids $\alpha\colon M\to \mc{O}_{X}$, one takes the stacky quotient by $\mc{O}_{X}^{\times}$ to obtain  a symmetric monoidal functor $L\colon \overline{M}\to \mc{O}_{X}/\mc{O}_{X}^{\times} \simeq \Div_{X}$, and sets $A=\overline{M}$. In other words a section of $A$ is sent by $L$ to the dual $L_{a}$ of the invertible sheaf associated to the $\bb{G}_{m}$-torsor given by the fiber $M_{a}$ of $M\to \overline{M}=A$ over $a$, and the restriction of $\alpha$ to $M_{a}\to \mc{O}_{X}$ gives the section of $L_{a}$.

In the other direction, starting with a DF structure $L\colon A\to \Div_{X}$ we can take the fibered product $A\times_{\Div_{X}}\mc{O}_{X}\to \mc{O}_{X}$, and verify that $M=A\times_{\Div_{X}}\mc{O}_{X}$ is equivalent to a sheaf.

\begin{notation}
From now on we will always abbreviate ``logarithmic'' with ``log''. Also, because of the above we will use the wording ``DF structure'' and ``log structure'' interchangeably.
\end{notation}

\begin{definition}
A morphism $(A,L)\to (B,N)$ of DF structures on a scheme $X$ is a pair $(f,\alpha)$ where $f\colon A\to B$ is a sheaf of monoids, and $\alpha\colon N\circ f  \simeq M$ is a natural isomorphism of symmetric monoidal functors $A\to \Div_{X}$.
\end{definition}

\begin{rmk}
If $f\colon X\to Y$ is a morphism of schemes and $(A,L)$ is a DF structure on $X$, one defines a \emph{pullback} DF structure $f^{*}(A,L)=(f^{*}A,f^{*}L)$ (see \cite{borne-vistoli}). The sheaf of monoids $f^{*}A$ is, like the notation suggests, the usual pullback as a sheaf of sets on $X_{\et}$, and if $g\colon Z\to X$ is another morphism of schemes, there is a canonical isomorphism $g^{*}f^{*}(A,L) \simeq (fg)^{*}(A,L)$.
\end{rmk}

\begin{definition}
A morphism of log schemes $(X,A,L)\to (Y,B,N)$ is a pair $(f,f^{\flat})$, where $f\colon X\to Y$ is a morphism of schemes, and $f^{\flat}\colon f^{*}(B,N)\to (A,L)$ is a morphism of DF structures on $X$.
\end{definition}

Of course morphisms can be composed in the evident way, and log schemes form a category.

\begin{definition}
A morphism of log schemes $(f,f^{\flat})\colon (X,A,L)\to (Y,B,N)$ is \emph{strict} if $f^{\flat}$ is an isomorphism.
\end{definition}

Strict morphisms are morphisms that do not change the log structure.

We will be interested only in DF structures that arise from local models. Recall that a homomorphism of monoids $\phi\colon P\to Q$ is a \emph{cokernel} if the induced homomorphism $P/\phi^{-1}(0) \to Q$ is an isomorphism. A morphism $A\to B$ of sheaves of monoids on $X_{\et}$ is a cokernel if the induced homomorphism between the stalks are all cokernels.

\begin{definition}

A \emph{chart} for a sheaf of monoids $A$ on $X_{\et}$ is a homomorphism of monoids $P\to A(X)$ such that the induced map of sheaves $P_{X}\to A$ is a cokernel.

A sheaf of monoids $A$ on $X_{\et}$ is \emph{coherent} if $A$ has charts with finitely generated monoids locally for the \'{e}tale topology of $X$.

A log scheme $(X,A,L)$ is \emph{coherent} if the sheaf $A$ is coherent.
\end{definition}

\begin{rmk}
Equivalently, a \emph{chart} for a DF structure $(A,L)$ on $X$ can be seen as a symmetric monoidal functor $P\to \Div(X)$ for a monoid $P$, that induces the functor $L\colon A\to \Div_{X}$ by sheafifying and trivializing the kernel.

\end{rmk}

This differs from the standard notion of chart for a log scheme, which is a morphism of monoids $P\to \mc{O}_{X}(X)$ that induces the log structure $\alpha\colon M\to \mc{O}_{X}$ (see \cite[Section 1]{kato} for details). We will call distinguish the two notions by calling the standard charts \emph{Kato charts}. Every Kato chart $P\to \mc{O}_{X}(X)$ induces a chart by composing with $\mc{O}_{X}(X)\to \Div(X)$.

Moreover one can show that having charts \'{e}tale locally is equivalent to having Kato charts \'{e}tale locally \cite[Proposition $3.28$]{borne-vistoli}.

\begin{rmk}
Note that for a monoid $P$, the scheme $\spec(k[P])$ has a natural DF structure induced by the composite $P\to k[P]\to \Div(\spec(k[P]))$. 

Giving a Kato chart on $X$ with monoid $P$ is the same as giving a strict morphism of log schemes $X\to \spec(k[P])$, and giving a chart is the same as giving a strict morphism $X\to [\spec(k[P])/\what{P}]$, where $\what{P}$ is the Cartier dual of $P^{\gp}$, the action is the natural one, and the DF structure of the quotient stack is defined by descent from the one of $\spec(k[P])$.

From now on $\spec(k[P])$ and $ [\spec(k[P])/\what{P}]$ will be equipped with these DF structures without further mention.
\end{rmk}

One can show that on a coherent log scheme, charts can be constructed by taking as monoid $P$ the stalk of the sheaf $A$ over a geometric point of $X$. More precisely, any geometric point ${x}$ of $X$ has an \'{e}tale neighborhood where we have a chart with monoid $A_{x}$ (\cite[Proposition 3.15]{borne-vistoli}).

A coherent log scheme $(X,A,L)$ has a maximal open subset $U\subseteq X$ over which the restriction of log structure is trivial. This coincides with the open subset where the stalks of the sheaf $A$ are trivial.

\begin{definition}\label{gen.trivial}
A coherent log scheme $(X,A,L)$ has \emph{generically trivial} log structure if the open subscheme $U\subseteq X$ where the log structure is trivial is dense in $X$.
\end{definition}

Using charts one can also describe the logarithmic part of morphisms between coherent log schemes by using homomorphisms of monoids.

\begin{definition}
A \emph{chart} for a morphism of sheaves of monoids $A\to B$ on $X_{\et}$ is given by two charts $P\to A(X)$ and $Q\to B(X)$ for $A$ and $B$, together with a homomorphism of monoids $P\to Q$ making the diagram
$$
\xymatrix{
P\ar[r]\ar[d] & Q\ar[d]\\
A(X)\ar[r] & B(X)
}
$$
commutative.

A \emph{chart} for a morphism of log schemes  $(f,f^{\flat})\colon (X,A,L)\to (Y,B,N)$ is a chart for the morphism of sheaves of monoids $f^{*}B\to A$ given by $f^{\flat}$.
\end{definition}

In other words, a chart for a morphism of log schemes $(f,f^{\flat})\colon (X,A,L)\to (Y,B,N)$ can be seen as two symmetric monoidal functors $P\to \Div(Y)$ and $Q\to \Div(X)$ that are charts for $(B,N)$ and $(A,L)$ respectively, and a morphism of monoids $P\to Q$ inducing $f^{\flat}\colon f^{*}(B,N)\to (A,L)$.

\begin{definition}
A \emph{Kato chart} for a morphism of log schemes $(f,f^{\flat})\colon (X,A,L)\to (Y,B,N)$ is a chart such that the functors $P\to \Div(Y)$ and $Q\to \Div(X)$ lift to $P\to \mc{O}_{Y}(Y)$ and $Q\to \mc{O}_{X}(X)$.
\end{definition}

Equivalently a Kato chart can be seen as a commutative diagram of log schemes
$$
\xymatrix{
(X,A,L)\ar[r]\ar[d] & \spec(k[Q])\ar[d]\\
(Y,B,N)\ar[r] & \spec(k[P])
}
$$
with strict horizontal arrows, and analogously a chart can be seen as such a commutative diagram, with the quotient stacks $[\spec(k[P])/\what{P}]$ and $[\spec(k[Q])/\what{Q}]$ in place of  $\spec(k[P])$ and $\spec(k[Q])$ respectively.

One can show \cite[Proposition $3.17$]{borne-vistoli} that it is always possible to find local charts for morphisms $(X,A,L)\to (Y,B,N)$ between coherent log schemes, and moreover one can take the monoids of the charts to be stalks of the sheaves $A$ and $B$.

We will be dealing mostly with fine and saturated log structures.

\begin{definition}
A log scheme $(X,A,L)$ is \emph{fine} if it is coherent and the stalks of $A$ are fine monoids (i.e. integral and finitely generated).

A log scheme $(X,A,L)$ is \emph{fine and saturated} (often abbreviated by \emph{fs} from now on) if it is fine, and the stalks of $A$ are (fine and) saturated monoids.
\end{definition}

\begin{rmk}
Equivalently one can check these conditions on charts, i.e. a log scheme $(X,A,L)$ is fine (resp. fs) if and only it admits local charts by fine (resp. fine and saturated) monoids.
\end{rmk}

From now on all our log schemes will be assumed to be fs.

\begin{example}[Log points]
Let $k$ be a field and $P$ a sharp monoid. Then we have a DF structure on $\spec(k)$ given by the functor $P\to \Div({\spec(k)})$ sending $0$ to $(k,1)$ and everything else to $(k,0)$. This is the \emph{log point} on $\spec(k)$ with monoid $P$. If $k$ is algebraically closed, one can show that every fine saturated DF structure on $\spec(k)$ is of this form.

In particular if $P=\bb{N}$ we obtain what is usually called the \emph{standard log point} (over $k$). 
\end{example}

\begin{example}[Divisorial log structure]\label{divisorial.log.str}
Let $X$ be a regular scheme and $D\subseteq X$ an effective Cartier divisor. There is a natural log structure associated with $D$, given in Kato's language by taking $M=\{ f \in \mc{O}_{X} \mid f|_{X\setminus D} \in \mc{O}_{X\setminus D}^{\times}\}$ as a subsheaf of $\mc{O}_{X}$, and the inclusion $\alpha\colon M\to\mc{O}_{X}$. This will be called the \emph{divisorial log structure} associated with $D$ on $X$.

If $D$ is simple normal crossings and it has irreducible components $D_{1},\hdots,D_{r}$, then the corresponding DF structure admits a global chart $\bb{N}^{r}\to \Div(X)$, sending $e_{i}$ to $(\mc{O}_{X}(D_{i}),s_{i})$, where $s_{i}$ is the canonical section of $\mc{O}_{X}(D_{i})$, i.e. the image of $1$ along $\mc{O}_{X}\to \mc{O}_{X}(D_{i})$. If $D$ is normal crossings but not simple normal crossings, then the divisorial log structure is still fs, but does not admit a global chart.

Note that one can also consider the DF structure induced by $\bb{N}\to \Div(X)$ sending $1$ to $(\mc{O}_{X}(D),s)$. This DF structure does not coincide with the divisorial log structure unless $D$ is smooth (basically, it does not separate the different branches of the divisor near a singular point).
\end{example}

Note that log points never have generically trivial DF structure (unless it is trivial, i.e. $P=0$), and on the contrary a divisorial log structure is always generically trivial.

\subsection{Root stacks and parabolic sheaves}\label{subsec.rs.ps}

Here we recall the definition and some basic properties of root stacks and parabolic sheaves, and the relation between them. Our main reference is \cite{borne-vistoli}.

Root stacks are stacks that parametrize roots of the log structure of a log scheme $(X,A,L)$ with respect to a system of denominators $A\to B$.

\begin{definition}
A \emph{system of denominators} on a fs log scheme $(X,A,L)$ is a coherent sheaf of monoids $B$ on $X_{\et}$ with a Kummer morphism of sheaves of monoids $j\colon A\to B$ (i.e. the induced morphism between the stalks is Kummer at every point). 
\end{definition}

Recall that a homomorphism of monoids $P\to Q$ is Kummer it is injective, and for every $q \in Q$ there is a positive integer $n$ such that $nq$ is in the image.

The prototype of a Kummer extension is the inclusion $A\to \frac{1}{n}A$ where $\frac{1}{n}A$ is the sheaf that has sections $\frac{a}{n}$ with $a$ a section of $A$. Note that if $n\mid m$, then we have an inclusion $\frac{1}{n}A\subseteq \frac{1}{m}A$. The direct limit $\varinjlim_{n}\frac{1}{n}A$ will be denoted by $A_{\bb{Q}}$ (and we will use the same notation for a torsion-free monoid $P$). A section of $A_{\bb{Q}}$ is locally (if $X$ is quasi-compact, even globally) of the form $\frac{a}{n}$ for some positive integer $n \in \bb{N}$.

The inclusion $A\subseteq A_{\bb{Q}}$ is the maximal Kummer extension of $A$: in fact note that since a Kummer morphism $j\colon A\to B$ induces an isomorphism $A_{\bb{Q}} \simeq B_{\bb{Q}}$, the sheaf $B$ can be regarded as a subsheaf of $A_{\bb{Q}}$. From this point on, this identification will be implicit and used without further mention.
\begin{rmk}
Note that the definition does not require $B$ to be saturated. Nevertheless, most of the times we will deal with systems of denominators $A\to B$ on fs log schemes where $B$ is fine and saturated as well.
\end{rmk}
Since both $A$ and $B$ are coherent sheaves of monoids, \'{e}tale locally on $X$ the morphism $j\colon A\to B$ admits a chart $j_{0}\colon P\to Q$ where $j_{0}$ is Kummer and $P$ and $Q$ are fine and sharp monoids, giving charts for $A$ and $B$ respectively.

For example one can take $P$ and $Q$ to be stalks of $A$ and $B$ at a geometric point, and the morphism $j_0$ to be the morphism induced by $j$ on the stalks (see \cite[Proposition $3.17$]{borne-vistoli}).

\begin{definition}
The \emph{root stack} $\radice{B}{X}$ of the log scheme $(X,A,L)$ with respect to the system of denominators $j\colon A\to B$ is the fibered category over $(\Sch/X)$ defined as follows: for a scheme $t\colon T\to X$, the category $\radice{B}{X}(T)$ is the category pairs $(N,\alpha)$ where $N$ is a DF structure $N\colon t^{*}B\to \Div_{T}$ on $T$, and $\alpha$ is a natural isomorphism between the pullback $t^{*}L$ and the composition $t^{*}A\to t^{*}B\stackrel{N}{\to} \Div_{T}$, and the arrows are morphisms of DF structures, compatible with the natural isomorphisms. Pullback is defined by pullback of DF structures.
\end{definition}

In particular by taking $B=\frac{1}{n}A$, we obtain what we might call the $n$-th root stack of $X$. We will denote it by $\radice{n}{X}$.

Standard descent theory shows that $\radice{B}{X}$ is a stack for the \'{e}tale (or even the fpqc) topology of $(\Sch/X)$.

\begin{rmk}
Note that the stack $\radice{B}{X}$ has a natural DF structure, with sheaf of monoids $\pi^{*}B$, where $\pi\colon \radice{B}{X}\to X$ is the natural projection, and the functor $\Lambda\colon \pi^{*}B\to \Div_{\radice{B}{X}}$ is given over a smooth morphism $T\to \radice{B}{X}$ from a scheme $T$ by the functor $N$ of the definition above.

This will be called the \emph{universal DF structure} on $\radice{B}{X}$.
\end{rmk}

When there is a Kummer morphism $j_{0}\colon P\to Q$ giving a chart for the system of denominators, one can give an analogous definition using $P$ and $Q$ in place of $A$ and $B$. If $\radice{Q}{X}$ denotes this last stack, then we have a canonical isomorphism $\radice{B}{X} \simeq \radice{Q}{X}$.

We have the following description of $\radice{Q}{X}$ as a quotient stack.

The chart given by $P$ corresponds to a morphism $X\to \left[\spec(k[P])/\widehat{P} \right]$, where as usual $\widehat{P}=D(P^{\gp})$ is the diagonalizable group scheme associated to $P^{\gp}$, and the morphism $P\to Q$ induces a morphism of quotient stacks $\left[\spec(k[Q])/\widehat{Q} \right]\to \left[\spec(k[P])/\widehat{P} \right]$.

\begin{prop}[{\cite[Proposition $4.13$]{borne-vistoli}}]\label{local.model.root.stack}
With notations as above, we have an isomorphism
$$
\radice{Q}{X} \simeq X\times_{\left[\spec(k[P])/\widehat{P} \right]} \left[\spec(k[Q])/\widehat{Q} \right].
$$
\end{prop}

This gives a quotient stack description of $\radice{Q}{X}$ itself: call $\eta\colon E\to X$ the $\widehat{P}$-torsor corresponding to the map $X\to \left[\spec(k[P])/\widehat{P} \right]$, and note that we have a $\widehat{P}$-equivariant map $E\to \spec(k[P])$. Then we have an isomorphism
$$
\radice{Q}{X} \simeq \left[ (E\times_{\spec(k[P])}\spec(k[Q]))/\widehat{Q} \right]
$$
for the natural action.

Moreover $E$ is affine over $X$, and if we set $R=\eta_{*}\mc{O}_{E}$, then we have $E\times_{\spec(k[P])}\spec(k[Q]) \simeq \underline{\spec}_{X}(R\otimes_{k[P]}k[Q])$. This gives a description of quasi-coherent sheaves on $\radice{Q}{X}$, that is the key to the relation with parabolic sheaves, and will be important in what follows: quasi-coherent sheaves on $\radice{Q}{X}$ are $Q^{\gp}$-graded quasi-coherent sheaves on $X$, of modules over the sheaf of rings $R\otimes_{k[P]}k[Q]$. The grading corresponds to $\widehat{Q}$-equivariance.

We have a second description as a quotient stack in presence of a Kato chart: if $P\to \Div(X)$ comes from a Kato chart $P\to \mc{O}_{X}(X)$, then the cartesian diagram expressing $\radice{Q}{X}$ as a pullback can be broken up
$$
\xymatrix{
\radice{Q}{X}\ar[r]\ar[d] & \left[\spec(k[Q]) / \mu_{Q/P}\right] \ar[r]\ar[d]& \left[\spec(k[Q])/\widehat{Q} \right]\ar[d]\\
X\ar[r] & \spec(k[P])\ar[r] & \left[\spec(k[P])/\widehat{P} \right]
}
$$
in two cartesian squares, where $\mu_{Q/P}$ is the Cartier dual $D(C)$ of the cokernel $C$ of the morphism $P^{\gp}\to Q^{\gp}$, a finite abelian group. Consequently, we also have an isomorphism
$$
\radice{Q}{X} \simeq \left[(X\times_{\spec(k[P])}\spec(k[Q])) /\mu_{Q/P}\right]
$$
for the natural action.

This local description gives the following properties for the root stacks.

\begin{prop}[{\cite[Proposition $4.19$]{borne-vistoli}}]\label{root.stack.algebraic}
Let $(X,A,L)$ be a log scheme and $j\colon A\to B$ a system of denominators. Then $\radice{B}{X}$ is a tame Artin stack. Moreover the morphism $\radice{B}{X}\to X$ is proper, quasi-finite and finitely presented, and if for every geometric point $x \to X$ the order of the group $B^{\gp}_{x}/A^{\gp}_{x}$ is prime to the characteristic of $k$ (for example if $\cha(k)=0$), then $\radice{B}{X}$ is Deligne--\hspace{0pt}Mumford.
\end{prop}

Let us recall some further facts about root stacks that will be used later on. 

\begin{prop}\label{coarse.space.root}
Assume that $A$ and $B$ are sheaves of fine and saturated monoids. Then the coarse moduli space of $\radice{B}{X}$ is the morphism $\radice{B}{X}\to X$.
\end{prop}

\begin{proof}
This is a local question on $X$, so we can assume to have a chart $P\to \Div(X)$ for $X$ coming from a Kato chart, and a chart $P\to Q$ for the system of denominators. Moreover, since in this case
$$
\radice{Q}{X}  \simeq X\times_{\spec(k[P])}\left[\spec(k[Q]) /\mu_{Q/P}\right]
$$
with the notation introduced above, by tameness we can reduce to showing that the morphism $\left[\spec(k[Q]) /\mu_{Q/P}\right]\to \spec(k[P])$ is a coarse moduli space.

This follows from the fact that the invariants of the action of $\mu_{Q/P}$ on $\spec(k[Q])$ are exactly $\spec(k[P])$. Recall how the action is constructed: $\mu_{Q/P}$ is the Cartier dual $D(C)$ of the cokernel $C$ of $P^{\gp}\to Q^{\gp}$, the algebra $k[Q]$ has a natural $Q^{\gp}$-grading that induces a $C$-grading, and this gives the action of $\mu_{Q/P}$.

The invariants are the piece of degree $0$ with respect to this $C$-grading, and are generated by the $x^{q}$'s such that $q\in Q$ goes to $0$ in $C$, i.e. with $q \in P^{\gp}\cap Q=P$, since $P$ and $Q$ are fine and saturated. This concludes the proof.
\end{proof}

The last proposition implies in particular that $\pi_{*}\colon \Qcoh(\radice{B}{X})\to \Qcoh(X)$ is exact, since $\radice{B}{X}$ is tame and $X$ is the coarse moduli space.

\begin{prop}\label{strict.cartesian}
Let $(X,A,L)$ be a log scheme and $j\colon A\to B$ a system of denominators. If $f\colon Y \to X$ is a morphism of schemes and we equip $Y$ with the pullback DF structure and consider the system of denominators $f^{*}j\colon f^{*}A\to f^{*}B$, then we have an isomorphism $\radice{f^{*}B}{Y} \simeq \radice{B}{X}\times_{X}Y$, i.e. the diagram
$$
\xymatrix{
\radice{f^{*}B}{Y}\ar[r]\ar[d] & \radice{B}{X}\ar[d]\\
Y\ar[r] & X
}
$$
is cartesian.
\end{prop}

\begin{proof}
This is immediate from the definitions.
\end{proof}

if $A\to B$ and $B\to B'$ are systems of denominators, then the composition $A\to B'$ is also a system of denominators, and we have a morphism $\radice{B'}{X}\to \radice{B}{X}$ defined by restricting the extension of the DF structure $B'_{T}\to \Div_{T}$ along $B_{T}\to B'_{T}$ for every scheme $T\to X$. We will sometimes call this operation an ``extension of denominators''.

This morphism is very similar to the projection $\radice{B}{X}\to X$ from a root stack to the scheme $X$.

\begin{prop}
Let $j\colon A\to B$ and $j'\colon B\to B'$ be two systems of denominators over the log scheme $X$. Then the root stack $\radice{B'}{X}$ can be identified with the root stack of the log stack $\radice{B}{X}$ with respect to the system of denominators $j'$.
\end{prop}

\begin{proof}
This is also immediate from the definitions.
\end{proof}

Because of this, the map $\radice{B'}{X}\to \radice{B}{X}$ behaves in some sense as a coarse moduli space. For example, we have a projection formula for quasi-coherent sheaves.

\begin{prop}[Projection formula for the root stacks]\label{proj.formula.finite}
With the notation of the preceding proposition, denote by $\pi\colon \radice{B'}{X}\to \radice{B}{X}$ the natural map, and assume that $A$, $B$ and $B'$ are fine and saturated. Then:
\begin{itemize}
\item the canonical morphism $\mc{O}_{\radice{B}{X}}\to \pi_{*}\mc{O}_{\radice{B'}{X}}$ is an isomorphism,
\item if $F \in \Qcoh(\radice{B'}{X})$ and $G \in \Qcoh(\radice{B}{X})$ the natural morphism $\pi_{*}F\otimes G\to \pi_{*}(F\otimes \pi^{*}G)$ is an isomorphism,
\item consequently for $F \in \Qcoh(\radice{B}{X})$ the natural map $F\to\pi_{*}\pi^{*}F$ on $\radice{B}{X}$ is also an isomorphism.
\end{itemize}
\end{prop}

\begin{proof}
The last bullet is consequence of the first two.

By flat base change along $T\to \radice{B}{X}$, we reduce to proving the statements for $\pi_{T}\colon \radice{B'}{T}\to T$, where the log structure on $T$ is the pullback of the universal DF structure of $\radice{B}{X}$. Now since $B$ and $B'$ are fine and saturated, by \ref{coarse.space.root} the morphism $\pi_{T}$ is a coarse moduli space of a tame Artin stack, and the claims follow: the first one is a general property of coarse moduli spaces, and the second follows for example from the proof of Corollary 5.3 of \cite{olsson-starr}: we can assume that $T$ is affine and that $G$ is a colimit of finitely presented $\mc{O}_T$-modules, and using the fact that $(\pi_T)_*$ commutes with colimits we are reduced to the case where $G$ is finitely presented. By right exactness of $\pi_T^*$ and exactness of $(\pi_T)_*$ we reduce further to checking the statement for $G=\mc{O}_T$, for which it is trivial.
\end{proof}

Now let us turn to parabolic sheaves. The main result of \cite{borne-vistoli} is that quasi-coherent sheaves on $\radice{B}{X}$ are exactly parabolic sheaves on $X$ with respect to $B$.

We will give the general definition of a parabolic sheaf on a log scheme right away, without going through the earlier definitions by Seshadri (\cite{seshadri}), Maruyama and Yokogawa (\cite{maruyama-yokogawa}) and Borne (\cite{borne}). A comparison with Maruyama and Yokogawa's definition will be given in \ref{maru.yoko}.

Let us first define parabolic sheaves on a log scheme $(X,A,L)$ when there is a chart $P\to \Div(X)$ for $L$, and with respect to a Kummer morphism $P\to Q$.

Let us recall the construction of the category of weights $Q^{\wt}$ associated to $Q$: objects are elements of $Q^{\gp}$, and an arrow $a\to b$ is an element $q \in Q$ such that $b=a+q$. We will write $a\leq b$ to mean that there is an arrow from $a$ to $b$. Note that if $Q$ is integral, the element $q$ that gives the arrow is uniquely determined.

The symmetric monoidal functor $L\colon P\to \Div(X)$ giving the DF structure extends to a symmetric monoidal functor $L^{\wt}\colon P^{\wt}\to \Pic(X)$ in the obvious way. If $p \in P^{\gp}$, we denote $L^{\wt}(p)$ simply by $L_{p}$.

\begin{definition}\label{def.parabolic.sheaf.chart}
A \emph{parabolic sheaf} on $X$ with weights in $Q$ is a functor $E\colon Q^{\wt}\to \Qcoh(X)$ that we denote by $a\mapsto E_{a}$, for $a$ an object or an arrow of $Q^{\wt}$, with an additional datum for any $p \in P^{\gp}$ and $a \in Q^{\gp}$ of an isomorphism of $\mc{O}_{X}$-modules
$$
\rho^{E}_{p,a}\colon E_{p+a} \simeq L_{p}\otimes E_{a}
$$
called the pseudo-periods isomorphism.

These isomorphism are required to satisfy some compatibility conditions. Let $p,p'\in P^{\gp}$, $r \in P$, $q \in Q$ and $a \in Q^{\gp}$. Then the following diagrams are commutative
$$
\xymatrix@C=50pt{
E_{a} \ar[r]^{E_{r}}\ar[d] & E_{r+a}\ar[d]^{\rho^{E}_{r,a}}\\
\mc{O}_{X}\otimes E_{a}\ar[r]^{\sigma_{r}\otimes \id} & L_{r}\otimes E_{a}
}
$$
$$
\xymatrix@C=50pt{
E_{p+a}\ar[r]^{\rho^{E}_{p,a}}\ar[d]_{E_{q}} & L_{p}\otimes E_{a}\ar[d]^{\id\otimes E_{q}}\\
E_{p+q+a}\ar[r]^{\rho^{E}_{p,q+a}} & L_{p}\otimes E_{q+a}
}
$$
$$
\xymatrix@C=50pt{
E_{p+p'+a}\ar[r]^{\rho^{E}_{p+p',a}}\ar[d]_{\rho^{E}_{p,p'+a}} &L_{p+p'}\otimes E_{a}\ar[d]^{\mu_{p,p'}\otimes \id} \\
L_{p}\otimes E_{p'+a}\ar[r]^{\id\otimes \rho^{E}_{p',a}} & L_{p}\otimes L_{p'}\otimes E_{a},
}
$$
where $\mu_{p,p'}\colon L_{p+p'} \simeq L_{p}\otimes L_{p'}$ is the natural isomorphism given by $L$, and the composite
$$
\xymatrix{
E_{a}=E_{0+a}\ar[r]^<<<<<{\rho^{E}_{0,a}} &  L_{0}\otimes E_{a} \simeq \mc{O}_{X}\otimes E_{a}
}
$$
coincides with the natural isomorphism $E_{a} \simeq \mc{O}_{X}\otimes E_{a}$.
\end{definition}

There is a notion of morphisms of parabolic sheaves, which is a natural transformation $E\to E'$ between the two functors $Q^{\wt}\to \Qcoh(X)$ which is compatible with the pseudo-periods isomorphism in the obvious sense, so we get a category $\Par(X,j_{0})$ of parabolic sheaves on $X$ with respect to $j_{0}\colon P\to Q$. This is in fact an abelian category in the natural way.

\begin{example}
Let $X=\spec(k)$ be the standard log point, and let us take the Kummer extension $\bb{N}\subseteq \frac{1}{2}\bb{N}$.

In this case $Q^{\wt} \simeq \frac{1}{2}\bb{Z}$ as a partially ordered set in the natural way, and a parabolic sheaf $E\colon \frac{1}{2}\bb{Z} \to \Qcoh(\spec(k))$ is determined by its values at $0$ and $\frac{1}{2}$, since the pseudo-periods isomorphism gives for any $\frac{1}{2}k$ and $n \in \bb{N}$ an isomorphism
$$
E_{\frac{1}{2}k+n} \simeq E_{\frac{1}{2}k}.
$$
In other words we can visualize $E$ as a pair of vector spaces $V_{0}$ and $V_{1}$ with maps
$$
\xymatrix{
0& \frac{1}{2} & 1\\
V_{0}\ar[r]^{a}& V_{1}\ar[r]^{b}& V_{0}
}
$$
such that $a\circ b=0$ and $b\circ a=0$ (since these compositions have to coincide with multiplication by the image of $1$ in $k$, i.e. $0$.

\end{example}

\begin{example}\label{example.1}
Let $X$ be a regular scheme with $D\subseteq X$ a smooth Cartier divisor and consider the divisorial log structure given by $D$. This has a chart $\bb{N}\to \Div(X)$ sending $1$ to $(\mc{O}_{X}(D), s)$, and consider the Kummer extension $\bb{N}\subseteq \frac{1}{2}\bb{N}$.

Then a parabolic sheaf in this case consists of quasi-coherent sheaves $E_{\frac{1}{2}k}$ for any $k \in  \bb{Z}$, and of morphisms $E_{\frac{1}{2}k}\to E_{\frac{1}{2}k+\frac{1}{2}n}$ for any $n \in \bb{N}$, with the properties as in the definition. In particular if $m \in \bb{Z}$ we have an isomorphism $E_{\frac{1}{2}k+m} \simeq E_{\frac{1}{2}k}\otimes \mc{O}_{X}(mD)$, and the map $E_{\frac{1}{2}k}\to E_{\frac{1}{2}k+m}$ for positive $m$ corresponds to multiplication by $s^{\otimes m}$. Note that if the sheaves $E_{\frac{1}{2}k}$ are torsion-free (say $X$ is integral for simplicity) then, since $s$ is not a zero-divisor, all these maps will be injective.

Because of the pseudo-periods isomorphism, we can identify a parabolic sheaf with the data consisting of the sheaves $E_{0},E_{\frac{1}{2}}$ together with the two maps
$$
\xymatrix{
0& \frac{1}{2} & 1\\
E_{0}\ar[r] & E_{\frac{1}{2}}\ar[r] &E_{1}.
}
$$
The rest of the data is completely determined by this diagram.

Clearly we could have as well chosen the sheaves  corresponding to $-1,-\frac{1}{2}, 0$. We will see that for us it will be more convenient to identify a parabolic sheaf with the sheaves and maps in this second range.
\end{example}

The same definition with minor variations defines a parabolic sheaf in absence of a global chart. Starting from the sheaf $B$ one defines a weight category $B^{\wt}$ in analogy with the preceding case.

\begin{definition}\label{def.parabolic.sheaf}
A \emph{parabolic sheaf} on $X$ with weights in the sheaf $B$ is a cartesian functor $E\colon B^{\wt}\to \Qcoh_{X}$, together with the datum for every $U\to X$ \'{e}tale, any $p \in A^{\gp}(U)$ and $a\in B^{\gp}(U)$ of an isomorphism of $\mc{O}_{U}$-modules
$$
\rho^{E}_{p,a}\colon E_{p+a} \simeq L_{p}\otimes E_{a}
$$
called the pseudo-periods isomorphism.

These morphism satisfy the conditions analogous to those of the preceding definition, and the following one in addition: if $f\colon U\to V$ is a morphism over $X$ and we have $p\in A^{\gp}(V)$ and $a \in B^{\gp}(V)$, then the isomorphism
$$
\rho^{E}_{f^{*}p,f^{*}a}\colon E_{f^{*}(p+a)} \simeq L_{f^{*}p}\otimes E_{f^{*}a}
$$
is the pullback to $U$ of $\rho^{E}_{p,a}\colon E_{p+a} \simeq L_{p}\otimes E_{a}$.
\end{definition}

As for the preceding case there is a notion of morphism (a natural transformation compatible with the pseudo-periods isomorphisms) that gives a category $\Par(X,j)$ of parabolic sheaves on $X$ with weights in $B$, and this is an abelian category. This construction has some functoriality property with respect to morphisms of log schemes $X\to Y$. We will discuss briefly pullbacks at the beginning of Section \ref{section.3}.

Furthermore, $\Par(X,j)$ can be extended to a fibered category $\fPar(X,j)$ over the small \'{e}tale site $X_{\et}$ by taking over an \'{e}tale morphism $U\to X$ the category $\Par(U,j|_{U})$ where $U$ has the pullback log structure. This fibered category is a stack for the \'{e}tale topology, by standard arguments of descent theory.

In case we have a global chart, we can use either one of the definitions.

\begin{prop}[{\cite[Proposition $5.10$]{borne-vistoli}}]
Let $(X,A,L)$ be a log scheme with a system of denominators $j\colon A\to B$, admitting a global chart $j_{0}\colon P\to Q$. Then we have an equivalence $\Par(X,j) \simeq \Par(X,j_{0})$.
\end{prop}

This says that when dealing with local statements about parabolic sheaves, we can assume that they are relative to a chart.

The following is the main result of \cite{borne-vistoli}, and relates parabolic sheaves on $X$ with respect to $j\colon A\to B$ to quasi-coherent sheaves on the root stack $\radice{B}{X}$.

\begin{thm}[{\cite[Theorem $6.1$]{borne-vistoli}}]\label{BV}
Let $(X,A,L)$ be a log scheme with a system of denominators $j\colon A\to B$. Then there is an equivalence of abelian categories $\Phi\colon \Qcoh(\radice{B}{X})\to \Par(X,j)$.
\end{thm}

We sketch the proof here, since the construction of the two functors will be important in what follows.

\begin{proof}[Sketch of proof]
Let us denote by $\pi\colon \radice{B}{X}\to X$ the projection, and by $\Lambda\colon \pi^{*}B\to \Div_{\radice{B}{X}}$ the universal DF structure on the root stack $\radice{B}{X}$.

Let us describe the functor $\Phi$. Given a quasi-coherent sheaf $F$ on $\radice{B}{X}$, we want to get a parabolic sheaf $\Phi(F)$. We set, for $U\to X$ \'{e}tale and $b \in B^{\gp}(U)$
$$
\Phi(F)_{b}=\pi_{*}(F\otimes \Lambda_{b}).
$$
This gives a cartesian functor $B^{\wt}\to \Qcoh_{X}$ by means of the maps $\Lambda_{b}\to \Lambda_{b+b'}$ for $b \in B^{\gp}$ and $b' \in B$, and there is a pseudo-periods isomorphism, basically coming from the fact that if $a \in A(U)$, then $\Lambda_{a} \simeq \pi^{*}L_{a}$, and using the projection formula for $\pi$.

Now since parabolic sheaves on $X$ and quasi-coherent sheaves on $\radice{B}{X}$ form a stack in the \'{e}tale topology of $X$, we can construct the quasi-inverse \'{e}tale locally, and so we can assume that we have a chart $j_{0}\colon P\to Q$ for the system of denominators. In this case recall that we have an isomorphism
$$
\radice{Q}{X} \simeq \left[\underline{\spec}_{X}(R\otimes_{k[P]}k[Q]) / \widehat{Q} \right]
$$
(see the discussion after \ref{local.model.root.stack} for the notation) and consequently quasi-coherent sheaves on $\radice{Q}{X}$ are $Q^{\gp}$-graded quasi-coherent sheaves on $X$, which are modules over the sheaf of rings $R\otimes_{k[P]}k[Q]$.

Starting from a parabolic sheaf $E \in \Par(X,j_{0})$, we define $\Psi(E)$ as the direct sum $\bigoplus_{q\in Q^{\gp}} E_{q}$. This is a $Q^{\gp}$-graded quasi-coherent sheaf on $X$, that has a structure of $R$-module (this uses the pseudo-periods isomorphism). Moreover it is also a sheaf of $k[Q]$-modules in the natural way, and the two actions are compatible over $k[P]$ by the properties of parabolic sheaves. This gives $\Psi(E)$ the structure of a $Q^{\gp}$-graded quasi-coherent sheaf of $R\otimes_{k[P]}k[Q]$-modules, i.e. of a quasi-coherent sheaf on $\radice{Q}{X}$.

One checks that these two constructions are quasi-inverses, and thus give equivalences.
\end{proof}

From the proof of this theorem we see that if the log structure of a noetherian log scheme $X$ is generically trivial, then the maps $E_{b}\to E_{b'}$ between the pieces of any parabolic sheaf are generically isomorphisms (over the open set where the log structure is trivial). Moreover in this case if we also assume that the maps $E_{b}\to E_{b'}$ are injective (this will be automatic for torsion-free parabolic sheaves, see Proposition \ref{gen.trivial.inj}) the pieces $E_{b}$ cannot be zero, unless the whole parabolic sheaf is.

This is not true in general, as easy examples show.

\begin{notation}
In the following we will systematically use the notation $\Phi$ and $\Psi$ for the equivalences of Theorem \ref{BV}, regardless of the log scheme $X$ they refer to. This should cause no confusion.

Moreover we will refer to both these functors as ``the BV equivalence'', for Borne-Vistoli.
\end{notation}

\begin{rmk}
A variant of this theorem also holds for log stacks: if $X$ is a log stack with a system of denominators $j\colon A\to B$, there is an equivalence between parabolic sheaves on $X$ with respect to $j$ and quasi-coherent sheaves on the root stack $\radice{B}{X}$. The proof is just a matter of taking an atlas and keeping track of descent data. We will use this without further comment, especially in Section \ref{section.4}.
\end{rmk}

To conclude, let us describe in parabolic terms pushforwards and pullbacks between root stacks: let $j\colon A\to B$ and $B\to B'$ be systems of denominators on $X$, and consider their composite $j'\colon A\to B'$ and the canonical map $\pi\colon \radice{B'}{X}\to \radice{B}{X}$.

We have a functor $F\colon \Par(X,j')\to \Par(X,j)$ given by ``restriction'': we have an inclusion $B^{\gp}\to B'^{\gp}$, that identifies $B^{\wt}$ with a subcategory of $B'^{\wt}$. Consequently given a parabolic sheaf $E \in \Par(X,j')$ we can restrict the functor $E\colon B'^{\wt}\to \Qcoh_{X}$ to $B^{\wt}$, and one checks that, together with the induced pseudo-periods isomorphism, this gives a parabolic sheaf in $\Par(X,j)$.

\begin{prop}
The functor $F$ described in the preceding discussion corresponds to the pushforward functor $\pi_{*}\colon \Qcoh(\radice{B'}{X})\to \Qcoh(\radice{B}{X})$.
\end{prop}

\begin{proof}
This is easily checked by using the definition of the functor $\Phi$ of \ref{BV}, and the projection formula \ref{proj.formula.finite}.
\end{proof}

Let us now turn to pullback, whose description is more complicated. We will describe it in the case where we have global charts $j_{0}\colon P\to Q$ and $j'_{0}\colon P\to Q'$ (induced by $j_0$ and by a chart $Q\to Q'$ for $B\to B'$) for $A\to B$ and $A\to B'$.

Let us define a functor $G\colon \Par(X,j_{0})\to \Par(X,j'_{0})$. Start with a parabolic sheaf $E\colon Q^{\wt}\to \Qcoh(X)$, and take an element $q' \in Q'^{\gp}$. Denote by $Q_{q'}$ the set
$$
Q_{q'}=\{q\in Q^{\gp}\mid q\leq q'\}
$$
where $q\leq q'$ means that there exists $a \in Q'$ such that $q+a=q'$. This is naturally a pre-ordered set, and we have a functor $Q_{q'}\to \Qcoh(X)$ by restricting $E$.

We define
$$
G(E)_{q'}=\varinjlim_{q \in Q_{q'}} E_{q}
$$
which is a quasi-coherent sheaf on $X$, being a colimit of quasi-coherent sheaves.

Note that in particular if $Q_{q'}$ has a maximum $m$ (i.e. if there is an element $m \in Q^{gp}$ such that $q\leq m$ for any $q \in Q_{q'}$), then $G(E)_{q'}=E_{m}$. Further, if $q \in Q^{\gp}$, we clearly have $G(E)_{q}=E_{q}$, where of course we see $q \in Q^{\gp}\subseteq Q'^{\gp}$.

If we have an arrow $q'\to q''$ in $Q'^{\wt}$, i.e. an element $a \in Q'$ such that $q'+a=q''$, then we have a homomorphism $Q_{q'}\to Q_{q''}$ given by inclusion, and this induces a map $G(E)_{q'}\to G(E)_{q''}$. This defines a functor $G(E)\colon Q'^{\wt}\to \Qcoh(X)$. Similar reasonings give a pseudo-periods isomorphism, so that $G(E)$ becomes a parabolic sheaf, and one checks that $G$ gives a functor $\Par(X,j_{0})\to \Par(X,j'_{0})$ as claimed.

It is also immediate to check that $G$ is left adjoint to the $F$ constructed above (in the case where we have global charts), and that the unit of the adjunction $\id \to F\circ G$ is an isomorphism.

\begin{prop}
Assume that we have global charts for $A\to B$ and $B\to B'$ as in the preceding discussion. Then the functor $G$ described in the preceding discussion corresponds to the pullback functor $\pi^{*}\colon \Qcoh(\radice{B}{X})\to \Qcoh(\radice{B'}{X})$.
\end{prop}

\begin{proof}
This follows from uniqueness of adjoint functors and the preceding proposition.
\end{proof}

\begin{example}
Consider the situation of Example \ref{example.1}, and the Kummer extensions $\bb{N}\subseteq \frac{1}{2}\bb{N}\subseteq \frac{1}{4}\bb{N}$. Call $\pi\colon X_{4}\to X_{2}$ the projection, and assume that we have a parabolic sheaf $E$ with respect to $\bb{N}\subseteq \frac{1}{2}\bb{N}$, given by
$$
\xymatrix{
0& \frac{1}{2} & 1\\
E_{0}\ar[r] & E_{\frac{1}{2}}\ar[r] &E_{1}
}
$$
as above. Then the pullback $\pi^{*}E$ on $X_{4}$ can be described as a parabolic sheaf as
$$
\xymatrix{
0& \frac{1}{4} & \frac{1}{2} & \frac{3}{4} & 1\\
E_{0}\ar@{=}[r] &E_{0}\ar[r] & E_{\frac{1}{2}}\ar@{=}[r] & E_{\frac{1}{2}}\ar[r]&E_{1}.
}
$$
In fact in this case the set $Q_{q'}$ of the description above has always a maximum, and the direct limit reduces to evaluating the parabolic sheaf $E$ at the maximum. For example if we take $q'=\frac{1}{4}$, then $Q_{q'}$ has $0$ as maximum, and consequently $(\pi^{*}E)_{\frac{1}{4}}$ will be just $E_{0}$.
\end{example}

\begin{cor}\label{pullback.fully.faithful}
Let $X$ be a log scheme with DF structure $L\colon A \to \Div_{X}$, and $A\to B$, $B\to B'$ two systems of denominators. Then pullback along $\pi\colon \radice{B'}{X}\to \radice{B}{X}$ is fully faithful.
\end{cor}

\begin{proof}
Being a local question in the \'{e}tale topology of $X$, this follows from the previous propositions and from the fact that the unit of the adjunction $G\dashv F$ is an isomorphism.
\end{proof}

\subsection{The infinite root stack}\label{irs}

The infinite root stack is a limit version of the ``finite'' root stacks, and was first introduced and studied in \cite{infinite}. We will recall here its definition and the basic properties that we will use later on.

Let $X$ be a fs log scheme. Note that the set of systems of denominators $\{j\colon A\to B \mid \mbox{ } j \mbox{ is a system of denominators }\}$ is partially ordered by inclusion, i.e. say that $(j\colon A\to B)\leq (j'\colon A\to B')$ if there is a factorization $A\stackrel{j}{\to} B\to B'$ of $j'$ (we can restrict to considering subsheaves $B\subseteq A_{\bb{Q}}$, as remarked before).

Note that if $(j\colon A\to B)\leq (j'\colon A\to B')$, then as we mentioned in the previous section there is a natural morphism $\radice{B'}{X}\to \radice{B}{X}$. Moreover if $(j\colon A\to B)\leq (j'\colon A\to B')$ and $(j'\colon A\to B')\leq (j''\colon A\to B'')$, then the morphism $\radice{B''}{X}\to \radice{B}{X}$ is equal (and not just isomorphic) to the composition of $\radice{B''}{X}\to \radice{B'}{X}$ and $\radice{B'}{X}\to \radice{B}{X}$.

Thus we can consider the inverse limit $\varprojlim_{({j\colon A\to B})}\radice{B}{X}$ as a stack over $(\Sch/X)$.

\begin{definition}
The \emph{infinite root stack} of $X$ is the inverse limit $\infroot{X}=\varprojlim_{({j\colon A\to B})} \radice{B}{X}$.
\end{definition}

\begin{rmk}
One can give an alternative definition of $\infroot{X}$ as $\radice{A_{\bb{Q}}}{X}$ (see \cite[Proposition 3.5]{infinite}), where $A\to A_{\bb{Q}}$ is the maximal Kummer extension of $A$, and $\radice{A_{\bb{Q}}}{X}$ is defined in the same way as the root stacks in \ref{subsec.rs.ps}. Note however that $A\to A_{\bb{Q}}$ is not a system of denominators, since the sheaf $A_{\bb{Q}}$ is not finitely generated.

One can also restrict to smaller subsystems to take the limit. For example if $A\to B$ is a fixed system of denominators, one can consider the subsystem $\{A\to \frac{1}{n}B\}_{n \in \bb{N}}$. Taking the limit $\varprojlim_{n}\radice{\frac{1}{n}B}{X}$, we find again $\infroot{X}$, basically because $\varinjlim_{n}\frac{1}{n}B=A_{\bb{Q}}$. In particular we have $\infroot{X} \simeq \varprojlim_{n}\radice{n}{X}$, where recall that $\radice{n}{X}=\radice{\frac{1}{n}A}{X}$.

More generally this holds for any subsystem $\{A\to B_{i}\}_{i \in I}$ such that $\varinjlim_{i \in I}B_{i}=A_{\bb{Q}}$.
\end{rmk}

If $X$ has a global chart (and thus \'{e}tale locally on any $X$) we have a quotient stack description of $\infroot{X}$ (see \cite[Section 3.1]{infinite}).

\begin{prop}
Let $X$ be a log scheme and $P\to \Div(X)$ a chart for $X$. Then we have an isomorphism
$$
\infroot{X} \simeq X\times_{[\spec(k[P])/\what{P}]}[\spec(k[P_{\bb{Q}}])/\what{P_{\bb{Q}}}].
$$
\end{prop}

As in \ref{subsec.rs.ps}, this turns into a quotient stack description (still in presence of a global chart): if $\eta\colon E\to X$ is the $\widehat{P}$-torsor corresponding to the map $X\to \left[\spec(k[P])/\widehat{P} \right]$ giving the global chart, then we have an isomorphism
$$
\infroot{X} \simeq \left[ (E\times_{\spec(k[P])}\spec(k[P_{\bb{Q}}]))/\widehat{P_{\bb{Q}}} \right]
$$
for the natural action.

We also have an alternative description as a quotient stack in presence of a Kato chart: if $P\to \Div(X)$ comes from a Kato chart $P\to \mc{O}_{X}(X)$, then we also have an isomorphism
$$
\infroot{X} \simeq \left[(X\times_{\spec(k[P])}\spec(k[P_{\bb{Q}}])) /\mu_{\infty}(P)\right]
$$
for the natural action, where $\mu_{\infty}(P)$ is the Cartier dual of the cokernel of the homomorphism $P^{\gp}\to P^{\gp}_{\bb{Q}}$.

This local description has the following consequence.

\begin{prop}[{\cite[Proposition 3.18]{infinite}}]
Let $X$ be a fs log scheme. Then the infinite root stack $\infroot{X}$ has an fpqc atlas, i.e. a representable fpqc morphism $U\to \infroot{X}$ from a scheme. Moreover the diagonal $\infroot{X}\to \infroot{X}\times_{X}\infroot{X}$ is affine.
\end{prop}

So, although $\infroot{X}$ is not an algebraic stack in the sense of Artin, it still has some weak algebraicity property. In particular, thanks to fpqc descent of quasi-coherent sheaves, we can define quasi-coherent (and coherent, and finitely presented) sheaves on $\infroot{X}$. See \cite[Section 4]{infinite} for details.

In particular we will denote by $\FP(\infroot{X})$ the category of finitely presented sheaves on $\infroot{X}$.

\begin{rmk}
One feature that complicates the discussion about sheaves is the fact that $\infroot{X}$ is not coherent in general (see \cite[Example 4.17]{infinite}), and consequently finitely presented and coherent sheaves are two distinct notions. Since we are interested in moduli of sheaves, we will exclusively be dealing with finitely presented sheaves.

In any case, since when considering moduli theory for varying weights we will assume that our log schemes are \emph{simplicial} (see \ref{sec.4.1}), then $\infroot{X}$ will be in fact coherent (\cite[Proposition 4.19]{infinite}).
\end{rmk}

Here are some properties of the infinite root stack that we are going to use in what follows.

First of all, the infinite root stack is compatible with strict base change.

\begin{prop}[{\cite[Proposition 3.4]{infinite}}]
Let $(X,A,L)$ be a log scheme, $f\colon Y \to X$ a morphism of schemes, and equip $Y$ with the pullback DF structure $f^{*}(A,L)$. Then we have an isomorphism $\infroot{Y} \simeq \infroot{X}\times_{X}Y$, i.e. the diagram
$$
\xymatrix{
\infroot{Y}\ar[r]\ar[d] & \infroot{X}\ar[d]\\
Y\ar[r] & X
}
$$
is cartesian.
\end{prop}

Now note that if $j\colon A\to B$ is a system of denominators, then the natural projection $\pi\colon \infroot{X}\to \radice{B}{X}$ induces pullback and pushforward, $\pi^{*}\colon \Qcoh(\radice{B}{X})\to \Qcoh(\infroot{X})$ and $\pi_{*}\colon \Qcoh(\infroot{X})\to \Qcoh(\radice{B}{X})$ (\cite[Section 4.1]{infinite}). One can show moreover that the pushforward $\pi_{*}$ takes finitely presented sheaves to finitely presented sheaves, see \cite[Proposition 4.16]{infinite} (this is clear for the pullback).

\begin{prop}[{\cite[Proposition 6.1]{infinite} \label{fp.sheaves}}]
Let $X$ be a quasi-compact fs log scheme. Then the natural functor
$$
\varinjlim_{n}\FP(\radice{n}{X})\to\FP(\infroot{X}) 
$$
is an equivalence.
\end{prop}

As in the case of the finite root stacks we have a projection formula for the infinite root stack.

\begin{prop}[{\cite[Proposition 4.16]{infinite}}] \label{proj.formula.infinite}
Let $(X,A,L)$ be a fs log scheme, $j\colon A\to B$ a system of denominators with $B$ saturated, and denote by $\pi\colon \infroot{X}\to \radice{B}{X}$ be the canonical projection. Then:
\begin{itemize}
\item $\pi_{*}\colon \Qcoh(\infroot{X})\to \Qcoh(\radice{B}{X})$ is exact,
\item the canonical morphism $\mc{O}_{\radice{B}{X}}\to \pi_{*}\mc{O}_{\infroot{X}}$ is an isomorphism,
\item if $F \in \Qcoh(\radice{B}{X})$ and $G\in \Qcoh(\infroot{X})$ the natural map $ F \otimes \pi_{*}G\to \pi_{*}(\pi^{*}F\otimes G)$ is an isomorphism,
\item consequently for $F \in \Qcoh(\radice{B}{X})$ the canonical map $F\to \pi_{*}\pi^{*}F$ on $\radice{B}{X}$ is also an isomorphism.
\end{itemize}
\end{prop}

To conclude, let us remark that one can give a definition of a parabolic sheaf with arbitrary rational weights, by replacing $Q$ with $P_{\bb{Q}}$ and $B$ with $A_{\bb{Q}}$ in Definitions \ref{def.parabolic.sheaf.chart} and \ref{def.parabolic.sheaf} (\cite[Section 7]{infinite}). They form an abelian category that we denote by $\Par(X)$.

\begin{example}
If $X$ is a regular scheme and $D\subseteq X$ is a smooth Cartier divisor, then a parabolic sheaf with rational weights on the log scheme $X$ with the divisorial log structure is given (by using the usual chart $\bb{N}\to \Div(X)$ on $X$) by a sheaf $E_{q}$ for every rational number $q$, together with maps $E_{q}\to E_{q'}$ for every pair with $q\leq q'$. Moreover we have an isomorphism $E_{q+1} \simeq E_{q}\otimes \mc{O}_{X}(D)$ for every $q \in \bb{Q}$, such that $E_{q}\to E_{q+1}$ corresponds to multiplication by the canonical section $s$ of $\mc{O}_{X}(D)$.
\end{example}

With the same proof of Theorem \ref{BV} one proves the following generalization.

\begin{thm}[{\cite[Theorem 7.3]{infinite}}]
Let $(X,A,L)$ be a fs log scheme. Then there is an equivalence of abelian categories $\Phi\colon \Qcoh(\infroot{X})\to \Par(X)$.
\end{thm}

Moreover we have a parabolic description of pullback and pushforward along a projection $\pi\colon \infroot{X}\to \radice{B}{X}$, analogous to the one described at the end of \ref{subsec.rs.ps} for morphisms between finite root stacks.

\begin{rmk}
The definitions of parabolic sheaves with arbitrary rational weights in the existing literature (\cite{borne}, \cite{maruyama-yokogawa}) all assume some ``finiteness of the jumps'' of the pieces of the parabolic sheaf. These conditions single out precisely parabolic sheaves coming from some finite root stack $\radice{n}{X}$ by pullback to $\infroot{X}$ (see \cite[Section 7.1]{infinite}).
\end{rmk}

\section{Moduli of parabolic sheaves with fixed denominators}\label{section.3}

This section is about the moduli theory of parabolic sheaves on a projective fs log scheme $X$ with a fixed system of denominators $j\colon A\to B$, together with a global chart $P\to Q$. This is basically an application of the theory of \cite{nironi} to coherent sheaves on the corresponding root stack.

From now on $X$ will be a fine and saturated log scheme with a DF structure $L\colon A\to \Div_{X}$, and $j\colon A \to B$ will be a fixed system of denominators. Moreover (from \ref{sec.3.3} on) $X$ will be projective over $k$, and $\m{X}$ will denote the root stack $\radice{B} X$. To apply Nironi's machinery we will also have to assume that $\radice{B} X$ is Deligne--Mumford. This holds if $\cha(k)$ does not divide the order of the group $B^{\gp}_{x}/A^{\gp}_{x}$ for any geometric point of $X$ (see \ref{root.stack.algebraic}).

\subsection{Families of parabolic sheaves}\label{sec.3.1}

To define a moduli stack for parabolic sheaves, the first thing to worry about is the notion of pullback. From the parabolic point of view it is not so clear how it should be defined, but the BV equivalence gives us a natural candidate.

Let $f\colon {Y}\to {X}$ a morphism of schemes, and equip ${Y}$ with the pullback log structure. We want to define a pullback functor $f^{*}$ from parabolic sheaves on $X$ with respect to $j$ to parabolic sheaves on $Y$ with respect to the pullback system of denominators $f^{*}j\colon f^{*}A \to f^{*}B$. 

Assume more generally that $Y$ is a (coherent) log scheme with log structure $N\colon C\to \Div_{Y}$, $h\colon C\to D$ is a system of denominators, $f\colon Y\to X$ is a morphism of log schemes and the morphism $f^{*}A\to C$ fits in a commutative diagram
$$
\xymatrix{
f^{*}A\ar[d]\ar[r]&C\ar[d]\\
f^{*}B\ar[r]&D.
}
$$
Then we have a natural morphism of root stacks $\Pi\colon \radice{D} Y\to \radice{B} X$, fitting in a commutative diagram
$$
\xymatrix{
\radice{D} Y\ar[r]\ar[d]& \radice{B} X\ar[d]\\
Y\ar[r]^{f}& X,
}
$$
and sending an object $M \colon t^{*}D\to \Div_{T}$ of $\radice{D} Y(T)$, where $t\colon {T}\to {Y}$, to the composite $\Pi(M)\colon (f\circ t)^{*}B\to t^{*}D\to  \Div_{T}$, which is an object of $\radice{B} X(T)$.

In case $N=f^{*}L$ and $h=f^{*}j$, the diagram is also cartesian (\ref{strict.cartesian}), i.e. $\radice{f^{*}B} Y \simeq \radice{B} X\times_{X}Y$.

\begin{definition}
Given a parabolic sheaf $E \in \Par(X,j)$, the \emph{pullback} $f^{*}E$ of $E$ along $f$ is the parabolic sheaf $\Phi(\Pi^{*}(\Psi(E))) \in \Par(Y,h)$ corresponding via the BV equivalence to the pullback of the quasi-coherent sheaf $\Psi(E) \in \Qcoh(\radice{B} X)$ along $\Pi$.
\end{definition}

The pullback $f^{*}E$ is unique up to isomorphism and has the usual functorial properties. Moreover by definition the diagram
$$
\xymatrix{
\Par(X,j)\ar[r]^{f^{*}}\ar[d]_{\Psi}&\Par(Y,h)\ar[d]^{\Psi}\\
\Qcoh(\radice{B} X)\ar[r]^{\Pi^{*}}& \Qcoh(\radice{D} Y)
}
$$
is $2$-commutative.

In the case where $X$ and the system of denominators have a global chart, the pullback of a parabolic sheaf along a strict morphism $f\colon {Y}\to X$ has a simple ``purely parabolic'' description (which seems to be lacking for example for a non-strict morphism): assume that we have charts $P\to \Div(X)$ for $L$ and $P \to Q$ for $j$. These charts also give charts on $Y$ for $f^{*}A$ and $f^{*}j$, and given a parabolic sheaf $E\colon Q^{\wt}\to \Qcoh(X)$, we can define $f^{*}E$ as the composition $Q^{\wt}\to \Qcoh(X)\to \Qcoh(Y)$, where the last functor is pullback of quasi-coherent sheaves.

\begin{prop}
The functor $f^{*}E$ is a parabolic sheaf on $Y$, and the corresponding quasi-coherent sheaf $\Psi(f^{*}E)$ on the root stack $\radice{f^{*}B} Y$ is naturally isomorphic to the pullback along $\Pi\colon \radice{f^{*}B} Y\to \radice{B} X$ of the quasi-coherent sheaf $\Psi(E)$ on $\radice{B} X$ corresponding to $E$. 
\end{prop}

\begin{proof}
It is clear that $f^{*}E$ is a parabolic sheaf, by applying $f^{*}$ to the pseudo-periods isomorphism $\rho^{E}$ and to all the relevant diagrams.

Let us now show that the parabolic sheaf $\Phi(\Pi^{*}\Psi(E)) \in \Par(Y,f^{*}j)$ is isomorphic to $f^{*}E$ as defined above.

For $q \in Q^{\gp}$, let us calculate
$$
(f^{*}E)_{q}=f^{*}E_{q}=f^{*}\pi_{*}(\Psi(E)\otimes \Lambda_{q})
$$
and
$$
\Phi(\Pi^{*}\Psi(E))_{q}=(\pi_{Y})_{*}(\Pi^{*}\Psi(E)\otimes (\Lambda_{Y})_{q}).
$$
Note now that $(\Lambda_{Y})_{q}=\Pi^{*}\Lambda_{q}$, so
$$
(\pi_{Y})_{*}(\Pi^{*}\Psi(E)\otimes (\Lambda_{Y})_{q}) \simeq (\pi_{Y})_{*}\Pi^{*}(\Psi(E)\otimes \Lambda_{q}).
$$
Now we apply Proposition $1.5$ of \cite{nironi} to the cartesian diagram
$$
\xymatrix{
\radice{f^{*}B} Y\ar[r]^{\Pi}\ar[d]_{\pi_{Y}} & \radice{B} X\ar[d]^{\pi}\\
Y\ar[r]^{f} & X
}
$$
to get a functorial base change isomorphism $f^{*}\pi_{*} \to (\pi_{Y})_{*}\Pi^{*}$ of functors $\Qcoh(\radice{B} X)\to \Qcoh(Y)$.

This gives an isomorphism $(f^{*}E)_{q}\to \Phi(\Pi^{*}\Psi(E))_{q}$ for any $q \in Q$. By functoriality, by putting all these isomorphisms together we get a natural isomorphism of functors $Q^{\wt}\to \Qcoh(X)$, which moreover respects the pseudo-periods isomorphisms, and so is an isomorphism of parabolic sheaves.
\end{proof}

This also gives a local description of the pullback as defined in general, using local charts for the Kummer morphism $j$.
\begin{rmk}
When there is no global chart, it is still possible to give a purely parabolic description of the pullback along a strict morphism, but we will not need it in the rest of the paper. The interested reader can check \cite[Proposition 3.1.4]{tesi} for a sketch of the construction.
\end{rmk}
Using this notion of pullback, we can now define a category of families of parabolic sheaves on a fixed log scheme $X$, and with respect to a Kummer morphism $j\colon A\to B$.

We define a fibered category $\upar_{X}\to (\Sch)$ by setting, for a scheme $T$
$$
\upar_{X}({T})=\Par(T\times_{k}X, \pi_{T}^{*}j)
$$
(with $\pi_{T}\colon {T\times_{k}X}\to {X}$ the second projection), and by declaring that pullback along a morphism $f\colon {S}\to {T}$ over $k$ is given by the pullback of parabolic sheaves along the induced morphism $S\times_{k}X\to T\times_{k}X$ of log schemes. Of course here $T\times_{k}X$ plays the role of a trivial family with fiber $X$, and a parabolic sheaf over $T\times_{k}X$ is seen as a ``naive'' (i.e. without any flatness hypothesis) family of parabolic sheaves over $X$.

On the other hand we have a second fibered category $\uqcoh_{\radice{B} X}\to (\Sch)$ of quasi-coherent sheaves over the root stack $\radice{B} X$, where for a scheme ${T}$ we have
$$
\uqcoh_{\radice{B} X}({T})=\Qcoh(T\times_{k}\radice{B} X)=\Qcoh(\radice{\pi_{T}^{*}B} {T\times_{k}X})
$$
with pullback along ${S}\to {T}$ defined by the induced morphism $S\times_{k}\radice{B} X \to T\times_{k}\radice{B} X$.

\begin{notation}
To ease the notation, for the rest of this section $\m{X}$ will denote the root stack $\radice{B} X$, and for a scheme $T$, a subscript $(-)_{T}$ will denote a base change to $T$, or to the fibered product $X_{T}=T\times_{k}X$ along the projection $\pi_{T}\colon X_{T}\to X$ (this ambiguity should cause no real confusion). For example $L_{T}\colon A_{T}\to \Div_{X_{T}}$ will denote the pullback log structure.
\end{notation}

Also, in what follows we will repeatedly consider properties of parabolic sheaves on $X_{T}$ relative to the base $T$. It will be useful to keep in mind the following diagram, where all the squares are cartesian.

$$
\xymatrix{
\m{X}_{T}\ar[r]\ar[d]  &\m{X}\ar[d]\\
X_{T}\ar[r]\ar[d]&X\ar[d]\\
 T\ar[r] &\spec(k)
}
$$
Note that $\m{X}_{T} \simeq \radice{B_{T}}{X_{T}}$, from Proposition \ref{strict.cartesian}. Moreover the projection $\m{X}_{T}\to X_{T}$ is a coarse moduli space for any $T$, since the log structure of $T$ is fine and saturated (\ref{coarse.space.root}).

The following is an immediate consequence of \ref{BV} and the definition of pullback of parabolic sheaves.

\begin{prop}
There are equivalences $\Phi\colon \uqcoh_{\m{X}}\to  \upar_{X}$ and $\Psi\colon \upar_{X}\to \uqcoh_{\m{X}}$, restricting to the BV equivalences on the fiber categories.
\end{prop}

This allows us to systematically transport various (absolute and relative) notions from ordinary quasi-coherent sheaves to parabolic sheaves. The notions that we care about also have a parabolic interpretation.

\begin{definition}[Meta-definition]\label{meta}
A parabolic sheaf $E \in \Par(X_T,j_T)$ has some property, absolute or over the base $T$ (for example is \emph{coherent, finitely generated, finitely presented, locally free, flat over} $T$), if the corresponding $\Psi(E) \in \Qcoh(\m{X}_{T})$ has said property.
\end{definition}

These definitions also make sense for an arbitrary log scheme, not of the form $X_{T}$. We restrict to this case because we are interested in families of parabolic sheaves over a fixed log scheme $X$.

Let us discuss with more details some properties of parabolic sheaves, that will single out the correct notion of ``family of coherent parabolic (pure) sheaves''.

Let us start with coherence: the notion we want to use is actually ``finitely presented'', since coherence does not behave well with respect to pullback in a (possibly) non-noetherian setting. The definition is contained in \ref{meta}, but let us spell it out again.

\begin{definition}
A parabolic sheaf $E \in \Par(X_T,j_T)$ is \emph{finitely presented} if the corresponding $\Psi(E) \in \Qcoh(\m{X}_{T})$ is finitely presented.
\end{definition}

Here is a parabolic interpretation of coherence.

\begin{prop}
A parabolic sheaf $E \in \Par(X_T,j_T)$ is finitely presented if and only if for every \'{e}tale morphism $U\to X_{T}$ and every section $b \in B_{T}^{\wt}(U)$, the quasi-coherent sheaf $E_{b} \in \Qcoh(U)$ is finitely presented.
\end{prop}

\begin{proof}
Assume first that $E$ is finitely presented. Then for any section $b  \in B_{T}^{\wt}(U)$ the sheaf $E\otimes \Lambda_{b}$ is finitely presented, and thus $E_{b}=(\pi_{T})_{*}(E\otimes \Lambda_{b})$ is also finitely presented, by \cite[Proposition 4.16]{infinite}: if $\Pi\colon \infroot{X_{T}}\to \m{X}_{T}$ denotes the projection, then (by \ref{proj.formula.infinite}) we have an isomorphism
$$
(\pi_{T})_{*}(E\otimes \Lambda_{b}) \simeq (\pi_{T}\circ \Pi)_{*}(\Pi^{*}(E\otimes \Lambda_{b}))
$$
and $\Pi^{*}(E\otimes \Lambda_{b})$ is finitely presented on $\infroot{X_{T}}$.

In the other direction, this is a local problem so we can assume that there is a global chart. In this case, recall (see the sketch of proof of \ref{BV}) that the quasi-coherent sheaf on $\m{X}$ corresponding to $E$ is obtained by taking $\bigoplus_{v \in Q^{\gp}} E_{v}$ as a sheaf on $X$, with an action of the sheaf of algebras $A=\bigoplus_{u \in P^{\gp}} L_{u}$, and then by descending it on the quotient stack $\m{X}$. Now it suffices to notice that a finite number of the $E_{v}$ generate the direct sum as a sheaf of $A$-modules (thanks to the pseudo-periods isomorphism), and since the $E_{v}$'s are finitely presented, we are done.
\end{proof}

Now let us turn to flatness.

\begin{definition}
A parabolic sheaf $E \in\Par(X_T,j_T)$ is \emph{flat} over $T$ if the corresponding $\Psi(E) \in \Qcoh(\m{X}_{T})$ is flat over $T$.
\end{definition}

The following proposition gives a parabolic interpretation of flatness.

\begin{prop}
A parabolic sheaf $E \in \Par(X_T,j_T)$ is flat over $T$ if and only if for every \'{e}tale morphism $U\to X_{T}$ and every section $b \in B_{T}^{\wt}(U)$ the quasi-coherent sheaf $E_{b} \in \Qcoh(U)$ is flat over $T$.
\end{prop}

\begin{proof}
Assume first that the sheaf $\Psi(E) \in \Qcoh(\m{X}_{T})$ is flat over $T$. Given $f\colon U\to X_{T}$ \'{e}tale, call $\m{U}$ the base change of $\m{X}_{T}$ to $U$, $\pi_{U}\colon \m{U}\to U$ the projection to the coarse moduli space, and $f_{\m{U}}\colon \m{U}\to \m{X}_{T}$ the base change of $f$.

Now for $b \in B_{T}^{\wt}(U)$ we have
$$
E_{b}= (\pi_{U})_{*} ( f_{\m{U}}^{*}\Psi(E) \otimes (\Lambda_{T})_{b} ) \in \Qcoh(U).
$$
Since $(\Lambda_{T})_{b}$ is invertible and $\Psi(E)$ is flat over $T$ by assumption, the sheaf $f_{\m{U}}^{*}\Psi(E) \otimes (\Lambda_{T})_{b}$ is flat over $T$. Furthermore since $\m{U}$ is tame, the pushforward along the projection to the coarse moduli space $\pi_{U}:\m{U}\to U$ preserves flatness over the base $T$ (Corollary $1.3$ of \cite{nironi}). In conclusion $E_{b}$ is flat over $T$.

Now let us prove the converse. Since the question is local we can assume that $X$ has a global chart. Recall once again (\ref{BV}) that the sheaf $\Psi(E)$ is defined by forming $\bigoplus_{q \in Q^{\gp}} E_{q}$, a quasi-coherent sheaf on $X$. Then this is regarded as a quasi-coherent sheaf on the relative spectrum of a sheaf of algebras on $X$, and by descent this gives a quasi-coherent sheaf on the root stack $\m{X}$ (which is a quotient stack of this relative spectrum). Now since by assumption all the $E_{q}$ are flat over $T$, their direct sum is flat over $T$, and the quasi-coherent sheaf induced on the relative spectrum mentioned above will be flat over $T$ as well. Finally by descent $\Psi(E)$ itself will be flat over $T$.
\end{proof}

In moduli theory of coherent sheaves (and in all of moduli theory, in fact), flatness is a crucial condition to impose on a family.

\begin{definition}
A \emph{family of parabolic sheaves} on $X$ with weights in $B/A$ over a base scheme $T$ is a finitely presented parabolic sheaf $E\in \Par(X_T,j_T)$ that is flat over $T$.

From now on the wording ``family of parabolic sheaves'' will always include this flatness condition.
\end{definition}

The last important concept in moduli of coherent sheaves is pureness. This is not a condition that we have to impose on the family, but rather on the geometric fibers of a family, so it is enough to consider it for parabolic sheaves on $X$ itself (the same notion will apply on any base change $X_{K}$ with $K$ an algebraically closed extension $k\subseteq K$).

For the definition of a pure sheaf on an algebraic stack we refer to \cite{nironi} and \cite{lieblich}. It is the natural generalization of the concept for schemes; one possible definition is that a coherent sheaf on a (noetherian) Artin stack is pure if and only if its pullback to a smooth presentation is. More intrinsically one can define the support of a coherent sheaf $F$ on an algebraic stack $\mc{X}$ as the closed substack defined by the kernel of the morphism $\mc{O}_{\mc{X}}\to \End(F)$. The \emph{dimension} of $F$ will be the dimension of the support, and a sheaf is pure of dimension $d$ if and only if all of its subsheaves have dimension $d$.

\begin{rmk}
We declare the zero sheaf to be pure of arbitrary dimension. This allows us to simplify some statements and does no harm, since the property of being zero for a flat sheaf on a projective family is open and closed.
\end{rmk}

We will say that a coherent sheaf $F$ is \emph{torsion-free} on a noetherian algebraic stack $\mc{X}$ if it is pure of maximal dimension (the dimension of $\mc{X}$). Note that this does not imply that $F$ is supported everywhere (topologically) unless $\mc{X}$ is irreducible.

Assume that the log scheme $X$ is noetherian.

\begin{definition}
A parabolic sheaf $E \in \Par(X,j)$ is \emph{pure of dimension} $d$ if it is finitely presented and the corresponding $\Psi(E) \in \Qcoh(\m{X})$ is pure of dimension $d$.
\end{definition}

As for the preceding properties, there is a parabolic interpretation of pureness.

\begin{prop}\label{prop.pure}
A parabolic sheaf $E \in \Par(X,j)$ is pure of dimension $d$ if and only if for every \'{e}tale morphism $U\to X$ and every section $b \in B^{\wt}(U)$ the coherent sheaf $E_{b} \in \Qcoh(U)$ is pure of dimension $d$.
\end{prop}

\begin{rmk}
Nironi gives a seemingly different definition of a torsion-free parabolic sheaf in \cite[Definition 7.4]{nironi}; one can easily see that it agrees with the one given here, basically thanks to \ref{gen.trivial.inj} below.
\end{rmk}

We stress again that the zero sheaf is considered to be pure of arbitrary dimension (in fact, it can happen that $E_{b}$ is zero for a pure parabolic sheaf $E$ and some section $b$). We will use the following lemmas, whose proofs are easy and left to the reader.

\begin{lemma}\label{lem.pure}
Let $E$ be a coherent sheaf on a noetherian scheme $X$, and set $d=\dim(F)$. Then $E$ is pure of dimension $d$ if and only if for every open subset $U\subseteq X$ such that $\dim(X\setminus U)<d$, the adjunction map $\sigma\colon E\to i_{*}i^{*}E$ is injective, where $i\colon U\to X$ is the inclusion.\qed
\end{lemma}

\begin{lemma}\label{cor.pure}
If $\mc{X}$ is a noetherian DM stack and $E \in \Coh(\mc{X})$ is pure of dimension $d$ then for every open substack $\mc{U}\subseteq \mc{X}$ such that $\dim(\mc{X}\setminus \mc{U})<d$, the adjunction map $\sigma\colon E\to i_{*}i^{*}E$ is injective, where $i\colon \mc{U}\to \mc{X}$ is the inclusion.\qed
\end{lemma}

\begin{proof}[Proof of Proposition \ref{prop.pure}]
Assume first that $G=\Psi(E)$ is pure of dimension $d$, and fix $f\colon U\to X$ \'{e}tale, and $b \in B^{\wt}(U)$. Moreover call $\m{U}$ the base change of $\m{X}$ to $U$, $\pi_{U}\colon \m{U}\to U$ the projection, and $f_{\m{U}}$ the base change of $f$.

We will show that the coherent sheaf $E_{b}=(\pi_{U})_{*}(f_{\m{U}}^{*}(G)\otimes \Lambda_{b}) \in \Coh(U)$ is pure of dimension $d$. Set $G'=f_{\m{U}}^{*}(G)$, and notice that this is still pure of dimension $d$ on $\m{U}$ (it may be zero).

Note that since $\pi_{U}$ is proper and quasi-finite, if $E_{b}$ is not zero (in which case there is nothing to prove) then it has dimension $d$. Take an open subset $V\subseteq U$ with $\dim(U\setminus V)<d$, and call $i\colon V\to U$ the inclusion. We will show that the adjunction map $\sigma\colon E_{b}\to i_{*}i^{*}E_{b}$ is injective, and by Lemma \ref{lem.pure} this will prove that $E_{b}$ is pure of dimension $d$.

Call $\m{V}$ the fiber product $V\times_{U}\m{U}$, denote by $j\colon \m{V}\to \m{U}$ the base change of the inclusion $V\subseteq U$ and $\pi_{V}\colon \m{V}\to V$ the base change of the projection. The map $j$ is an open immersion, so $\m{V}$ is an open substack of $\m{U}$, with complement of dimension less than $d$. Since $G'$ is pure, by Lemma \ref{cor.pure} the map $\sigma'\colon G'\to j_{*}j^{*}G'$ is injective.

Now the pushforward $(\pi_{U})_{*}$ and tensor product with $\Lambda_{b}$ are both exact functors, so the induced map
$$
E_{b}=(\pi_{U})_{*}(G'\otimes \Lambda_{b}) \to (\pi_{U})_{*}(j_{*}j^{*}(G')\otimes \Lambda_{b})
$$
is still injective. Now note that by the projection formula we have
$$
j_{*}j^{*}(G')\otimes \Lambda_{b} \simeq j_{*}(j^{*}(G')\otimes j^{*}(\Lambda_{b})) \simeq j_{*}j^{*}(G'\otimes\Lambda_{b}).
$$
From the cartesian diagram
$$
\xymatrix{
\m{V}\ar[r]^{j}\ar[d]_{\pi_{V}} & \m{U} \ar[d]^{\pi_{U}}\\
V \ar[r]^{i} & U
}
$$
we first see that $(\pi_{U})_{*}j_{*}=i_{*}(\pi_{V})_{*}$, and since $i$ is flat, by base change (Proposition $1.5$ of \cite{nironi}) we also have a canonical isomorphism $(\pi_{V})_{*}j^{*} \simeq i^{*}(\pi_{U})_{*}$.

By putting everything together we have that the composite
$$
E_{b}=(\pi_{U})_{*}(G'\otimes \Lambda_{b}) \to (\pi_{U})_{*}j_{*}j^{*}(G'\otimes \Lambda_{b}) \simeq i_{*}i^{*}(\pi_{U})_{*}(G'\otimes \Lambda_{b})  \simeq i_{*}i^{*}E_{b},
$$
which coincides with the adjunction map of $E_{b}$, is injective, and this is what we had to show.

Note that some of these $E_{b}$ can be zero, which is consistent with our convention about the zero sheaf, but if $E$ is not zero as a parabolic sheaf, then necessarily we will have $E_{b}\neq 0$ for some $U\to X$ \'{e}tale and $b \in B^{\wt}(U)$.

Now for the converse, assume that all the $E_{b}$'s are pure of dimension $d$, and that $E$ is not zero (otherwise there is nothing to prove). If by contradiction $\Psi(E)$ is not pure, then there is a non zero pure subsheaf $\Psi(G)\subseteq \Psi(E)$, of dimension strictly less than $d$, say $d' \geq 0$. Now pick $U\to X$ \'{e}tale and $b\in B^{\wt}(U)$ such that $G_{b}\neq 0$.

By the first part of the proof $0\neq G_{b}\subseteq E_{b}$ is pure of dimension $d'$, so $E_{b}$ is non zero and thus of dimension $d>d'$. In particular $E_{b}$ is not pure, and this contradicts the assumption.
\end{proof}

In case the log structure of $X$ is generically trivial, the maps between the pieces of a torsion-free parabolic sheaf are injective.

\begin{prop}\label{gen.trivial.inj}
Let $X$ be a noetherian log scheme with generically trivial log structure and with a chart $P\to \Div(X)$, and let $j\colon P\to Q$ be a Kummer extension of fine saturated monoids. Take a torsion-free parabolic sheaf $E \in \Par(X,j)$. Then for any pair $q,q' \in Q^{\wt}$ such that $q\leq q'$, the morphism $E_{q}\to E_{q'}$ is injective.
\end{prop}

\begin{proof}
If $E_{q}$ is zero there is nothing to prove.

Otherwise, by assumption we have a dense open subscheme $U\subseteq X$ on which the log structure is trivial. Consequently the restriction of the projection $\radice{Q}{X}\to X$ to $U$ is an isomorphism $\radice{Q}{U} \simeq U$, and the morphism $E_{q}\to E_{q'}$ is an isomorphism on $U$ (this follows for example from the construction of the BV equivalence \ref{BV}). If $K$ is the kernel of this morphism, it follows that $K$ has dimension strictly less then the dimension of $X$, but then since $E_{q}$ is pure of maximal dimension by \ref{prop.pure}, we must have $K=0$, i.e. the map is injective.
\end{proof}

Finally, we give the definition of a family of pure parabolic sheaves.

\begin{definition}
A \emph{family of pure $d$-dimensional parabolic sheaves} on $X$ with weights in $B/A$ over a base scheme $T$ is a finitely presented sheaf $E \in \Par(X_T,j_T)$ that is flat over $T$, and such that for any geometric point $t \to T$, the pullback $E_{t}$ on $X_{t}$ is pure of dimension $d$.
\end{definition}

\subsection{(semi-)stability of parabolic sheaves and moduli spaces}\label{sec.3.3}

Let us now specify the situation and add some hypotheses in order to talk about moduli of sheaves.

\begin{assumptions}
From now on $X$ will be projective over $k$, with a fixed very ample line bundle $\mc{O}_{X}(1)$. This polarization will be fixed throughout all that follows, so we will omit it from the notations.

\end{assumptions}

The strategy now is to reduce the moduli problem for parabolic sheaves on $X$ to moduli of coherent sheaves on the root stack $\m{X}$. This was basically done in the previous section: families of pure $d$-dimensional parabolic sheaves correspond to families of pure $d$-dimensional coherent sheaves on the corresponding root stack.

But as in the case of sheaves on schemes, if we want to construct moduli \emph{spaces} we need to come up with some good notion of stability. For this we apply  the theory for moduli of coherent sheaves on algebraic stacks of \cite{nironi}. His hypothesis are that the stack is a tame DM stack with projective coarse moduli space, together with the choice of a generating sheaf (see below). 

Instead of giving a separate summary of this theory, we will recall it along the way, applying it directly to our case.

The basic idea is to define a Hilbert polynomial (and then a notion of stability) for coherent sheaves on $\m{X}$, and the first thing that comes to mind is to push forward to the moduli space, and take the Hilbert polynomial there. This doesn't work, because pushing forward loses too much information. The first step to fix this is to find an appropriate \emph{generating sheaf} on the root stack $\m{X}$.

\begin{definition}
A locally free sheaf $\mc{E}$ of finite rank on a tame Artin stack $\m{X}$ is a \emph{generating sheaf} if for any geometric point $x \to \m{X}$ the fiber $\mc{E}_{x}$ contains every irreducible representation of $\Stab(x)$.
\end{definition}

Tensoring by $\mc{E}^{\vee}$ before taking the pushforward fixes the problem. We will recall later of such a sheaf allows to define a good notion of (semi-)stability. For generalities about generating sheaves see \cite{nironi} and the references therein.

Although root stacks of a fs log scheme are probably always global quotient stacks, so that they will have generating sheaves, in general there does not seem to be a canonical choice of such a sheaf. We are able to single out a distinguished generating sheaf in presence of a global chart for the logarithmic structure and the system of denominators. This global chart plays a role analogous to that of the polarization $\mc{O}_{X}(1)$ in the classical moduli theory of coherent sheaves on a projective scheme. For this reason we adopt the following terminology.

\begin{definition}
A \emph{polarized log scheme} over $k$ is a fs log scheme $X$ over $k$, together with a polarization of the underlying scheme and a global chart for the log structure.
\end{definition}

From now on we will assume that our log scheme $X$ is polarized in this sense. The chart for the log structure will be usually denoted by $P\to \Div(X)$.

We will focus on the construction of the generating sheaf shortly. Once we have such a sheaf, we can define what Nironi calls the \emph{modified Hilbert polynomial}. We will drop the adjective ``modified'' for brevity.

\begin{definition}
The \emph{Hilbert polynomial} (with respect to $\mc{E}$) of a coherent sheaf $F \in \Coh(\m{X})$ is the Hilbert polynomial
$$
P_{\mc{E}}(F)=P(\pi_{*}(F\otimes\mc{E}^{\vee})) \in \bb{Q}[m]
$$
of the coherent sheaf $\pi_{*}(F\otimes \mc{E}^{\vee})$ on $X$, with respect to $\mc{O}_{X}(1)$.
\end{definition}

\begin{rmk}
Of course the dual in $\mc{E}^{\vee}$ is not essential, and could be removed by dualizing the choice of the generating sheaf below. We decided to leave it in order to keep the notations from \cite{nironi}.
\end{rmk}

Note that, since $\pi_{*}(-\otimes \mc{E}^{\vee})$ preserves the dimension (see \cite[Proposition 3.6]{nironi}), $P_{\mc{E}}(F)$ will be a polynomial of degree $d$ where $d=\dim(F)$, and as usual we can write it as
$$
P_{\mc{E}}(F)(m)=\sum_{i=0}^{d}\frac{\alpha^{i}(F)}{i!}m^{i}
$$
where $\alpha^{i}(F)$ are rational numbers that also depend on $\mc{E}$.

The number $\alpha^{d}(F)$, which is always positive, will be called the \emph{multiplicity} of the sheaf $F$. If $X$ is integral and $F$ has maximal dimension, it is closely related to the rank of the sheaf $\pi_{*}(F\otimes \mc{E}^{\vee})$ on $X$.

\begin{definition}\label{gen.slope}
The \emph{reduced Hilbert polynomial} or \emph{(generalized) slope} of a coherent sheaf $F \in \Coh(\m{X})$ is the polynomial
$$
p_{\mc{E}}(F)=\frac{P_{\mc{E}}(F)}{\alpha^{d}(F)}=\frac{1}{d!}m^{d}+\hdots+ \frac{\alpha^{0}(F)}{\alpha^{d}(F)} \in \bb{Q}[x].
$$
\end{definition}

\begin{rmk}
We are aware that the word ``slope'' is usually reserved for the quotient
$$
\mu(F)=\frac{\alpha^{d-1}(F)}{\alpha^{d}(F)},
$$
that in the case of curves is closely related to the ratio $\deg(F)/\rk(F)$, but nonetheless in this document we will use it to mean the reduced Hilbert polynomial, since the slope $\mu(F)$ will not play a significant role here.
\end{rmk}

This slope will give a notion of (semi-)stable parabolic sheaves, and we will restrict to that class in order to get well-behaved moduli stacks and moduli spaces. Before describing how this happens (which we will do in \ref{results}), let us focus on the choice of the generating sheaf.

\subsubsection{Choice of the generating sheaf: the case of a variety with a divisor}\label{maru.yoko}

To get some clues for the choice of the generating sheaf, let us look at the case of a smooth projective variety $X$ over $k$, equipped with the log structure induced by an effective Cartier divisor. By this, in this case, we mean the DF structure induced by the functor $\bb{N}\to \Div(X)$ sending $1$ to $(\mc{O}_{X}(D),s)$, and \emph{not} the divisorial log structure of \ref{divisorial.log.str}.

In this situation, moduli spaces of parabolic sheaves with rational weights have been constructed in \cite{maruyama-yokogawa}, by adapting the GIT construction of moduli spaces of (semi-)stable coherent sheaves on a projective scheme. Their result in turn generalizes the first results of Seshadri (\cite{seshadri}) on curves.

Let us recall their definition of a parabolic sheaf. Let $X$ be a projective smooth scheme over $k$, and $D\subseteq X$ an effective Cartier divisor.

\begin{definition}
A \emph{MY-parabolic sheaf} $E_{*}$ on $X$ is given by the following data:
\begin{itemize}
\item a coherent torsion-free sheaf $E \in \Coh(X)$,
\item a sequence of real numbers $a_{1},\hdots,a_{k}$ called \emph{weights}, such that $0\leq a_{1}<a_{2}<\cdots<a_{k}<1$, and
\item a filtration $E(-D)=F_{k+1}(E)\subset F_{k}(E)\subset \cdots \subset F_{1}(E)=E$ of $E$, where $E(-D)$ is the image of the natural map $\mc{O}_{X}(-D)\otimes E\to E$.
\end{itemize}
\end{definition}

The \emph{rank} of $E_{*}$ will be the rank of the torsion-free sheaf $E$. We will assume that the weights $a_{1},\hdots, a_{k}$ are rational numbers. This is crucial for the correspondence with quasi-coherent sheaves on a root stack, and also for the moduli theory developed in \cite{maruyama-yokogawa} (the rationality assumption is on page $94$). Moreover, in light of what follows it is more convenient to think about the opposites $-1<-a_{k}<\cdots <-a_{1}\leq 0$.

Let us now describe explicitly how this definition is connected with our definition of a parabolic sheaf.

As we already remarked, the divisor $D$ induces a log structure on $X$, given by the morphism $L\colon \bb{N}\to \Div(X)$ that sends $1 \in \bb{N}$ to $(\mc{O}_{X}(D),s)$, where $s$ is the section of $\mc{O}_{X}(D)$ corresponding to the natural map $\mc{O}_{X}\to \mc{O}_{X}(D)$. Given a MY-parabolic sheaf $E$, let us take $n$ to be the least common multiple of the denominators of the weights $a_{i}$, and consider the system of denominators $j\colon \bb{N}\to \frac{1}{n}\bb{N}$.

Now we define a parabolic sheaf, that we still denote by $E \in \Par(X,j)$, as follows. Set $a_{i}=\frac{b_{i}}{n}$ with $b_{i}\in \bb{N}$, and for $q=\frac{a}{n}$ with $-n\leq  a < 0$ define
$$
E_{\frac{a}{n}}=F_{i}(E) \hspace{1cm} \mbox{           where           }  \hspace{1cm}-{b_{i}}\leq {a}< -{b_{i-1}},
$$
with the convention that $b_{0}=0$ and $b_{k+1}=n$. For an arbitrary $\frac{a}{n} \in \bb{Q}$, we set 
$$
E_{\frac{a}{n}}=E_{\frac{a'}{n}}\otimes \mc{O}_{X}(bD) \hspace{1cm} \mbox{ where } \hspace{1cm} -n\leq a'< 0,\;\; b \in \bb{Z} \hspace{0.2cm} \mbox{ and } \hspace{0.2cm}\frac{a}{n}=\frac{a'}{n}+b,
$$
and for $\frac{a}{n}\leq \frac{a'}{n}$ there is an obvious morphism $E_{\frac{a}{n}}\to E_{\frac{a'}{n}}$, which is either the identity or an inclusion. Moreover by construction there is a pseudo-periods isomorphism, and this gives a parabolic sheaf in our sense.

Conversely, given a parabolic sheaf $E \in \Par(X,j)$ such that the maps $E_{\frac{a}{n}}\to E_{\frac{a'}{n}}$ are all injective, we obtain a MY-parabolic sheaf by taking as weights the opposites of the numbers $\frac{i}{n} \in \bb{Q}\cap \left.( -1,0 \right]$ such that $E_{\frac{i-1}{n}}\to E_{\frac{i}{n}}$ is not an isomorphism, and the filtration consisting of the sheaves $E_{\frac{a}{n}}$ with $-n\leq a \leq 0$, but without repetitions.

This gives an equivalence between MY-parabolic sheaves and parabolic sheaves with injective maps. The injectivity condition is implied by torsion-freeness of $E$, see Proposition \ref{gen.trivial.inj}.

From this description we see that the weights of the MY definition are nothing else than (the opposites of) what we could call ``jumping numbers'' for a parabolic sheaf with injective maps, i.e. the numbers $\frac{i}{n} \in \left.(-1,0\right]$ where the subsheaf $E_{\frac{i}{n}}\subseteq E$ ``jumps'' with respect to the preceding one. 

\begin{rmk}
If $E_{*}$ is a MY-parabolic sheaf and $L \in \Pic(X)$ is an invertible sheaf, then there is a natural MY-parabolic sheaf $E_{*}\otimes L$ obtained by tensoring everything with $L$. In particular this gives for any $m \in \bb{Z}$ the MY-parabolic sheaf $E_{*}(m)=E_{*}\otimes \mc{O}_{X}(m)$.
\end{rmk}

In \cite{maruyama-yokogawa}, in order to construct moduli spaces, they define a parabolic Hilbert polynomial. Let us briefly recall their definitions and results.

\begin{definition}
The \emph{MY-parabolic Euler characteristic} of a MY-parabolic sheaf $E_{*}$ is the rational number
$$
\chi_{MY}(E_{*})= \chi(E(-D))+\sum_{i=1}^{k} a_{i}\chi(G_{i}),
$$
where $G_{i}$ is the quotient $F_{i}(E)/F_{i+1}(E)$.

The \emph{MY-parabolic Hilbert polynomial} of $E_{*}$ is the polynomial with rational coefficients given by
$$
P_{MY}(E_{*})(m)=\chi_{MY}(E_{*}(m)),
$$
where $E_{*}(m)=E_{*}\otimes \mc{O}_{X}(m)$.

The \emph{MY-reduced parabolic Hilbert polynomial} of $E_{*}$ is the polynomial
$$
p_{MY}(E_{*})=\frac{P_{MY}(E_{*})}{\rk(E_{*})}.
$$
\end{definition}

\begin{rmk}
The rank of $E_{*}$ is the rank of the torsion-free sheaf $E$. Note that, since $G_{i}=F_{i}(E)/F_{i+1}(E)$ is generically zero, if instead of using $\rk(E_{*})$ we use the leading coefficient of $P_{MY}(E_{*})$ (as one does when dealing with pure sheaves, not necessarily torsion-free), we would get a scalar multiple of $p_{MY}(E_{*})$, which of course would then give the same stability condition.
\end{rmk}

Maruyama and Yokogawa define then (semi-)stability by defining an appropriate notion of subsheaf, and requiring that 
$$
p_{MY}(F_{*})\mbox{ } (\leq) \mbox{ }p_{MY}(E_{*}).
$$
for every subsheaf $F_{*}\subseteq E_{*}$.

The resulting notion has many properties resembling the ones of classical (semi-)stability for coherent sheaves, for example the existence of Harder-Narasimhan and Jordan-H\"{o}lder filtrations. Moreover by restricting to (semi-)stable sheaves one can construct moduli spaces (Theorem 3.6 of \cite{maruyama-yokogawa} and Theorems 2.9 and 4.6 of \cite{yokogawa}).

To extend these result to general log schemes we aim to find in this particular case a generating sheaf $\mc{E}$ on $\m{X}=\radice{n} X=\radice{\frac{1}{n}\bb{N}}X$ that gives the parabolic Hilbert polynomial of Maruyama and Yokogawa, where the parabolic sheaves have weights in $\frac{1}{n}\bb{N}$.

A little thought produces the locally free sheaf
$$
\mc{E}=\mc{O}_{\m{X}}(\mc{D})\oplus \mc{O}_{\m{X}}(2\mc{D}) \cdots\oplus \mc{O}_{\m{X}}(n\mc{D})=\bigoplus_{i=1}^{n} \mc{O}_{\m{X}}(i\mc{D}),
$$
where $\mc{D}$ is the universal root on $\m{X}$ of the pullback of the divisor $D$: a simple calculation shows that $p_{MY}(E_{*})$ and $P_{\mc{E}}(E)$ coincide up to a constant factor $\frac{1}{n}$ with this choice of $\mc{E}$.

This was basically already observed by Nironi (\cite[Section 7.2]{nironi}).

\begin{rmk}
This says that the notion of (semi-)stability for MY-parabolic sheaves is equivalent to the notion of stability when we use the generating sheaf $\mc{E}$ introduced above, so the two moduli theories that we get should be the same.

There is a minor detail, though, related to the fact that by working on the root stack $\m{X}=\radice{n} X$ we only bound the denominators of the weights (in the divisibility sense), when in \cite{maruyama-yokogawa} and \cite{yokogawa}, the authors fix the jumps of the parabolic sheaves, and the parabolic Hilbert polynomials of the quotients $F_{i}(E)/F_{i+1}(E)$.

One can check that the moduli stacks and spaces of Maruyama and Yokogawa are connected components of the ones that we will construct.
\end{rmk}

\subsubsection{Choice of the generating sheaf: the general case}\label{global.sec}

Now we turn to the general case of a log scheme $X$ with a global chart $L\colon P\to \Div(X)$, and a Kummer extension $j\colon P\to Q$, that gives a chart for the system of denominators $A\to B$. The previous example suggests the following construction: since $Q$ is sharp and fine, its finite number of indecomposable elements are a minimal set of generators (see for example \cite[Proposition 2.1.2]{ogus}), call them $q_{1},\dots,q_{r}$. Moreover, call $d_{i}$ the order of the image of $q_{i}\in Q^{\gp}$ in the quotient $Q^{\gp}/P^{\gp}$.

If $\Lambda \colon Q\to \Div(\m{X})$ is the universal lifting of the log structure of $X$, for every $q_{i}$ we have an associated invertible sheaf $\Lambda_{i}=\Lambda(q_{i})$ on $\m{X}$, and we consider the locally free sheaf
$$
\displaystyle\mc{E}=\mc{E}_{Q/P}=\bigoplus_{1\leq a_{i}\leq d_{i}} \Lambda\left(\sum_{i}a_{i}q_{i}\right)=\bigotimes_{i=1,\dots,r} \left(\bigoplus_{j=1,\dots,d_{i}} \Lambda_{i}^{\otimes j}\right).
$$  
Note that if $X$ is the log scheme given by a variety with an effective Cartier divisor, then this sheaf corresponds to the one described in the last section.

We will denote this sheaf by $\mc{E}$ when the Kummer extension is clear, and by $\mc{E}_{Q/P}$ when it needs to be specified. In particular we will write $\mc{E}_{n}$ for $\mc{E}_{\frac{1}{n}P/P}$, or more generally $\mc{E}_{\frac{1}{n}Q/P}$ for a fixed Kummer extension $P\subseteq Q$, which will be clear from the context.

\begin{rmk}\label{gen.choice}
One could argue that the sheaf
$$
\displaystyle\mc{E}'=\bigoplus_{0\leq a_{i}< d_{i}} \Lambda\left(\sum_{i}a_{i}q_{i}\right)=\bigotimes_{i=1,\dots,r} \left(\bigoplus_{j=0,\dots,d_{i}-1}\Lambda_i^{\otimes j}\right).
$$
in which we take $\mc{O}_{\m{X}}$ instead of $\Lambda_{i}^{\otimes d_{i}}$ (which is the pullback of something from $X$, so corresponds to the trivial representation of the stabilizer at any point of $\m{X}$) in each summand would be somewhat more natural. In fact one could twist different pieces of the direct sum with an invertible sheaf coming from $X$ and still have a perfectly good generating sheaf.

The choice of the one we singled out is guided by the fact that in the case of a variety with a divisor it gives back the (semi-)stability of Maruyama and Yokogawa, and, as we will see in the next section, it will allow semi-stability to be preserved after changing denominators, something that does not happen for example with the generating sheaf written down in the last formula. Actually we will also need to use the alternative sheaf $\mc{E}'$ in that instance, but only as an accessory. 
\end{rmk}

\begin{prop}\label{gen.sheaf}
The locally free sheaf $\mc{E}$ is a generating sheaf on $\m{X}$.
\end{prop}

\begin{proof}
This follows with a routine check from the definition of the generating sheaf: each of the invertible sheaves $\Lambda(\sum_{i} a_{i}q_{i})$ carries the action of a single character of the stabilizer group of a point of $\m{X}$, and the fact that the $q_{i}$'s generate the quotient $Q^{\gp}/P^{\gp}$ assures that all characters appear.

In more detail, recall from the discussion following Proposition \ref{local.model.root.stack} that the global chart $P\to Q$ for the system of denominators gives the following description for the stack of roots: the map $X\to [\spec(k[P])/\what{P}]$ corresponding to the chart $P\to \Div(X)$ gives a $\what{P}$-torsor $\eta\colon E\to X$, and $\m{X}$ is isomorphic to the quotient stack $[E\times_{\spec(k[P])}\spec(k[Q]) /\what{Q}]$, where the action of $\what{Q}$ on the first factor is induced by the action of $\what{P}$ on $E$ and the natural homomorphism $\what{Q}\to \what{P}$. In particular a quasi-coherent sheaf on $\m{X}$ corresponds to a $\what{Q}$-equivariant  quasi-coherent sheaf on $E\times_{\spec(k[P])}\spec(k[Q])$, or equivalently to a $Q^{\gp}$-graded quasi-coherent sheaf of modules over the sheaf of rings $B=A\otimes_{k[P]} k[Q]$, where $A=\eta_{*}\mc{O}_{E}$.

Now fix a geometric point $p \to \m{X}$. We will show that the fiber $\mathcal{E}_{p}=p^{*}\mc{E}$ at $p$ of $\mathcal{E}$ contains every irreducible representation of the stabilizer group $\Stab(p)\subseteq \mu_{Q/P} \subseteq \what{Q}$, where $\mu_{Q/P}$ is the Cartier dual of the cokernel $C$ of $P^{\gp}\to Q^{\gp}$ (as in the discussion preceding \ref{root.stack.algebraic}). Note that, being $\Stab(p)$ a closed subgroup of a diagonalizable group we will have $\Stab(p)=\D (M)$ for a quotient $M$ of the group $C$, and the action of $\Stab(p)$ on $\mc{E}_{p}$ will correspond to an $M$-grading. Moreover since $\Stab(p)$ is diagonalizable, irreducible representations correspond to characters, so what we need to verify is that in the $M$-grading every piece is non zero.

This grading is obtained as follows: the $Q^{\gp}$-grading on the sheaf corresponding to $\mc{E}$ on $E\times_{\spec(k[P])}\spec(k[Q])$ is inherited by the various summands, and by construction the sheaf corresponding to $\Lambda_{i}=\Lambda(q_{i})$ is in degree $q_{i}$. This gives a $Q^{\gp}$ grading on $\mc{E}_{p}$ by pulling back, and we finally get the $M$-grading by means of the homomorphism $Q^{\gp}\to C\to M$.

More explicitly, following through the above we find
$$
\mc{E}_{p}\cong \bigoplus_{m \in M} k(p)^{\oplus{a(m)}}
$$
where $a(m)$ is the number of $r$-tuples $(e_{1},\hdots, e_{r})$ of integers such that $0<e_{i}\leq d_{i}$ and $e_{1}m_{1}+\hdots+e_{r}m_{r}=m$, where $m_{i}$ is the image of $q_{i}$ in $M$. Since the $q_{i}$'s generate $C$, the $m_{i}$'s will generate $M$ (and still have order at most $d_{i}$), so $a(m)\geq 1$ for any $m$. This means that every character of $\Stab(p)$ appears in $\mc{E}_p$, so $\mc{E}$ is a generating sheaf on $\m{X}$.
\end{proof}

This settles the choice of a generating sheaf for a polarized log scheme and a Kummer extension of monoids.
\begin{rmk}
One would hope to be able to generalize this construction to an arbitrary log scheme (i.e. without a global chart for the log structure), by patching the local generating sheaves on open subsets where there is a chart.

Unfortunately, it is not so clear how to do this, even in very natural situations, for example for the log scheme given by a projective surface $X$ with an irreducible nodal curve $D\subseteq X$.

One can generalize the construction of the generating sheaf by adding an extra piece of data (which is not guaranteed to exist in general) to the log structure, that we could call a \emph{locally constant sheaf of charts}, which is a sheaf of monoids which is locally constant, and gives a chart for the sheaf $B$ of the system of denominators $A\to B$.

Since this slightly more general situation will not play any role in the rest of the paper, we refer the interested reader to \cite[Section 3.3.3]{tesi}.
\end{rmk}

\subsubsection{Results}\label{results}

Let us describe the results that we get by applying Nironi's theory, with the choice of generating sheaf that we just described. The proof of the results that are just stated here can all be found in \cite{nironi}.

\begin{rmk}\label{characteristic}
In order to apply Nironi's theory we have to assume that the root stack $\radice{B} X$ is Deligne--\hspace{0pt}Mumford. For example, this is assured by the condition the $\cha(k)$ does not divide the order of the quotient $B_{x}^{\gp}/A_{x}^{\gp}$ for any geometric point $x\to X$ (\ref{root.stack.algebraic}). We will include this assumption from now on.

This will force us to assume that $\cha(k)=0$ in the next section, since we will have to consider a cofinal system of root stacks, and, for example if we consider the system of root stacks $\radice{n}{X}$, we would have to exclude indices divisible by some fixed prime $p$.

We remark that it seems likely that Nironi's theory also applies to tame Artin stacks, without the Deligne--\hspace{0pt}Mumford assumption. If this were true our results would hold in arbitrary characteristic.
\end{rmk}

\begin{notation}
From now on we will use the same letter to denote a coherent sheaf $E \in \Coh(\m{X})$ on the root stack and the corresponding parabolic sheaf $\Phi(E) \in \Par(X,j)$. In particular we will denote by $E_{b}$ the piece $\pi_{*}(\Phi(E)\otimes \Lambda_{b}) \in \Coh(X)$ of the parabolic sheaf corresponding to the element $b \in B^{\wt}(U)$. 

We will also be drawing parabolic sheaves more often. We recall how to think about them, in the case where there is a global chart $P\to Q$ for the system of denominators (Definition \ref{def.parabolic.sheaf.chart}): there is a quasi-coherent sheaf $E_{q}$ on $X$ on each point $q$ of the lattice $Q^{\gp}$ and maps $E_{q}\to E_{q'}$ exactly when $q\leq q'$, in the sense that there exists $q'' \in Q$ such that $q'=q+q''$. If $p \in P$, then the sheaf $E_{q+p}$ is isomorphic to $E_{q}\otimes L_{p}$, and the map $E_{q}\to E_{q+p}$ corresponds to multiplication by the distinguished section of $L_{p}$. In practice it will be enough to consider a small portion of the sheaf, which will determine it uniquely via the pseudo-periods isomorphism (see below for examples).
\end{notation}

The starting point of the moduli theory of sheaves is the definition of the generalized slope $p_{\mc{E}}(E) \in \bb{Q}[m]$ for a parabolic sheaf $E \in \Par(X,j)$ (Definition \ref{gen.slope}). We recall that it is defined as
$$
p_{\mc{E}}(E)(m)=\frac{P_{\mc{E}}(E)(m)}{\alpha^{d}(E)}=\frac{\chi(\pi_{*}(E\otimes \mc{E}^{\vee})(m))}{\alpha^{d}(E)}
$$
where $d$ is the degree of the Hilbert polynomial $P_{\mc{E}}(E)=P(\pi_{*}(E\otimes \mc{E}^{\vee}))$, and $\alpha^{d}(E)$ is $d!$ times its leading coefficient, a positive rational number.

Let us take a closer look at this Hilbert polynomial and slope. Since $\mc{E}$ is a sum of line bundles, $\pi_{*}(E\otimes \mc{E}^{\vee})$ splits as a direct sum
\begin{align*}
\pi_{*}(E\otimes \mc{E}^{\vee})&=\pi_{*} (E\otimes ( \bigoplus_{1\leq a_{i}\leq d_{i}} {\Lambda_{1}^{\otimes a_{1}}\otimes \cdots \otimes \Lambda_{r}^{\otimes a_{r}}} )^{\vee} )\\ &=\bigoplus_{1\leq a_{i}\leq d_{i}}\pi_{*}(E\otimes ({\Lambda_{1}^{\otimes a_{1}}\otimes \cdots \otimes \Lambda_{r}^{\otimes a_{r}}})^{\vee})=\bigoplus_{1\leq a_{i}\leq d_{i}} E_{-\sum a_{i}q_{i}}
\end{align*}
of pieces of $E$ in some kind of (negative) ``fundamental region'' for the Kummer extension $P\to Q$.

We will give a name to the pieces of $E$ that show up in this decomposition.

\begin{definition}
The \emph{fundamental pieces} of a parabolic sheaf $E $ are the pieces $E_{q}$ with $q=-\sum a_{i}q_{i}$ and $1\leq a_{i}\leq d_{i}$, where $q_{i}$ are the indecomposable elements of $Q$ and $d_{i}$ is the order of the image of $q_{i}$ in the quotient $Q^{\gp}/P^{\gp}$.
\end{definition}

\begin{example}\label{n2sq}
For example, if $P=\bb{N}^{2}$ and we are considering the extension $j\colon \bb{N}^{2}\to \frac{1}{2}\bb{N}^{2}$, then for a parabolic sheaf $E \in \Par(X,j)$ the fundamental pieces are the four sheaves in the ``negative unit square''
$$
\xymatrix{
  -1 & -\frac{1}{2} &   \\
 E_{-1,-\frac{1}{2}} \ar[r] & E_{-\frac{1}{2},-\frac{1}{2}} & -\frac{1}{2}\\
 E_{-1,-1}\ar[u]\ar[r]& E_{-\frac{1}{2},-1} \ar[u] & -1
}
$$
A similar description holds if $P$ is free and we are considering the extension $P\to \frac{1}{n}P$.
\end{example}

From the fundamental pieces of a parabolic sheaf we can reconstruct all of its pieces, since for any $q\in Q^{\gp}$ there is a $p \in P^{\gp}$ such that $q+p=-\sum a_{i}q_{i}$ for $1\leq a_{i}\leq d_{i}$, and consequently $E_{q} \simeq E_{-\sum a_{i}q_{i}} \otimes L_{p}^{\vee}$ (we can even reconstruct the whole sheaf from the fundamental pieces and the maps  $E_{q}\to E_{q+q_{i}}$, where $E_{q}$ is a fundamental piece).

\begin{rmk}
Note that if $P$ is not free, then it is not necessarily the case that every fundamental piece shows up exactly once in $\pi_{*}(E\otimes \mc{E}^{\vee})$.
\end{rmk}

From the splitting of $\pi_{*}(E\otimes \mc{E}^{\vee})$ described above we also see that
$$
P_{\mc{E}}(E)=P(\pi_{*}(E\otimes \mc{E}^{\vee}))=P(\bigoplus_{1\leq a_{i}\leq d_{i}} E_{-\sum a_{i}q_{i}})=\sum_{1\leq a_{i}\leq d_{i}}P(E_{-\sum a_{i}q_{i}})
$$
is the sum of the Hilbert polynomials of the fundamental pieces of $F$. Consequently, assuming that the fundamental pieces of $F$ all have dimension $d$ (and recall that by our conventions the zero sheaf is pure of any dimension), for the slope of $F$ we have
\begin{equation}\label{weighted.mean}
p_{\mc{E}}(E)=\frac{\sum_{1\leq a_{i}\leq d_{i}}P(E_{-\sum a_{i}q_{i}})}{\sum_{1\leq a_{i}\leq d_{i}}\alpha^{d}(E_{-\sum a_{i}q_{i}})}=\sum_{1\leq a_{i}\leq d_{i}} \gamma_{(a_{i})}p(E_{-\sum a_{i}q_{i}})
\end{equation}
where $d$ is the dimension of $E$, and
$$
\gamma_{(a_{i})}=\frac{\alpha^{d}(E_{-\sum a_{i}q_{i}})}{\sum_{1\leq a_{i}\leq d_{i}}\alpha^{d}(E_{-\sum a_{i}q_{i}})}
$$
are rational numbers such that $0\leq \gamma_{(a_{i})}\leq 1$ and $\sum_{1\leq a_{1}\leq d_{i}}\gamma_{(a_{i})}=1$.

In other words the slope of the parabolic sheaf $E$ (provided that all its non-zero pieces are of the same dimension) is a weighted mean of the slopes (i.e. reduced Hilbert polynomials as sheaves on $X$) of its non-zero fundamental pieces. The condition about the pieces is satisfied in particular if $E$ is pure (\ref{prop.pure}).

\begin{definition}
A parabolic sheaf is \emph{(semi-)stable} if it is pure, and for any subsheaf $G\subseteq E$ we have
$$
p_{\mc{E}}(G)\mbox{ } (\leq) \mbox{ }p_{\mc{E}}(E).
$$
\end{definition}

As is usually done in moduli theory of sheaves, we write $(\leq)$ to indicate that one should read $\leq$ when considering semi-stability, and $<$ when considering stability.

\begin{rmk}
It is natural to ask if this notion of (semi-)stability is independent of the global chart (through the construction of the associated generating sheaf). The answer, as we will show with an example, is no (see \ref{dependence}).
\end{rmk}

\begin{example}
It is clear from the preceding discussion that if the fundamental pieces of a parabolic sheaf $E$ are all (Gieseker) semi-stable (as coherent sheaves on $X$, with respect to the same polarization that we fixed at the beginning), then $E$ will be semi-stable.
\end{example}

This notion of stability has many of the properties of the classical notion of Gieseker stability on a projective scheme.

For example (semi-)stability can be checked on \emph{saturated} subsheaves $G\subseteq E$, i.e. subsheaves such that the quotient $E/G$ is pure of the same dimension as $E$. This implies that line bundles, if they are pure (essentially if $X$ is Cohen-Macaulay), are all stable. Moreover, a direct sum $E_{1}\oplus E_{2}$ is never stable, and is semi-stable if and only if $E_{1}$ and $E_{2}$ are semi-stable of the same slope.

The following two results are also identical to the corresponding ones for classical moduli theory of sheaves. They provide filtrations that ``break up'' a parabolic sheaf in semi-stable pieces, and a semi-stable parabolic sheaf in stable pieces.

\begin{prop}[Harder-Narasimhan filtration]
For any parabolic sheaf $E \in \Par(X,j)$ there is a filtration
$$
0=E_{0}\subset E_{1}\subset \cdots\subset E_{k}\subset E_{k+1}=E
$$
such that
\begin{itemize}
\item the quotients $E_{i}/E_{i-1}$ are semi-stable for $i=1,\hdots,k+1$; call $p_{i}$ the slope $p_{\mc{E}}(E_{i}/E_{i-1})$,
\item the slopes are such that $p_{1}> \cdots> p_{k+1}$.
\end{itemize}
Moreover this filtration is unique, and it is called the \emph{Harder-Narasimhan filtration} of the parabolic sheaf $E$.
\end{prop}

\begin{prop}[Jordan-H\"{o}lder filtration]
For any semi-stable parabolic sheaf $E \in \Par(X,j)$ there is a filtration
$$
0=F_{0}\subset F_{1}\subset \cdots \subset F_{h}\subset F_{h+1}=E
$$
such that the quotients $F_{i}/F_{i-1}$ are stable with slope $p_{\mc{E}}(E)$ for $i=1,\hdots, h+1$.

This filtration is not unique, but the set $\{ F_{i}/F_{i-1}\}_{i=1,\hdots, h+1}$ of partial quotients of the filtration is unique, as is their direct sum
$$
\gra(E)=\bigoplus_{i=1}^{h+1} F_{i}/F_{i-1},
$$
sometimes called the \emph{associated graded sheaf} of $E$.

Any such filtration is called a \emph{Jordan-H\"{o}lder filtration} of the parabolic sheaf $E$.
\end{prop}

Note that since the quotients $F_{i}/F_{i-1}$ of the above filtration are stable with the same slope, the parabolic sheaf $\gra(E)$ is semi-stable with the same slope as $E$.

\begin{definition}[S-equivalence]
Two parabolic sheaves $E,E' \in \Par(X,j)$ are said to be \emph{S-equivalent} if their associated graded sheaves $\gra(E)$ and $\gra(E')$ are isomorphic.

Equivalently one can say that the sets $\{F_{i}/F_{i-1}\}$ and $\{F'_{i}/F'_{i-1}\}$ of quotients of a Jordan-H\"{o}lder filtrations of the two sheaves are the same, i.e. such quotients are pairwise isomorphic.
\end{definition}

Recall that a semi-stable sheaf is called \emph{polystable} if it is a direct sum of stable sheaves, which then must all have the same slope. Every parabolic semi-stable sheaf $E \in Par(X,j)$ is S-equivalent to exactly one polystable sheaf, the sheaf $\gra(E)$.

The notion of (semi-)stability satisfies some openness and boundedness conditions, as shown in \cite{nironi}. We summarize the final product of the theory.

\begin{definition}
Fix a polynomial $H \in \bb{Z}[x]$, and define the \emph{moduli stack of semi-stable parabolic sheaves on} $X$, denoted by $\mc{M}^{ss}_{H}$, as the stack over $(\Sch)$ having as objects of $\mc{M}^{ss}_{H}(T)$ for a scheme $T$ families of parabolic sheaves $E \in \Par(X_T,j_T)$ (in the sense of \ref{sec.3.1}) such that for every geometric point $t\to T$, the restriction $E_{t} \in \Par(X_{t},j_{t})$ is pure and semi-stable with Hilbert polynomial $H$.

Denote by $\mc{M}^{s}_{H} \subseteq \mc{M}^{ss}_{H}$ the subcategory parametrizing families of parabolic sheaves that are stable on the fibers, instead of just semi-stable. This is an open substack.
\end{definition}

The pullback $\mc{M}^{ss}_{H}(T)\to  \mc{M}^{ss}_{H}(S)$ for $S\to T$ is the pullback of parabolic sheaves we discussed earlier (\ref{sec.3.1}). By the discussion in the same section, the stack $\mc{M}^{ss}_{H}$ coincides with the stack of families of semi-stable sheaves with Hilbert polynomial $H$ on the root stack $\m{X}$ studied by Nironi (\cite{nironi}).

Note that of course this stack also depends on the Kummer extension $P\to Q$, but we did not include it in the notation to keep it lighter.

\begin{rmk}
To define (semi-)stability on the base change $X_{t}=X\times_{k}\spec(k(t))$ we use the pullback of the generating sheaf $\mc{E}$ that we have on $\m{X}$ along the natural map $\radice{B_{t}}{X_{t}}\to \radice{B} X=\m{X}$.
\end{rmk}

Here is the result that we obtain from \cite[Section $6$]{nironi}, just by applying Theorems 6.21 and 6.22 to our situation.

\begin{thm}\label{thm.fixed.weights}
Let $X$ be a polarized log scheme over $k$, with global chart $P\to \Div(X)$, and with a Kummer morphism of fine saturated monoids $P\to Q$. Moreover assume that the root stack $\radice{Q} X$ is Deligne--Mumford.

Then the stack  $\mc{M}^{ss}_{H}$ of semi-stable parabolic sheaves is an Artin stack of finite type over $k$, and it has a good (resp. adequate, in positive characteristic) moduli space in the sense of Alper (\cite{alper,alperadeq}), that we denote by $M^{ss}_{H}$. This moduli space is a projective scheme.

The open substack $\mc{M}^{s}_{H}\subseteq \mc{M}^{ss}_{H}$ of stable sheaves has a coarse moduli space $M^{s}_{H}$, which is an open subscheme of $M^{ss}_{H}$, and the map $\mc{M}^{s}_{H}\to M^{s}_{H}$ is a $\bb{G}_{m}$-gerbe. \qed

\end{thm}

Some comments about this theorem.

\begin{rmk}
We chose to fix the Hilbert polynomial in this formulation, but one can also fix other invariants of coherent sheaves, for example Chern classes (up to numerical equivalence), or the reduced Hilbert polynomial $h$. The corresponding moduli stacks are defined analogously, and the results one obtains translate verbatim.

In particular, in the next section we will fix the reduced Hilbert polynomial $h\in \bb{Q}[x]$, that is obtained from $H$ by dividing it by $d!$ times its leading coefficient, $d$ being the degree, and consider the corresponding stacks $\mc{M}^{s}_{h}\subseteq \mc{M}^{ss}_{h}$. Note that for example $\mc{M}^{ss}_{h}$ will be a disjoint union
$$
\mc{M}^{ss}_{h}=\bigsqcup_{\overline{H}=h} \mc{M}^{ss}_{H}
$$
where $\overline{H}$ is $H$ divided by $d!$ times its leading coefficient. This has a good moduli space, the disjoint union of the corresponding moduli spaces, and the same is true for the substack of stable sheaves.

In the same fashion one can form the disjoint union $\mc{M}^{ss}=\bigsqcup_{H \in \bb{Z}[x]}\mc{M}^{ss}_{H}$ and the analogous one for stable sheaves. This stack, that parametrizes parabolic sheaves on $X$ with respect to the Kummer extension $P\to Q$ without fixing invariants, will still have a good moduli space, the disjoint union of the $M^{ss}_{H}$.
\end{rmk}

\begin{rmk}
The points of the good moduli space ${M}^{ss}_{H}$ do not correspond to isomorphism classes of semi-stable sheaves, but rather to S-equivalence classes, or, in other words, to isomorphism classes of polystable sheaves. This follows from the GIT construction.

Moreover, a point of the stack $\mc{M}^{ss}_{H}$ is closed if and only if the corresponding sheaf is polystable. This follows from the description as a quotient and from the fact that an orbit of a point is closed if and only if it is polystable (see Theorem $6.20$  of \cite{nironi}).
\end{rmk}

\begin{rmk} \label{comp.maru.yoko}

We remark that this construction recovers the moduli spaces of Maruyama and Yokogawa. The only difference with our construction is that they fix the weights (or ``jumping numbers'', see \ref{maru.yoko}) of the parabolic sheaves.

One can check easily that fixing the weights gives a connected component in our moduli stack of parabolic sheaves, and that this component gives back the moduli spaces of Maruyama and Yokogawa.
\end{rmk}

\subsubsection{Dependance of (semi-)stability on the global chart}\label{dependence}
Since there are in general many choices for a chart of a logarithmic structure with a Kummer morphism (when at least one exists), the problem of the dependance of the (semi-)stability of a parabolic sheaf on the chart is a very natural one. It turns out that the (semi-)stability is not independent of the chart, as we will show with the following example.

Take $X=\mathbb{P}^{1}\times\mathbb{P}^{1}$, with effective divisor $D=D_{1}+D_{2}$ where $D_{1}$ and $D_{2}$ are two distinct closed fibers of the first projection $X\to \mathbb{P}^{1}$, so that $\mathcal{O}(D_{1}) \simeq \mathcal{O}(D_{2}) \simeq \mathcal{O}(1,0)$ and $\mathcal{O}(D) \simeq \mathcal{O}(2,0)$. The DF structure induced by $D$, call it $L\colon A\to \Div_{X}$, has two natural charts $l\colon \mathbb{N}\to \Div(X)$, sending $1$ to $(\mathcal{O}(D),s_{D})$, with $s_{D}$ the canonical section of $\mathcal{O}(D)$ as usual, and $l'\colon \mathbb{N}^{2}\to \Div(X)$, sending $(1,0)$ to $(\mathcal{O}(D_{1}),s_{D_{1}})$ and $(0,1)$ to $(\mathcal{O}(D_{2}),s_{D_{2}})$. Let us focus on $l'$.

Notice that any cokernel of monoids $P\to \mathbb{N}^{2}$ would give us a new chart $P\to \Div(X)$ for the DF structure $L$ by composing with $l'$, since the composite of two cokernels is still a cokernel (this is an easy verification). 

Consider the monoid $P=\mathbb{N}^{4}/(e_{1}+e_{2}=e_{3}+e_{4})$, where the $e_{i}$'s are the canonical basis,  call $p_{i}$ the image of $e_{i}$ in $P$, and consider the morphism $\phi\colon P\to \mathbb{N}^{2}$ determined by
$$
\begin{array}{l}
\phi(p_{1})=(1,0)\\
\phi(p_{2})=(0,1)\\
\phi(p_{3})=(1,1)\\
\phi(p_{4})=(0,0).
\end{array}
$$
We leave to the reader to check that $\phi$ is a cokernel, and so the composition $P\to \mathbb{N}^{2}\to \Div_{X}$ gives a chart for the DF structure $L$.

Now take the Kummer morphism $j\colon A\to \frac{1}{2}A$, and as usual call $\m{X}$ the stack of roots of $X$ with respect to $j$, denote by $\mathcal{D}_{1}, \mathcal{D}_{2}$ and $\mathcal{D}=\mathcal{D}_{1}+\mathcal{D}_{2}$ the universal square roots of $D_{1},D_{2}$ and $D$ respectively, and call $\mc{E}$ and $\mc{E}'$ the two generating sheaves associated to the charts $l'\colon \mathbb{N}^{2}\to \Div(X)$ and $l'\circ \phi\colon P\to \Div(X)$. By following the construction of the generating sheaf we get
$$
\mc{E}=\mc{O}(\mc{D})\oplus\mathcal{O}(2\mathcal{D}_{1}+\mc{D}_{2})\oplus\mathcal{O}(\mc{D}_{1}+2\mathcal{D}_{2})\oplus\mathcal{O}(2\mathcal{D})
$$
and, noting that the indecomposable elements of $P$ are precisely the $p_{i}$'s, we get $\mc{E}'=\mc{E}''\oplus \mc{E}''$ (since $\Lambda_{p_{4}} \simeq\Lambda_{2p_{4}} \simeq\mathcal{O}$), where
\begin{multline*}
\mc{E}''=\mc{O}(2\mc{D})\otimes (\mathcal{O}\oplus\mathcal{O}(\mathcal{D}_{1})\oplus\mathcal{O}(\mathcal{D}_{2})\oplus\mathcal{O}(\mathcal{D})\oplus \\ \oplus \mathcal{O}(2\mathcal{D}_{1}+\mathcal{D}_{2})\oplus\mathcal{O}(\mathcal{D}_{1}+2\mathcal{D}_{2})\oplus\mathcal{O}(\mathcal{D})\oplus\mathcal{O}(2\mathcal{D}));
\end{multline*}
in particular $p_{\mc{E}'}(F)=p_{\mc{E}''}(F)$ for any parabolic sheaf $F \in \Par(X,j)$.

Our objective is to find a parabolic sheaf $F \in \Par(X,j)$ that is $\mc{E}$-semi-stable but not $\mc{E}'$-semi-stable. Our example will be an extension of two line bundles $L$ and $L'$ on the root stack $\m{X}$: the point will be that such an extension is semi-stable (and in fact stable, if is not trivial) if and only if $L$ and $L'$ have the same slope. To find an $F$ with the property we want, it will suffice then to find $L$ and $L'$ such that
$$
p_{\mc{E}}(L)=p_{\mc{E}}(L')
$$
but
$$
p_{\mc{E}'}(L)\neq p_{\mc{E}'}(L').
$$

Now recall that to give a torsion-free parabolic sheaf $F \in \Par(X,j)$, it suffices to give a torsion-free coherent sheaf $F_{0} \in \Coh(X)$, together with a subsheaf $F_{1}\subseteq F_{0}$ such that $F_{0}(-D)\subseteq F_{1}$. In particular we can take $F_{1}=F_{0}$, and we get a (somewhat trivial) parabolic sheaf, we will denote it by $\widetilde{F_{0}} \in \Par(X,j)$. For such a parabolic sheaf, an easy calculation shows that
$$
p_{\mc{E}}\left(\widetilde{F_{0}}\right) =\frac{1}{4}(p(F_{0})+2p(F_{0}(-D_{1}))+p(F_{0}(-D)))
$$
where $p$ denotes the usual Gieseker (generalized) slope of the coherent sheaf $F_{0}$ on $X$ with respect to the fixed polarization, that in our case will be $H=\mathcal{O}(1,1)$.

The same calculation for $\mc{E}''$ gives
$$
p_{\mc{E}''}\left(\widetilde{F_{0}}\right)  =  \frac{5}{8}p(F_{0}(-D))+\frac{1}{8}\left(p(F_{0}(-2D))+2p(F_{0}(-D-D_{1}))\right)
$$
Now take $F_{0}=\mathcal{O}(2,0)$ and $F_{0}'=\mathcal{O}(1,1)$. Finally using
$$
\chi(\mathcal{O}(a,b)(m))=(a+1+m)(b+1+m)
$$
and $\mathcal{O}(-D) \simeq \mathcal{O}(-2,0)$, $\mathcal{O}(-D_{1}) \simeq\mathcal{O}(-1,0)$, we get
$$
p_{\mc{E}}\left(\widetilde{L_{0}}\right)=m^{2}+3m+2 =p_{\mc{E}}\left(\widetilde{L_{0}'}\right)
$$
but
$$
p_{\mc{E}''}\left(\widetilde{L_{0}}\right)=m^{2}+\frac{3}{2}m+\frac{1}{2}
$$
and
$$
p_{\mc{E}''}\left(\widetilde{L_{0}'}\right)=m^{2}+\frac{3}{2}m-1
$$
as we wanted.

\begin{rmk}
This shows that the notion of semi-stability depends on the chart. We do not know if the moduli spaces for different charts are isomorphic (which might happen even if the notion semi-stability is different) or not.
\end{rmk}

\section{Moduli of parabolic sheaves with arbitrary rational weights}\label{section.4}

In this section we consider moduli of parabolic sheaves with rational weights on a polarized log scheme $X$, without bounding the denominators or fixing a finitely generated Kummer extension. In other words we consider moduli theory of finitely presented sheaves on the infinite root stack $\infroot{X}$.

We remark that the previous works of Mehta-Seshadri (\cite{metha-seshadri}) and Maruyama-Yokogawa (\cite{maruyama-yokogawa}) only considered moduli spaces for parabolic sheaves with a fixed set of weights. 

The natural approach for this problem is to take a limit of the moduli theory at finite levels, and this is what we will do in this section.

The global chart $P\to \Div(X)$ of the polarized log scheme $X$ gives us, as we saw in the last section, the generating sheaves $\mc{E}_{n}$ on the root stacks $\radice{n}{X}=\radice{\frac{1}{n}P}{X}$, and the moduli spaces and stacks $\mc{M}^{s}_{n}\subseteq \mc{M}^{ss}_{n}$ and $M^{s}_{n}\subseteq M^{ss}_{n}$ of (semi-)stable parabolic sheaves on $X$ (here the subscript keeps track of the denominators, and the Hilbert polynomial is not fixed for now).

In order to have this for every $n$, in this section we will assume that the characteristic of $k$ is zero. As remarked in the last section (see \ref{characteristic}), this would be unnecessary if we knew that Nironi's machinery works on tame Artin stacks.

The ideal situation to take a limit would be the following.

\begin{ithm}\label{false.thm}
Let $X$ be a projective log scheme over $k$ with a global chart $P\to \Div(X)$. Then for every pair $n,m \in \bb{N}$ with $n \mid m$ there is a morphism $\iota_{n,m}\colon \mc{M}^{ss}_{n}\to \mc{M}^{ss}_{m}$, that induces $i_{n,m}\colon M^{ss}_{n}\to M^{ss}_{m}$ between the good moduli spaces, given by the pullback along $\radice{m}{X}\to \radice{n}{X}$. Moreover these morphisms are open and closed immersions.
\end{ithm}

This would allow us to make sense of the direct limit $\varinjlim_{n} M^{ss}_{n}$ as a scheme, which would be a moduli space for parabolic sheaves with arbitrary rational weights on $X$. Even without the last statement about the morphisms, the direct limit would still make sense formally as an ind-scheme, and it would still be a good candidate for a ``moduli space'' of parabolic sheaves with arbitrary rational weights.

The first apparent problem about this is that without a flatness hypothesis on $\radice{m}{X}\to \radice{n}{X}$, it is not even clear that the pullback of a semi-stable sheaf is pure. It is easy to construct examples of the form $\spec(k[P])$ where there are pure sheaves whose pullback is not pure.

This leads us to impose some condition that ensure flatness of projections $\radice{Q_{m}}{X}\to \radice{Q_{n}}{X}$, at least for a cofinal system of finite root stacks $\radice{Q_{n}}{X}$, where $\{Q_{n}\}$ is a sequence of Kummer extensions $P\subseteq Q_{n}$, with $Q_{n}\subseteq Q_{m}$ for $n \mid m$ and $\bigcup_{n}Q_{n}=P_{\bb{Q}}$ .

It turns out that requiring this on the universal model $\spec(k[P])$ (for $P$ fine and saturated) forces the monoid $P$ to be \emph{simplicial}. The interested reader can find the proof in \cite[Section~4.1]{tesi}.

\subsection{Simplicial log structures}\label{sec.4.1}

Let us briefly describe simplicial monoids and simplicial log structures. The following definition is inspired by the definition of simplicial cones (see for example the discussion following Corollary 2.3.13 in \cite{ogus}).

\begin{definition}
A monoid $P$ is simplicial if it is fine, saturated, sharp and the positive rational cone $P_{\bb{Q}}$ it generates inside $P^{\gp}_{\bb{Q}}$ is simplicial, meaning that its extremal rays are linearly independent.
\end{definition}

\begin{definition}
An indecomposable element $p \in P$ that lies on an extremal ray of the rational cone $P_{\bb{Q}}$ will be called \emph{extremal}. Non-extremal indecomposables will be called \emph{internal}.
\end{definition}

In other words, an indecomposable $p \in P$ is extremal if $q+r \in \langle p \rangle $ implies $q,r \in \langle p \rangle$. 

Assume $P$ is a simplicial monoid, and call $p_{1},\dots, p_{r}$ its extremal indecomposable elements, and $q_{1},\dots, q_{s}$ its internal ones. For any $q \in P$, we can write $q=\sum_{i}a_{i}p_{i}$ in $P_{\bb{Q}}$, where $a_{i} \in \bb{Q}$, and by simpliciality of $P$ the $a_{i}$ are uniquely determined.

In particular for every $q_{j}$ we have get a relation $c_{j}q_{j}=\sum_{i}a_{ij}p_{i}$ in $P$ where $(c_{j}, \{a_{ij}\})=1$. These relations will be called the \emph{standard relations} of $P$.

\begin{prop}\label{simplicial.envelope}
Every simplicial monoid has a Kummer morphism to some free monoid $\bb{N}^{r}$. Conversely, if a fine saturated monoid $P$ has a Kummer morphism $P\subseteq \bb{N}^{r}$, then $P$ is simplicial.
\end{prop}

In fact there is a minimal such Kummer extension, that we well call the \emph{free envelope} of $P$.
\begin{rmk}
The preceding proposition is the reason for introducing this simpliciality hypothesis. The Kummer extension $P\subseteq \bb{N}^{r}$ gives us a sequence $P\subseteq \frac{1}{n}\bb{N}^{r}=P_{n}$ of finitely generated Kummer extensions such that $\bigcup_{n} P_{n}=P_{\bb{Q}}$, and since $\bb{N}^{r}$ is free, the transition maps $\radice{m}{X}\to \radice{n}{X}$ between the corresponding root stacks are flat, as the following lemma shows.
\end{rmk}
\begin{lemma}\label{free.flat}
Let $X$ be a log stack with a global chart $\bb{N}^{r}\to \Div(X)$. Then for any $n$, the projection $\pi\colon \radice{n}{X}\to X$ is flat.
\end{lemma}

This implies that all projections $\radice{m}{X}\to \radice{n}{X}$ between root stacks are flat as well, and actually for this one needs to assume that the log structure of $X$ is locally free, in the sense that the stalks of $A$ are all free monoids. For example if $D\subseteq X$ is a normal crossings divisor, but not simple normal crossings, then the induced log structure on $X$ is locally free, but does not have a global chart.

\begin{proof}
We can assume that $X$ is a log scheme. Then this follows from the fact that the projection $\radice{n}{X}\to X$ is a base change of the morphism $\left[ \bb{A}^{r}/\mu_{n}^{r} \right]\to \bb{A}^{r}$ induced by the map $\bb{A}^{r}\to \bb{A}^{r}$ given by raising the variables to the $n$-th power. This last morphism is flat, and the conclusion follows.
\end{proof}

This assures that purity of coherent sheaves is preserved by pullback (recall that a semi-stable sheaf is pure). This cofinal system of root stacks will also be crucial for the arguments that we will use in the rest of this section.

Let us now construct for a simplicial monoid $P$ the minimal Kummer extension to a free monoid.

\begin{proof}[Proof of Proposition \ref{simplicial.envelope}]
Let $c_{j}q_{j}=\sum_{i} a_{ij}p_{i}$ be the standard relations of $P$, and let $b_{ij}=c_{j}/\gcd(c_{j},a_{ij})$, a positive integer. The standard relations can be rewritten as follows
$$
q_{j}=\sum_{i}\frac{a_{ij}}{\gcd(c_{j},a_{ij})}\cdot \frac{p_{i}}{b_{ij}}.
$$
Finally, let $d_{i}=\lcm(b_{ij}\mid j=1,\hdots,r)$, and let $F(P)$ be the (free) submonoid of $P^{\gp}\otimes_{\bb{Z}}\bb{Q}$ generated by the elements $\frac{p_{1}}{d_{1}},\hdots, \frac{p_{r}}{d_{r}}$. By construction we have $P\subseteq F(P)$, and $P_{\bb{Q}}=F(P)_{\bb{Q}}$, so the morphism is Kummer.

The converse is clear, since if we have a Kummer morphism $P\subseteq \bb{N}^{r}$, then $P_{\bb{Q}} \simeq \bb{Q}_{+}^{r}$, which is a simplicial cone.
\end{proof}

\begin{definition}
We will call the monoid $F(P)$ constructed in the proof the \emph{free envelope} of $P$. The \emph{rank} of $P$ will be the rank of the free monoid $F(P)$, or equivalently of the free abelian group $P^{\gp}$.
\end{definition}

\begin{example}
Let $P=\langle p,q,r \mid p+q=2r\rangle$. Then $p$ and $q$ are the extremal indecomposables, and the only standard relation
$$
r=\frac{p}{2}+\frac{q}{2}
$$
gives the two generators $\frac{p}{2}$ and $\frac{q}{2}$ for the free envelope $F(P)$.

If we identify $P$ with the submonoid of $\bb{N}^{2}$ generated by $(2,0),(1,1),(0,2)$, then $F(P)$ coincides with $\bb{N}^{2}$, as $\frac{p}{2}=(1,0)$ and $\frac{q}{2}=(0,1)$.
\end{example}

The free envelope has the following universal property.

\begin{prop}
For any Kummer homomorphism $\phi\colon P\to \mathbb{N}^{r}$ there exists a unique (injective) homomorphism $\bar{\phi}\colon F(P)\to \mathbb{N}^{r}$ extending $\phi$.
\end{prop}

The proof is easy and left to the reader.

One can give the following definition of a simplicial log scheme.

\begin{definition}
A fs log scheme $X$ is simplicial if for any geometric point $x\to X$ the stalk $A_{x}$ is a simplicial monoid. 
\end{definition}

Since charts can be made up from stalks, a simplicial log scheme has local charts $P\to \Div(X)$ with $P$ a simplicial monoid.

The converse (if there are simplicial charts, then the stalks are simplicial) is also true, and follows from the fact that the kernel of a morphism $P\to Q$ from a simplicial monoid to a sharp fs monoid is generated by extremal indecomposables. From this one sees that the map $P_{\bb{Q}}\to Q_{\bb{Q}}$ corresponds to a quotient by the span of a subset of a basis of $P_{\bb{Q}}^{\gp}$, and consequently $Q_{\bb{Q}}$ is still a simplicial cone inside $Q_{\bb{Q}}^{\gp}$.

\begin{rmk}
Despite this general definition, for the rest of this section we will assume that $X$ has a global chart $P\to \Div(X)$, in which $P$ is moreover simplicial.
\end{rmk}

\subsection{(semi-)stability and extension of denominators}\label{sec.4.2}

Let us fix some assumptions that will hold for the rest of the section.

\begin{assumptions}
From now on $X$ is a polarized simplicial log scheme, and in the global chart $P\to \Div(X)$ the monoid $P$ is simplicial of rank $r$. 
\end{assumptions}

The first thing we want to do is to replace $X$ by the root stack $X_{1}=\radice{F(P)}{X}$, where $F(P) \simeq \bb{N}^{r}$ is the free envelope of $P$ introduced above. After we have done this, when considering parabolic sheaves on $\radice{n}{X_{1}}=\radice{\frac{1}{n}{F(P)}}X$ we can see them as parabolic sheaves on the log stack $X_{1}$, where the log structure has a free chart, and the transition maps $\radice{m}{X_{1}}\to \radice{n}{X_{1}}$ will be flat (see \ref{free.flat}). This way we can effectively argue as if the log structure on $X$ itself had a free chart to start with.

\begin{notation}
From now on, for brevity, $\radice{n}{X}$ will denote the $n$-th root stack of $X_{1}$, i.e. $\radice{n}{X_{1}}=\radice{\frac{1}{n}F(P)}{X}$, instead of the usual $\radice{\frac{1}{n}P}{X}$.
\end{notation}

Denote by $\mc{E}_{n}$ the generating sheaf on $\radice{n}{X}$ coming from the root stack structure over $X$, and $\widetilde{\mc{E}_{n}}$ the generating sheaf that comes from seeing it as a root stack over $X_{1}$. We would like to say that these two generating sheaves give the same stability. This is true, provided that we equip $X_{1}$ with the right generating sheaf relative to $X$.

The following lemma relates the generating sheaves of two root stacks of $X$, where one of them is obtained by taking $n$-th roots over the other one.

\begin{lemma}\label{gen.sheaves}
Let $X$ be a log scheme with a global chart $P\to \Div(X)$, and let $P\subseteq Q$ be a Kummer extension. Consider the commutative diagram
$$
\xymatrix{
\radice{\frac{1}{n}Q}{X}\ar[rr]^{\pi}\ar[dr]_{p'} & & \radice{Q}{X} \ar[dl]^{p} \\
 & X &
}
$$
and the generating sheaves $\mc{E}_{n}$ on $\radice{\frac{1}{n}Q}{X}$ and $\mc{E}$ on $\radice{Q}{X}$ relative to $X$, and $\mc{E}_{rel}$ on $X_{\frac{1}{n}Q/P}$ obtained by seeing it as a root stack over $\radice{Q}{X}$ with respect to the Kummer extension $Q\subseteq \frac{1}{n}Q$. Denote by $L\colon Q\to \Div(\radice{Q}{X})$ the universal DF structure on $\radice{Q}{X}$ and by $q_{i}$ the indecomposable elements of $Q$.

Then we have an isomorphism
$$
\mc{E}_{n} \simeq \pi^{*}(\mc{E}\otimes M)\otimes \mc{E}_{rel}
$$
where $M=\left(\bigotimes_{i=1}^{r} L(q_{i})\right)^{\vee}=L(\sum_{i}- q_{i})$.
\end{lemma}

\begin{proof}
This follows immediately from an easy calculation and the definition of the generating sheaves (\ref{global.sec}).
\end{proof}

\begin{rmk}
The locally free sheaf $\mc{E}\otimes M$ on $\radice{Q}{X}$ of the previous lemma is still a generating sheaf, and it is precisely the generating sheaf $\mc{E}'$ of Remark \ref{gen.choice}.
\end{rmk}

If we equip $X_{1}$ with the generating sheaf $\mc{E}\otimes M$, then the stability notions on $\radice{n}{X}$ corresponding to $\mc{E}_{n}$ relative to $X$ and $\widetilde{\mc{E}_{n}}$ relative to $X_{1}$ are the same. Indeed if $F \in \Coh(\radice{n}{X})$, then (keeping the notation of the lemma, with $Q=F(P)$) from the previous lemma and the projection formula for $\pi$ we see that
$$
\begin{array}{rcl}
p_{\mc{E}_{n}}(F) & = &p_X(p'_{*}(F\otimes \mc{E}_{n}^{\vee}))\\ 
&= & p_X(p_{*}\pi_{*}(F\otimes \pi^{*}(\mc{E}\otimes M)^{\vee}\otimes \widetilde{\mc{E}_{n}}^{\vee}))\\
& = & p_X(p_{*}(\pi_{*}(F\otimes \widetilde{\mc{E}_{n}}^{\vee})\otimes (\mc{E}\otimes M)^{\vee}))\\
& = & p_{\mc{E}\otimes M}(\pi_{*}(F\otimes \widetilde{\mc{E}_{n}}^{\vee}))
\end{array}
$$
where $p_X$ denotes the reduced Hilbert polynomial on $X$. Note also that if $P$ is already free, then $\mc{E}\otimes M$ is indeed trivial. In conclusion we can replace $X$ by $X_{1}$ in what follows, even though we will keep this notation for clarity.

\begin{notation}\label{notations}
From now on we will fix an isomorphism $F(P) \simeq \bb{N}^{r}$, and denote the canonical log structure on $\radice{n}{X}$ by $L_{n}\colon \frac{1}{n}\bb{N}^{r}\to \Div(\radice{n}{X})$. Moreover $p_{i}$ will denote the indecomposable elements of $F(P) \simeq \bb{N}^{r}$, and for any $r$-tuple of integers $(a_{1},\hdots,a_{r})$, we will denote by $L_{n}^{(a_{i})}$ the invertible sheaf $L_{n}\left(\sum_{i} a_{i}\frac{p_{i}}{n}\right)$ on the root stack $\radice{n}{X}$.

In the same spirit, if $E$ is a parabolic sheaf on $\radice{n}{X}$ and $(e_{1},\hdots, e_{r})$ is an element of $\bb{Z}^{r}$, we denote by $E_{(e_{i})}$ the piece of the parabolic sheaf $E$ corresponding to the element $(\frac{e_{1}}{n},\hdots, \frac{e_{r}}{n})$ of $\frac{1}{n}\bb{Z}^{r}$.

If $F$ is a coherent sheaf on $\radice{n}{X}$, with $p_{n}(F)$ we will denote the reduced Hilbert polynomial $p_{\mc{E}_{n}}(F)$ obtained by using the generating sheaf $\mc{E}_{n}$. We will also denote just by $p$ the reduced Hilbert polynomial on $X_{1}$, with respect to the generating sheaf $\mc{E}\otimes M$ discussed above.

\end{notation}

We summarize here the results of the present section about behavior of (semi-)stability with respect to pullback along maps between root stacks.

\begin{thm}\label{results.4.2}
Let $X$ be a polarized simplicial log scheme over $k$ such that in the global chart $P\to \Div(X)$ the monoid $P$ is simplicial, and $n,m$ two natural numbers with $n \mid m$. Then:
\begin{itemize}
\item the pullback along $\pi \colon \radice{m}{X}\to \radice{n}{X}$ of a semi-stable sheaf is semi-stable (with the same reduced Hilbert polynomial), so we get a morphism $\iota_{n,m}\colon \mc{M}^{ss}_{n}\to \mc{M}^{ss}_{m}$. This morphism in turn induces $i_{n,m}\colon M^{ss}_{n}\to M^{ss}_{m}$ between the good moduli spaces.
\item $\iota_{n,m}$ is always an open immersion, and $i_{n,m}$ is proper, open and injective on geometric points (in particular it is also finite).
\item if pullback along $\pi$ preserves stability (for example this happens if the log structure of $X$ is generically trivial), then $\iota_{n,m}$ restricts to a morphism $\iota^{o}_{n,m}\colon \mc{M}^{s}_{n}\to \mc{M}^{s}_{m}$ between the stacks of stable sheaves, and correspondingly $i_{n,m}$ restricts to $i^{o}_{n,m}\colon M^{s}_{n}\to M^{s}_{m}$ between the good moduli spaces. Moreover in this case all the maps are open and closed immersions.
\end{itemize}
\end{thm}

\begin{rmk}
The fact that $i_{n,m}$ is an open and closed immersion (i.e. an immersion of a union of connected components) will allow us to make sense of the direct limit of the moduli spaces as a scheme.

We will see that the pullback along $\pi$ does not preserve stability in the general hypotheses of the theorem, and if this happens we will not be able to take the direct limit of the stack/space of stable sheaves, nor to conclude that $i_{n,m}$ is an immersion of a union of connected components. Being open, closed and injective, its image will still be a union of connected components, but the map might not be an isomorphism onto the image. We do not have any example where this does happen, though.
\end{rmk}

The following lemma, a particular case of \ref{gen.sheaves}, relates the generating sheaves of $\radice{n}{X}$ and $\radice{m}{X}$, and is the starting point of the proof.

\begin{lemma}
Set $m=nq$, and consider the commutative diagram
$$
\xymatrix{
\radice{m}{X}\ar[rr]^{\pi}\ar[dr]_{p'} & & \radice{n}{X} \ar[dl]^{p} \\
 & X_{1}. &
}
$$
We have an isomorphism
$$
\mc{E}_{m} \simeq \pi^{*}\mc{E}_{n}\otimes \mc{E}_{n,m} \otimes M
$$
where $\mc{E}_{n,m}$ is the generating sheaf of $\radice{m}{X}$ as a root stack of $\radice{n}{X}$ and $M=\left(\bigotimes_{i=1}^{r} L_{m}^{i}\right)^{\otimes (-q)}=L_{m}(\sum_{i}-q \frac{p_{i}}{m})$.\qed
\end{lemma}

\begin{lemma}\label{weighted}
With the notation of the previous lemma, let $G \in \Coh(\radice{m}{X})$ be a coherent sheaf on $\radice{m}{X}$. Then $p_{m}(G)$ is a weighted mean of the reduced Hilbert polynomials of the non-zero sheaves among $\pi_{*}\left(G\otimes  L_{m}^{(d_{i})}\right)$ on $\radice{n}{X}$, with $0\leq d_{i}<q$.
\end{lemma}

\begin{proof}
Let us compute $p_{m}(G)$, using the previous lemma and the projection formula for the morphism $\pi$ (Proposition \ref{proj.formula.finite}):
$$
\begin{array}{rcl}
p_{m}(G) &= &p(p'_{*}(G\otimes \mc{E}_{m}^{\vee}))\\
& = & p(p_{*}\pi_{*}(G\otimes \pi^{*}\mc{E}_{n}^{\vee}\otimes \mc{E}_{n,m}^{\vee}\otimes M^{\vee}))\\
& = & p(p_{*}(\pi_{*}(G\otimes \mc{E}_{n,m}^{\vee}\otimes M^{\vee})\otimes \mc{E}_{n}^{\vee}))\\
& = & p_{n}(\pi_{*}(G\otimes \mc{E}_{n,m}^{\vee}\otimes M^{\vee}))
\end{array}
$$
(where $p$ denotes the reduced Hilbert polynomial on $X_{1}$) and since we have $\mc{E}^{\vee}_{n,m}\otimes M^{\vee}=\bigoplus_{0\leq d_{i} <q} L_{m}\left(\sum_{i} d_{i}\frac{p_{i}}{m}\right)=\bigoplus_{0\leq d_{i}<q} L_{m}^{(d_{i})}$, the last expression is equal to
$$
p_{n}(\bigoplus_{0\leq d_{i}<q} \pi_{*}(G\otimes L_{m}^{(d_{i})}))
$$
and this is a weighted mean of the polynomials $p_{n}(\pi_{*}(G\otimes L_{m}^{(d_{i})} ))$, as claimed.

Note that if for some $(d_{i})$ the sheaf $\pi_{*}(G\otimes L_{m}^{(d_{i})})$ is zero, then the corresponding Hilbert polynomial will not contribute to the reduced Hilbert polynomial of $G$ (this accounts for the  ``non-zero'' part of the statement).
\end{proof}

\begin{rmk}\label{description.components}
Let us describe the sheaf $G^{(d_{i})}=\pi_{*}(G\otimes L_{m}^{(d_{i})})$ in a more concrete way as a parabolic sheaf on $X_{1}$. This will be important for the proof of the next results.

Let us take $(e_{i}) \in \bb{Z}^{r}$ with $0\leq e_{i}<n$, and let us calculate the component $(G^{(d_{i})})_{(e_{i})} \in \Coh(X_{1})$.

We have
$$
\begin{array}{rcl}
(G^{(d_{i})})_{(e_{i})}& = &p_{*}(\pi_{*}(G\otimes L_{m}^{(d_{i})})\otimes L_{n}^{(e_{i})})\\
& = & p_{*}(\pi_{*}(G\otimes L_{m}^{(d_{i})}\otimes \pi^{*}L_{n}^{(e_{i})})) \\
& = & p_{*}(\pi_{*}(G\otimes L_{m}^{(d_{i})}\otimes L_{m}^{(q e_{i})}))\\
& = & p_{*}\pi_{*}(G\otimes L_{m}(\sum_{i} (d_{i}+q e_{i})\frac{p_{i}}{m}))\\
& = & p'_{*}(G\otimes L_{m}(\sum_{i} (d_{i}+q e_{i})\frac{p_{i}}{m}))\\
& = & G_{(d_{i}+qe_{i})}
\end{array}
$$
so the pieces of the sheaf $G^{(d_{i})}$ are naturally identified with pieces of $G$. Moreover the maps between the pieces of $G^{(d_{i})}$ can be identified with compositions of the ones coming from $G$, in the same way.

\end{rmk}

\begin{prop}\label{stability.preserved}
With the notation of the previous lemmas, let $F \in \Coh(\radice{n}{X})$ be a coherent sheaf on $\radice{n}{X}$. Then $p_{m}(\pi^{*}F)=p_{n}(F)$, and if $F$ is semi-stable, then $\pi^{*}F$ is semi-stable as well.
\end{prop}

For the proof, we will need the following lemma.

\begin{lemma}
Let $Y$ be a log stack with a free global chart $L\colon \bb{N}^{r}\to \Div(Y)$, and consider the root stack $\pi\colon \radice{n}{Y}\to Y$, with $L_{n}\colon \frac{1}{n}\bb{N}^{r}\to \Div(\radice{n}{Y})$ the canonical lifting of the log structure of $Y$. Then for any $0\leq d_{i}<n$, we have
$$
\pi_{*} L_{n}^{(d_{i})}  \simeq \mc{O}_{Y}.
$$
\end{lemma}

\begin{proof}
This is a calculation on the universal model for the root stack.

First of all by taking a presentation we can assume that $Y$ is a log scheme with a free global chart. The chart gives a cartesian diagram
$$
\xymatrix{
\radice{n}{Y}\ar[r] \ar[d]_{\pi} & \left[ \bb{A}^{r}/\mu_{n}^{r}\right] \ar[d]^{ \pi '}\\
Y\ar[r] & \bb{A}^{r}
}
$$
where the vertical map $\pi '$ is induced by raising the variables to the $n$-th power. Now $\pi '$ is a coarse moduli space of a tame DM stack and the diagram is cartesian, so we have a base change formula (Proposition 1.5 of \cite{nironi}), and $L_{n}^{(d_{i})}$ is a pullback of the corresponding sheaf on $\left[ \bb{A}^{r}/\mu_{n}^{r} \right]$, so we can reduce to proving the statement in the universal case.

In this case, the invertible sheaf $L_{n}^{(d_{i})}$ over $\left[ \bb{A}^{r}/\mu_{n}^{r} \right]$ corresponds to the module of rank one over $A=k[x_{1},\hdots,x_{r}]$ generated by $x_{1}^{-d_{1}}\cdots x_{r}^{-d_{r}}$. Pushing forward amounts to taking invariants for $\mu_{n}^{r}$, and if $0\leq d_{i}< n$ it is clear that the invariants are $k[x_{1}^{n},\hdots, x_{r}^{n}]$. This shows that $\pi'_{*}L_{n}^{(d_{i})} \simeq \mc{O}_{\bb{A}^{r}}$ in this case, and concludes the proof.
\end{proof}

\begin{proof}[Proof of Proposition \ref{stability.preserved}]
We will apply the last lemma to the morphism $\pi\colon \radice{m}{X}\to \radice{n}{X}$, which is a relative root stack morphism.

First let us prove that $p_{m}(\pi^{*}F)=p_{n}(F)$: by Lemma \ref{weighted}, $p_{m}(\pi^{*}F)$ is a weighted mean of the polynomials $p_{n}(\pi_{*}(\pi^{*}F\otimes L_{m}^{(d_{i})}))$. But in this case by the projection formula for $\pi$ (Proposition \ref{proj.formula.finite}) and the previous lemma we have
$$
p_{n}(\pi_{*}(\pi^{*}F\otimes L_{m}^{(d_{i})}))=p_{n}(F\otimes \pi_{*}L_{m}^{(d_{i})})=p_{n}(F)
$$
so that $p_{m}(\pi^{*}F)=p_{n}(F)$.

Now let us show that if $F$ is semi-stable on $\radice{n}{X}$, then $\pi^{*}F$ is semi-stable on $\radice{m}{X}$. For any subsheaf $G\subseteq \pi^{*}F$, we know that $p_{m}(G)$ is a weighted mean of the non-zero ones among the reduced Hilbert polynomials $p_{n}(\pi_{*}(G\otimes L_{m}^{(d_{i})}))$ for $0\leq d_{i}< q$. Now note that by exactness of $\pi_{*}$ the inclusion $G\otimes L_{m}^{(d_{i})}\subseteq \pi^{*}F \otimes L_{m}^{(d_{i})}$ will induce
$$
\pi_{*}(G\otimes L_{m}^{(d_{i})})  \subseteq  \pi_{*}(\pi^{*}F\otimes L_{m}^{(d_{i})}) \simeq F\otimes \pi_{*}L_{m}^{(d_{i})} \simeq F
$$
and by semi-stability of $F$ we see that if $\pi_{*}(G\otimes L_{m}^{(d_{i})})$ is non-zero, then
$$
p_{n}(\pi_{*}(G\otimes L_{m}^{(d_{i})}))\leq p_{n}(F).
$$
This in turn implies that $p_{m}(G)\leq p_{n}(F)=p_{m}(\pi^{*}F)$, so we conclude that $\pi^{*}F$ is semi-stable on $\radice{m}{X}$.
\end{proof}

\begin{cor}
The pullback functor along $\pi\colon \radice{m}{X}\to \radice{n}{X}$ induces a morphism $\iota_{n,m}\colon \mc{M}^{ss}_{n}\to \mc{M}^{ss}_{m}$ of stacks over $\aff$, and (consequently) a corresponding morphism $i_{n,m}\colon M^{ss}_{n}\to M^{ss}_{m}$ between the good moduli spaces.\qed
\end{cor}

\begin{rmk}
Proposition \ref{stability.preserved} shows that the reduced Hilbert polynomial (unlike the non-reduced one) is preserved by pullback, so that the morphism $\iota_{n,m}$ restricts to a map $\mc{M}^{ss}_{h,n}\to \mc{M}^{ss}_{h,m}$ between the substacks of families of parabolic sheaves with fixed reduced Hilbert polynomial $h \in \bb{Q}[x]$.
\end{rmk}

\begin{prop}\label{semistable.open}
The morphism $\iota_{n,m}$ is an open immersion.
\end{prop}

\begin{proof}
Let us consider a morphism $f\colon S\to \mc{M}^{ss}_{m}$ from a scheme, and the cartesian diagram
$$
\xymatrix{
\m{S}\ar[r]\ar[d] & S\ar[d]^{f} \\
\mc{M}^{ss}_{n}\ar[r]^{\iota_{n,m}} & \mc{M}^{ss}_{m}.
}
$$
The morphism $f$ corresponds to a family $F \in \Coh(\radice{m}{X}\times_{k}S)$ of semi-stable sheaves on $\radice{m}{X}$, and by construction $\m{S}(T)$ is the category of triples $(\phi\colon T\to S, G, \beta)$ with $G \in \Coh(\radice{n}{X}\times_{k}T)$ a family of semi-stable sheaves and $\beta\colon \iota_{n,m}(G) \simeq \phi^{*}F$ as coherent sheaves on $\radice{m}{X}\times_{k}T$. Note that by adjunction we have a map $\alpha\colon \pi_{S}^{*}{\pi_{S}}_{*}F\to F$ of sheaves on $\radice{m}{X}\times_{k}S$.

Consider the locus $S^{0}\subseteq S$ of points where $\alpha$ is an isomorphism. One can easily check that this is an open subscheme of $S$, and it represents the fibered product $\m{S}$.
\end{proof}

Now we turn our attention to the behavior of stable sheaves.

\begin{prop}\label{stable.restriction}
Assume that the pullback along $\pi$ of any stable sheaf is still stable. Then $\iota_{n,m}$ restricts to an open immersion $\iota^{o}_{n,m}\colon \mc{M}^{s}_{n}\to \mc{M}^{s}_{m}$, inducing $i^{o}_{n,m}\colon M^{s}_{n}\to M^{s}_{m}$ (which coincides with the restriction of $i_{n,m}$).
\end{prop}

We will need the following lemma.

\begin{lemma}\label{stable.cartesian}
The square
$$
\xymatrix{
\mc{M}^{s}_{n}\ar[r]^{\iota^{o}_{n,m}}\ar[d]& \mc{M}^{s}_{m}\ar[d]\\
\mc{M}^{ss}_{n}\ar[r]^{\iota_{n,m}} & \mc{M}^{ss}_{m}
}
$$
is cartesian.
\end{lemma}

\begin{proof}
This follows easily from the fact that if $G \in \Coh(\radice{n}{X})$ is a sheaf such that $\pi^{*}G \in \Coh(\radice{m}{X})$ is stable, then $G$ is stable on $\radice{n}{X}$, which is an immediate consequence of the fact that pullback along $\radice{m}{X}\to \radice{n}{X}$ is flat and preserves the reduced Hilbert polynomial.
\end{proof}

\begin{rmk}
This lemma holds also if we are considering variants with fixed reduced Hilbert polynomial $h \in \bb{Q}[x]$ (or with some other fixed datum, compatible with pullback).
\end{rmk}

\begin{proof}[Proof of Proposition \ref{stable.restriction}]
The fact that pullback preserves stability implies that $\iota_{n,m}$ maps $\mc{M}^{s}_{n}$ to $\mc{M}^{s}_{m}$, so the map $\iota^{o}_{n,m}$ is well-defined. Lemma \ref{stable.cartesian} implies that $\iota^{o}_{n,m}$ is an open immersion, since we know that $\iota_{n,m}$ is an open immersion, and the statement for $i^{o}_{n,m}$ follows from the properties of good moduli spaces.
\end{proof}

The morphism $\iota_{n,m}\colon \mc{M}^{ss}_{n}\to \mc{M}^{ss}_{m}$ is not always closed.

\begin{example}\label{map.not.closed}
Consider the case of the standard log point, i.e. $X=\spec(k)$ with the log structure $L\colon \bb{N}\to \Div(\spec(k))$, sending $0$ to $(k,1)$ and everything else to $(k,0)$. Consider the projection $\pi \colon X_{2}\to X$, and the family of parabolic sheaves $\{E_{t}\}_{t \in k}$ with weights in $\frac{1}{2}\bb{N}$, over $\bb{A}^{1}_{k}$, given by $(E_{t})_{a}=k$ for any $a \in \frac{1}{2}\bb{Z}$, and maps
$$
(E_{t})_{-1}\to (E_{t})_{-1/2}\to (E_{t})_{0}=k\stackrel{\cdot t}{\to}k \stackrel{0}{\to}k
$$

This is a flat family of semi-stable sheaves over $\bb{A}^{1}_{k}$, i.e. an object of $\mc{M}^{ss}_{2}(\bb{A}^{1}_{k})$. Notice also that the diagram
$$
\xymatrix{
k \ar[r]^{\cdot t}\ar@{=}[d] & k \ar[r]^{0} & k\ar@{=}[d]\\
k \ar@{=}[r] & k \ar[r]^{0}\ar[u]^{\cdot t} & k
}
$$
shows that for $t\neq 0$, $E_{t}$ is isomorphic to the pullback  of the unique invertible sheaf on $\spec(k)$, but when $t=0$ this is clearly not true.

This essentially shows that the following diagram is cartesian
$$
\xymatrix{
\bb{A}^{1}_{k}\setminus \{0\} \ar[r]\ar[d] & \bb{A}^{1}_{k}\ar[d]^{E_{t}}\\
\mc{M}^{ss}_{1}\ar[r]^{\iota_{1,2}}& \mc{M}^{ss}_{2}
}
$$
and this implies that $\iota_{1,2}$ is not closed in this case.
\end{example}

In this example the pullback of a stable sheaf need not be stable in general. Let us examine directly a larger class of examples where this happens.

\begin{example}\label{stable.not.preserved}
Assume that $X$ is a log scheme with a chart $L\colon \bb{N}\to \Div(X)$ such that $L(1)=(L_{1},0) \in \Div(X)$, and let $F$ be a stable sheaf on $X_{1}$. Then for every $0\leq i <n$ we have the stable parabolic sheaf
$$
\xymatrix{
 & -1 & & \cdots & & -\frac{i}{n} & & \cdots & & 0 \\
F_{i}=  & 0\ar[r]&0\ar[r] & \cdots \ar[r]&0\ar[r] & F \ar[r] &0\ar[r]& \cdots \ar[r]&0\ar[r] & 0}
$$
on $\radice{n}{X}$, with one copy of $F$ in place $-\frac{i}{n}$, and $0$ everywhere else (of course when $i=0$ we will also have $F\otimes L_{1}^{-1}$ in place $-1$).

Given $m=nq$, the pullback along $\pi\colon \radice{m}{X}\to \radice{n}{X}$ of $F_{i}$ is given by
$$
\xymatrix@C=23pt{
 & -1 & \cdots & -\frac{qi}{m}  & & \cdots  & & -\frac{qi-q+1}{m} & \cdots & 0\\
\pi^{*}F_{i}=  & 0\ar[r] & \cdots \ar[r] & F\ar@{=}[r] & F \ar@{=}[r]& \cdots \ar@{=}[r] & F\ar@{=}[r]& F\ar[r] & \cdots \ar[r] & 0.
}
$$
The sheaves $\pi^{*}F_{i}$ are semi-stable, but not stable anymore: for example we have a subsheaf $G\subseteq \pi^{*}F$ given by
$$
\xymatrix@C=16pt{
 & -1 & \cdots & -\frac{qi}{m} & \cdots & -\frac{qi-q+2}{m} & -\frac{qi-q+1}{m} & -\frac{q(i-1)}{m} & \cdots & 0\\
G= & 0 \ar[r] & \cdots \ar[r] & 0 \ar[r] & \cdots \ar[r] & 0 \ar[r] & F\ar[r] & 0 \ar[r] & \cdots \ar[r] & 0,
}
$$
and clearly $p_{m}(G)=p_{m}(\pi^{*}F_{i})=p(F)$.

Moreover it is easy to describe the stable factors of $\pi^{*}F_{i}$: they are precisely the sheaves $G_{j}$ with one copy of $F$ in place $-\frac{qi-j}{m}$ for $0\leq j \leq q-1$, and all zeros otherwise. Note that all of the sheaves on $\radice{n}{X}$ described in Remark \ref{description.components} (obtained starting from one of the stable factors $G_{j}$) coincide with the original $F_{i}$ or are zero, in this case.

Note also that $\pi^{*}F_{i}$ is not even polystable: of the sheaves $G_{j}$ just described, only $G=G_{q-1}$ is a subsheaf of $\pi^{*}F_{i}$. The polystable sheaf $\bigoplus G_{j}$ which is S-equivalent to $\pi^{*}F_{i}$ is
$$
\xymatrix@C=23pt{
 & -1 & \cdots & -\frac{qi}{m}  & & \cdots  & & -\frac{qi-q+1}{m} & \cdots & 0\\
\bigoplus G_{j}=  & 0\ar[r] & \cdots \ar[r] & F\ar[r]^{0} & F \ar[r]^{0}& \cdots \ar[r]^{0} & F\ar[r]^{0}& F\ar[r] & \cdots \ar[r] & 0
}
$$
where all the maps are zero.
\end{example}

This example can be generalized in arbitrary rank. For example if $X$ has a chart $L\colon \bb{N}^{2}\to \Div(X)$ with both $(1,0)$ and $(0,1)$ going to $(L_{1,0}, 0)$ and $(L_{0,1},0)$, this example carries through verbatim (so that again there will be stable sheaves that become strictly semi-stable after pullback), but we can also do something different.

Let us introduce some notation first.

\begin{notation}\label{notation}
From now on we need to draw parabolic sheaves on $\radice{n}{X}$, where $X$ is a log stack with a free log structure $\bb{N}^{r}\to \Div(X)$.

When $r=1$, we can draw parabolic sheaves easily as the segment in $[-1,0]$
$$
\xymatrix{
 \cdots&-1 & -\frac{n-1}{n} & \cdots & -\frac{1}{n} & 0 &\cdots\\
 \cdots  \ar[r]& F_{0}\otimes L^\vee \ar[r]&F_{n-1}\ar[r] & \cdots \ar[r] & F_{1}\ar[r] &  F_{0} \ar[r]  &\cdots
}
$$
of a ``sequence'' of sheaves arranged on the real line. If $r=2$, as we already did several times, we can draw the sheaf as the square $[-1,0]^{2}$ inside a ``diagram'' on the plane with a sheaf on every point with integral coordinates and maps going up and to the right.

If the rank is bigger this becomes less feasible, but we have an ``inductive'' way of drawing parabolic sheaves in higher rank. For example, a parabolic sheaf with $r=2$ can be drawn in the following way: say that the DF structure is given by $L\colon \bb{N}^{2}\to \Div(X)$, and consider the new DF structure given by the composition $\bb{N}\subseteq \bb{N}^{2}\to \Div(X)$, where $\bb{N}\subseteq \bb{N}^{2}$ is the inclusion of the first or second component. Call the resulting log schemes $\widetilde{X}_{1}$ and $\widetilde{X}_{2}$ respectively.

Then a parabolic sheaf on the root stack $\radice{2}{X}$ can be drawn as a diagram 
$$
\xymatrix{
\cdots&-1 &  -\frac{1}{2} & 0 &\cdots\\
 \cdots \ar[r] &F_{0}\otimes L_{(1,0)}^\vee \ar[r]&F_{1} \ar[r] &  F_{0}\ar[r]& \cdots
}
$$
with the formal properties of a parabolic sheaf on $\widetilde{X}_{1}$, but where the sheaves $F_{0}$ and $F_{1}$ are parabolic sheaves on $\widetilde{X}_{2}$. In other words we are ``collapsing'' the vertical direction, and the price is to use parabolic sheaves in place of quasi-coherent sheaves (the tensor product with a line bundle has an obvious definition).

In general if $X$ is a log scheme with a chart $L\colon \bb{N}^{r}\to \Div(X)$, let us consider the DF structure given by the inclusion $\bb{N}^{r-1}\subseteq \bb{N}^{r}\to \Div(X)$ that omits the $i$-th standard generator $e_{i}$. Call $\widetilde{X}$ the resulting log scheme. Then a parabolic sheaf on $\radice{n}{X}$ can be seen as a diagram
$$
\xymatrix@C=20pt{
-1 & & \cdots & & -\frac{j}{n} & & \cdots & & 0 &\\
 & & & & & & & & &\ar[d]_{\mbox{other directions}}\ar[lllllllll]_{i\mbox{-th direction}}\circ\\
 F_{0}\otimes L_{e_{i}}^\vee \ar[r]&F_{n-1}\ar[r] & \cdots \ar[r]&F_{j+1}\ar[r] & F_{j} \ar[r] &F_{j-1}\ar[r]& \cdots \ar[r]&F_{1}\ar[r] & F_{0} &}
$$
where each of the $F_{i}$'s is a parabolic sheaf on the root stack $\radice{n}{\widetilde{X}}$.

We will use this notation several times in the following arguments.
\end{notation}

Let us describe another type of situation in which the pullback of a stable sheaf is not stable.

\begin{example}
Take a log scheme $X$ with a global chart $L\colon \bb{N}^{2}\to \Div(X)$, and now assume only that $L((0,1))=(L_{0,1},0) \in \Div(X)$, and the section of $L_{1,0}$ can be non-zero, and let us fix $n=2$ for simplicity. Consider as in the preceding discussion the log scheme $\widetilde{X}$, which has the same underlying scheme as $X$, but the log structure is given by $\bb{N}\subseteq \bb{N}^{2}\to \Div(X)$, where the map is the immersion as the first component, so the image of $1$ is $(L_{1,0},s_{1,0})$.

Take a stable sheaf $F \in \Coh(\radice{2}{\widetilde{X}})$, say

$$
\xymatrix{
 & -1 & -1/2 & 0  \\
F = & F_{0}\otimes L_{1,0}^\vee \ar[r] & F_{1}\ar[r] & F_{0}
}
$$
and form the following parabolic sheaf, call it $F'$, on $\radice{2}{X}$

$$
\xymatrix{
 & -1 & -\frac{1}{2} & 0 &   \\
&0\ar[r] &0\ar[r] &0 & 0 \\
F' = & F_{0}\otimes L_{1,0}^\vee \ar[r]\ar[u] & F_{1}\ar[r]\ar[u] & F_{0}\ar[u] & -\frac{1}{2}\\
 & 0\ar[u]\ar[r]&0\ar[r]\ar[u] &0\ar[u] & -1
}
$$
or, with the notation of \ref{notation},
$$
\xymatrix@C=60pt{
& -1 & -\frac{1}{2} &  0 &\\
 &  & & &\ar[d]_{\mbox{horizontal direction}}\ar[lll]_{\mbox{vertical direction}}\circ\\
F' =  & 0 \ar[r] & F \ar[r] & 0. &}
$$
Note that this is well defined because the given section of $L_{0,1}$ is zero. It is clear that $F'$ is also stable on $\radice{2}{X}$, since its subsheaves correspond exactly to subsheaves of $F$ on $\radice{2}{\widetilde{X}}$, and the slopes are the same. Assume also that $\widetilde{\pi}^{*}F$ is stable on $\radice{4}{\widetilde{X}}$, where $\widetilde{\pi}\colon \radice{4}{\widetilde{X}}\to \radice{2}{\widetilde{X}}$ is the projection.

Now consider the pullback of $F'$ along $\pi\colon\radice{4}{X}\to \radice{2}{X}$
$$
\xymatrix{
& -1 & -\frac{3}{4} &-\frac{1}{2} & -\frac{1}{4} &  0 &\\
 &  & & &  & &\ar[d]_{\mbox{horizontal direction}}\ar[lllll]_{\mbox{vertical direction}}\circ\\
\pi^*F' =  & 0 \ar[r] & 0\ar[r] & \widetilde{\pi}^{*}F\ar@{=}[r] & \widetilde{\pi}^{*}F \ar[r] & 0 &
}
$$
and notice that this is not stable, because the following sheaf $G \in \Coh(\radice{4}{X})$
$$
\xymatrix{
& -1 & -\frac{3}{4} &-\frac{1}{2} & -\frac{1}{4} &  0 &\\
 &  & & &  & &\ar[d]_{\mbox{horizontal direction}}\ar[lllll]_{\mbox{vertical direction}}\circ\\
G =  & 0 \ar[r] & 0\ar[r] & 0\ar[r] & \widetilde{\pi}^{*}F \ar[r] & 0 &
}
$$
is a subsheaf of $\pi^{*}F'$, and clearly has $p_{4}(G)=p_{4}(\pi^{*}F')=p_{2}(F')$.

Once again  we can easily describe the stable factors: they are the sheaf $G$, and the analogous one with the rows corresponding to $-\frac{1}{4}$ and $-\frac{1}{2}$ switched (note that this is not a subsheaf of $\pi^{*}F'$), so that the polystable sheaf S-equivalent to $\pi^{*}F'$ is
$$
\xymatrix{
-1 & -\frac{3}{4} &-\frac{1}{2} & -\frac{1}{4} &  0 &\\
  & & &  & &\ar[d]_{\mbox{horizontal direction}}\ar[lllll]_{\mbox{vertical direction}}\circ\\
 0 \ar[r] & 0\ar[r] & \widetilde{\pi}^{*}F\ar[r]^{0} & \widetilde{\pi}^{*}F \ar[r] & 0. &
}
$$

As in the previous example, we can completely reconstruct $F'$ from any of the stable factors of $\pi^*F'$, for example from $G$, as $\pi_{*}(G\otimes L_{4}^{(d_{i})})$ for some $(d_{i})$ (for example $(1,1)$ works). Finally, one checks that $\pi_{*}(G\otimes L_{4}^{(d_{i})})$ is either isomorphic to $F'$, or is zero.
\end{example}

We will see now that the behavior in the previous example is in fact typical for stable sheaves with non-stable pullback.

\begin{notation}
We denote by $X^{i}$ (for $i=1,\dots,r$) the log stack given by $X_{1}$, together with the log structure induced by the composition $\bb{N}^{r-1}\subseteq \bb{N}^{r}\to \Div(X_{1})$, where $\bb{N}^{r-1}\subseteq \bb{N}^{r}$ is the inclusion that omits the $i$-th basis element.
\end{notation}

Let $\Coh(\radice{n}{X^{i}})_{s_{i}}$ denote the subcategory of $\Coh(\radice{n}{X^{i}})$ of sheaves annihilated by the section $s_{i}$ of $L_{i}$ coming from the log structure (meaning that every component of the parabolic sheaf is annihilated by $s_{i}$). We define fully faithful functors $I^{i}_{n,j}\colon \Coh(\radice{n}{X^{i}})_{s_{i}}\to \Coh(\radice{n}{X})$ for $i=1,\dots, r$, and $j=1,\hdots, n$ as follows: for $F \in \Coh(\radice{n}{X^{i}})_{s_{i}}$, we set (as a parabolic sheaf)
$$
I^{i}_{n,j}(F)_{a_{1},\dots,a_{r}}=\left\{
	\begin{array}{ll}
		F_{a_{1},\dots,\widehat{a_{i}},\dots,a_{r}}  & \mbox{for }  a_{i} = -\frac{j}{n} \\
		0 & \mbox{otherwise.}
	\end{array}
\right.
$$
with maps
$$
I^{i}_{n,j}(F)_{a_{1},\dots,a_{r}}\to I^{i}_{n,j}(F)_{a_{1},\dots, a_{k}+\frac{1}{n},\dots,a_{r}}
$$
defined to be zero, except if $a_{i}=-\frac{j}{n}$ and $k\neq i$, in which case it is defined as the corresponding map $$F_{a_{1}, \dots, \widehat{a_{i}},\dots,a_{r}}\to F_{a_{1},\dots, a_{k}+\frac{1}{n},\dots, \widehat{a_{i}},\dots, a_{r}}$$ of the sheaf $F$.

In other words, $I^{i}_{n,j}(F)$ is obtained by placing $F$ in the ``slice'' $a_{i}=-\frac{j}{n}$, and filling the rest with zeros. Note that this is well defined only if the components of $F$ are annihilated by $s_{i}$.

If we use the notation of \ref{notation}, we can draw $I^{i}_{n,j}(F)$ as
$$
\xymatrix@C=20pt{
& -1 & & \cdots & & -\frac{j}{n} & & \cdots & & 0 &\\
 & & & & & & & & & &\ar[d]_{\mbox{other directions}}\ar[lllllllll]_{i\mbox{-th direction}}\circ\\
I^{i}_{n,j}(F) =  & 0\ar[r]&0\ar[r] & \cdots \ar[r]&0\ar[r] & F \ar[r] &0\ar[r]& \cdots \ar[r]&0\ar[r] & 0 &}
$$
and from this description, it is clear that:
\begin{itemize}
\item we have $p_{n}(I_{n,j}^{i}(F))=p_{n}(F)$,
\item subsheaves of $I_{n,j}^{i}(F)$ correspond bijectively to subsheaves of $F$ via $I_{n,j}^{i}$,
\item so $I_{n,j}^{i}(F)$ is (semi-)stable on $\radice{n}{X}$ if and only if $F$ is (semi-)stable on $\radice{n}{X^{i}}$.
\end{itemize}

Now let us set $m=nq$ and assume that $F$ is stable, so that $I_{n,j}^{i}(F)$ is also stable. Consider the pullback of $I_{n,j}^{i}(F)$ along $\pi\colon \radice{m}{X}\to \radice{n}{X}$, and consider also the projection $\pi_{i}\colon \radice{m}{X^{i}}\to \radice{n}{X^{i}}$. We can write the pullback as
$$
\xymatrix@C=15pt{
 & -1 & \cdots & -\frac{qj}{m}  & & \cdots  & & -\frac{qj-q+1}{m} & \cdots & 0 &\\
 & & & & & & & & & &\circ \ar[d]_{\mbox{other directions}}\ar[lllllllll]_{i\mbox{-th direction}}\\
\pi^{*}I_{n,j}^{i}(F)=  & 0\ar[r] & \cdots \ar[r] & \pi_{i}^{*}F\ar@{=}[r] & \pi_{i}^{*}F \ar@{=}[r]& \cdots \ar@{=}[r] & \pi_{i}^{*}F\ar@{=}[r]& \pi_{i}^{*}F\ar[r] & \cdots \ar[r] & 0, & 
}
$$
and, as in example \ref{stable.not.preserved}, we see that $\pi^{*}I_{n,j}^{i}(F)$ is not stable: the sheaf $I_{m,qj-q+1}^{i}(\pi_{i}^{*}F) \in \Coh(\radice{m}{X})$ having one copy of $\pi_{i}^{*}F$ in the ``slice'' $a_{i}=-\frac{qj-q+1}{m}$
$$
\xymatrix@C=15pt{
 & -1 & \cdots & -\frac{qj}{m}  &  \cdots  &  & -\frac{qj-q+1}{m} & \cdots & 0 &\\
 & & & & & & & & &\circ \ar[d]_<<<<<<{\mbox{other directions}}\ar[llllllll]_{i\mbox{-th direction}}\\
I_{m,kj-k+1}^{i}(\pi_{i}^{*}F)= & 0\ar[r] & \cdots \ar[r] &0\ar[r] & \cdots \ar[r] &  0\ar[r]  & \pi_{i}^{*}F\ar[r] & \cdots \ar[r] & 0& 
}
$$
is a proper subsheaf of the pullback $\pi^{*}I_{n,j}^{i}(F)$ and has $p_{m}(I_{m,qj-q+1}^{i}(\pi_{i}^{*}F))=p_{m}(\pi^{*}I_{n,j}^{i}(F))$, as they are both equal to $p_{n}(F)$ on $(X^{i})_{n}$.

We can describe the stable factors of $\pi^{*}I_{n,j}^{i}(F)$ if $\pi_{i}^{*}F$ if stable on $\radice{m}{X^{i}}$ (which is not always the case): the quotient $\pi^{*}I_{n,j}^{i}(F)/I_{m,qj-q+1}^{i}(\pi_{i}^{*}F)$ is the sheaf
$$
\xymatrix@C=15pt{
  -1 & \cdots & -\frac{qj}{m}  & & \cdots  & & -\frac{qj-q+1}{m} & \cdots & 0 &\\
  & & & & & & & & &\circ \ar[d]_{\mbox{other directions}}\ar[lllllllll]_{i\mbox{-th direction}}\\
 0\ar[r] & \cdots \ar[r] & \pi_{i}^{*}F\ar@{=}[r] & \pi_{i}^{*}F \ar@{=}[r]& \cdots \ar@{=}[r] & \pi_{i}^{*}F\ar[r]& 0 \ar[r] & \cdots \ar[r] & 0 & 
}
$$
with one less copy of $\pi_{i}^{*}F$ at the end, and it has $I_{m,qj-q+2}^{i}(\pi_{i}^{*}F)$ as a subsheaf with the same slope. Inductively, we see that the stable factors of $\pi^{*}I_{n,j}^{i}(F)$ are the sheaves $I_{m,qj-h}^{i}(\pi_{i}^{*}F)$ for $h=0,\hdots,q-1$, and the semi-stable sheaf S-equivalent to $\pi^{*}I_{n,j}^{i}(F)$ is
$$
\xymatrix@C=20pt{
 -1 & \cdots & -\frac{qj}{m}  & & \cdots  & & -\frac{qj-q+1}{m} & \cdots & 0 &\\
 & & & & & & & & &\circ \ar[d]_{\mbox{other directions}}\ar[lllllllll]_{i\mbox{-th direction}}\\
 0\ar[r] & \cdots \ar[r] & \pi_{i}^{*}F\ar[r]^{0} & \pi_{i}^{*}F \ar[r]^{0}& \cdots \ar[r]^{0} & \pi_{i}^{*}F\ar[r]^{0}& \pi_{i}^{*}F\ar[r] & \cdots \ar[r] & 0 & 
}
$$
with zeros instead of identity maps.

The next proposition says that every stable sheaf $F \in \Coh(\radice{n}{X})$ such that $\pi^{*}F \in \Coh(\radice{m}{X})$ is not stable is of this form.

\begin{prop}\label{pullback.not.stable}
Let $F \in \Coh(\radice{n}{X})$ be a stable sheaf. Then $\pi^{*}F \in \Coh(\radice{m}{X})$ is not stable if and only if $F$ is in the image of one of the functors $I_{n,j}^{i}$, for some $i,j$.
\end{prop}

\begin{proof}
The ``if'' part is contained in the previous discussion.

For the other direction, let us consider a subsheaf $G\subseteq \pi^{*}F$, along with the subsheaves $\pi_{*}(G\otimes L_{m}^{(d_{i})})\subseteq F$ for $0\leq d_{i}<q$ (where $m=nq$). Recall that by proposition \ref{weighted}, the slope $p_{m}(G)$ is a weighted mean of the polynomials $p_{n}(\pi_{*}(G\otimes L_{m}^{(d_{i})}))$, with $\pi_{*}(G\otimes L_{m}^{(d_{i})})$ non-zero.

The only possibility for $G$ to be destabilizing is that $p_{n}(\pi_{*}(G\otimes L_{m}^{(d_{i})}))=p_{n}(F)$ for all non-zero $\pi_{*}(G\otimes L_{m}^{(d_{i})})$, and by stability of $F$ this implies $\pi_{*}(G\otimes L_{m}^{(d_{i})})=F$ for those values of $(d_{i})$.

Note also that if $\pi_{*}(G\otimes L_{m}^{(d_{i})})=F$ for all values of $(d_{i})$, then we must have $G=\pi^{*}F$: this can be seen directly from the description of the sheaves $\pi_{*}(G\otimes L_{m}^{(d_{i})})$ given in Remark \ref{description.components}, or from the fact that the sheaf $\mc{E}=(\bigoplus_{0\leq d_{i}<q}L_{m}^{(d_{i})})^{\vee}$ is a generating sheaf for the relative root stack $\pi\colon \radice{m}{X}\to \radice{n}{X}$, and the cokernel of $G\subseteq \pi^{*}F$ would be sent to zero by $\pi_{*}(-\otimes \mc{E}^{\vee})$ (cfr. \cite[Lemma $3.4$]{nironi}).

This implies that if $\pi^{*}F$ is not stable, then there is a subsheaf $G$ with
\begin{itemize}
\item $\pi_{*}(G\otimes L_{m}^{(d_{i})})=0$ or $\pi_{*}(G\otimes L_{m}^{(d_{i})})=F$ for all $0\leq d_{i}<q$, and 
\item each of the two cases occur for at least one $(d_{i})$.
\end{itemize}
Now we will see that this implies that $F$ is in the image of one of the functors $I_{n,j}^{i}$. From now on for brevity we will write $G^{(d_{i})}=\pi_{*}(G\otimes L_{m}^{(d_{i})}) \in \Coh(\radice{n}{X})$.

Observe first that if $G^{(d_{i})}=F$ for some $(d_{i})$, then $G^{(e_{i})}=F$ also for any $(e_{i})\geq (d_{i})$ in the componentwise order. This is because, since $G$ is a subsheaf of $\pi^{*}F$, the diagram
$$
\xymatrix{
F=(\pi^{*}F)^{(d_{i})}\ar@{=}[r] & (\pi^{*}F)^{(e_{i})}=F\\
F=G^{(d_{i})}\ar[r]\ar@{^{(}->}[u] & G^{(e_{i})}\ar@{^{(}->}[u]
}
$$
commutes, and this forces $G^{(e_{i})}=F$. This implies that if $G^{(d_{i})}=0$, then $G^{(e_{i})}=0$ for any $(e_{i})\leq (d_{i})$. In particular we necessarily have $G^{(0,\dots,0)}=0$ and $G^{(q-1,\dots,q-1)}=F$. Now we justify the fact that there is a direction $i_{0}\in \{1,\dots,r\}$ such that $G^{(d_{i})}=0$ and $F$ both occur for $d_{i}$ with $i\neq i_{0}$ fixed, and $d_{i_{0}}$ ranging from $0$ to $q-1$.

Look first at the sheaves $G^{(a,0,\dots,0)}$ for $0\leq a <q$: if $G^{(q-1,0\dots,0)}=F$, we are done. Otherwise, all the sheaves of this form are $0$, and we look at $G^{(q-1,a,0,\dots,0)}$ for $0\leq a <q$, and so on. If we are unlucky, at the $(r-1)$-th step we will find $G^{(q-1,\dots,q-1,0)}=0$, and so the sheaves $G^{(q-1,\dots,q-1,a)}$ satisfy the requirement, since for $a=0$ we have $0$ and for $a=q-1$ we have $F$.

Now we claim that all the maps of the parabolic sheaf $F$ in the direction $i_{0}$ are necessarily zero: in fact take $(a_{1},\dots,a_{r}) \in ([-1,0) \cap \frac{1}{n}\bb{Z})^{r}$, and consider the map
$$f \colon F_{(a_{1},\dots,a_{r})}\to F_{(a_{1},\dots, a_{i_{0}}+\frac{1}{n},\dots, a_{r})}.$$
By looking at the hypercubes corresponding to these two sheaves in the pullback $\pi^{*}F$, along with the subsheaf $G$ and its property that in the direction $i_{0}$ it has both $F$ and $0$, we see that the following diagram commutes (the top row is in $\pi^{*}F$, the bottom in the subsheaf $G$)
$$
\xymatrix{
F_{(a_{1},\dots,a_{r})} \ar[r]^<<<<<{f} & F_{(a_{1},\dots, a_{i_{0}}+\frac{1}{n},\dots, a_{r})}\ar@{=}[r]&F_{(a_{1},\dots, a_{i_{0}}+\frac{1}{n},\dots, a_{r})} \\
F_{(a_{1},\dots,a_{r})}\ar[r]\ar@{=}[u] & 0\ar[r]\ar@{^{(}->}[u] & F_{(a_{1},\dots, a_{i_{0}}+\frac{1}{n},\dots, a_{r})}\ar@{=}[u] ,
}
$$
from which we deduce that $f$ is zero.

The above discussion implies that, using the notation of \ref{notation}, the sheaf $F$ can be written as
$$
\xymatrix{
& -1 & -\frac{n-1}{n} & \cdots & -\frac{1}{n} & 0 &\\
 & & & & & &\ar[d]_{\mbox{other directions}}\ar[lllll]_{i_{0}\mbox{-th direction}}\circ\\
F =  & F_{0}\otimes L_{i_{0}}^\vee \ar[r]^{0}&F_{n-1}\ar[r]^{0} & \cdots \ar[r]^{0} & F_{1}\ar[r]^{0} &  F_{0} &
}
$$
where all the parabolic sheaves $F_{h} \in \Coh(\radice{n}{X^{i_{0}}})$ are annihilated by the section $s_{i_{0}}$, so in fact $F_{h} \in \Coh(\radice{n}{X^{i_{0}}})_{s_{i_{0}}}$. Now note that unless only one of the $F_{h}$'s is non-zero, a sheaf of this form cannot be stable, since it is the direct sum of the parabolic sheaves having $F_{h}$ in place $-\frac{h}{n}$ (in the direction $i_{0}$) and zero everywhere else, and a semi-stable direct sum is stable if and only if there is only one factor, which moreover is stable.

In conclusion $F$ is of the form
$$
\xymatrix@C=20pt{
& -1 & -\frac{n-1}{n} & \cdots &-\frac{h+1}{n} & -\frac{h}{n} &\frac{h-1}{n} &  \cdots & 0 & \\
 & & & & & & & & &\ar[d]_{\mbox{other directions}}\ar[llllllll]_{i_{0}\mbox{-th direction}}\circ\\
F =  & 0\ar[r]&0 \ar[r] & \cdots \ar[r]  &0 \ar[r] & F_{h}\ar[r] &0\ar[r] &\cdots \ar[r] &  0  &}
$$
or, in other words, $F=I_{n,h}^{i_{0}}(F_{h})$ with $F_{h}\in \Coh(\radice{n}{X^{i_{0}}})_{s_{i_{0}}}$ a stable sheaf, and this concludes the proof.
\end{proof}

\begin{lemma}\label{stable.s.equiv}
Let $F \in \Coh(\radice{n}{X})$ be a stable sheaf and let $F' \in \Coh(\radice{m}{X})$ be one of the stable factors of $\pi^{*}F$. Then for any $0\leq d_{i}< q$ the sheaf $\pi_{*}(F'\otimes L_{m}^{(d_{i})}) \in \Coh(\radice{n}{X})$ is isomorphic to $F$ or zero (and both cases occur if $\pi^{*}F$ is not stable).
\end{lemma}

\begin{proof}
If $\pi^{*}F$ is still stable, this is clear from the description of the pullback and by remark \ref{description.components}. In the other case we know that $F$ must be of the form $I_{n,j}^{i}(G)$ with $G$ stable on $\radice{n}{X^{i}}$ from the preceding proposition, and we know the stable factors $F'$ of $\pi^{*}F$ from the discussion preceding the proof of \ref{pullback.not.stable}, if the pullback of $G$ along $\radice{m}{X^{i}}\to \radice{n}{X^{i}}$ is stable. If this pullback is not stable we can apply Proposition \ref{pullback.not.stable} again to $G$, and after a finite number of steps we will get down to a stable sheaf.

From the explicit form of the stable factors and the description of the sheaves $\pi_{*}(F'\otimes L_{m}^{(d_{i})}) \in \Coh(\radice{n}{X})$ of Remark \ref{description.components}, the conclusion follows.\end{proof}

\begin{lemma}
Let $F,G \in \Coh(\radice{n}{X})$ be stable sheaves such that $\pi^{*}F$ and $\pi^{*}G$ are S-equivalent on $\radice{m}{X}$. Then $F \simeq G$ on $\radice{n}{X}$.
\end{lemma}

\begin{proof}
If one of $\pi^{*}F$ or $\pi^{*}G$ is stable, then $\pi^{*}F \simeq \pi^{*}G$ (since S-equivalent implies isomorphic, if one of the sheaves is stable) and since $\pi^{*}$ is fully faithful (\ref{pullback.fully.faithful}) we conclude that $F \simeq G$.

If both $\pi^{*}F$ and $\pi^{*}G$ are not stable, denote by $F_{i}$ and $G_{j}$ their stable factors. Since $\pi^{*}F$ and $\pi^{*}G$ are S-equivalent, they have the same stable factors, so for some $i$ and $j$ we have $F_{i} \simeq G_{j}$. Now by Lemma \ref{stable.s.equiv}, the sheaves $\pi_{*}(F_{i}\otimes L_{m}^{(d_{i})})$ and $\pi_{*}(G_{j}\otimes L_{m}^{(d_{i})})$ for $0\leq d_{i}<q$ are isomorphic to $F$ or $G$ respectively, or zero. Since $F$ and $G$ are not zero, and the isomorphism $F_{i} \simeq G_{j}$ will induce isomorphisms $\pi_{*}(F_{i}\otimes L_{m}^{(d_{i})}) \simeq \pi_{*}(G_{j}\otimes L_{m}^{(d_{i})})$ for any $(d_{i})$, we necessarily get an isomorphism $F \simeq G$.
\end{proof}

\begin{prop}
The morphism $i_{n,m}\colon M^{ss}_{n}\to M^{ss}_{m}$ between the good moduli spaces is geometrically injective. In particular, being proper (since $M^{ss}_n$ and $M^{ss}_m$ are both proper), it is also finite.
\end{prop}

\begin{proof}
Fix an algebraically closed extension $k\subseteq K$, and let us show that $M^{ss}_{n}(K)\to M^{ss}_{m}(K)$ is injective. This means that if $F,G \in \Coh(\radice{n}{X_{K}})$ are semi-stable sheaves such that $\pi^{*}F, \pi^{*}G \in \Coh(\radice{n}{X_{K}})$ are S-equivalent, then $F$ and $G$ are S-equivalent themselves. We can assume that $F$ and $G$ are polystable, and write $F=\bigoplus_{i}F_{i}$ and $G=\bigoplus_{j} G_{j}$ as sums of stable sheaves on $\radice{n}{X_{K}}$.

We will proceed by induction on $N=\max\{\#\mbox{stable factors of }F,\#\mbox{stable factors of }G\}$.

For $N=1$, this is the previous lemma, applied to $X_{K}$.

If $N>1$, write $F_{i,h}$ and $G_{j,k}$ for the stable factors of $\pi^{*}F_{i}$ and $\pi^{*}G_{j}$ respectively. Since
$$\pi^{*}F=\bigoplus_{i} \pi^{*}F_{i}$$ and $$\pi^{*}G=\bigoplus_{j} \pi^{*}G_{j}$$ are S-equivalent, they will have the same stable factors, so for some $i,j,k,h$ we have $F_{i,h} \simeq G_{j,k}$.

Now let us look at the sheaves $\pi_{*}(F_{i,h}\otimes L_{m}^{(d_{i})})$ and $\pi_{*}(G_{j,k}\otimes L_{m}^{(d_{i})})$ on $\radice{n}{X_{K}}$ for all $0\leq d_{i} < q$: by Lemma \ref{stable.s.equiv}, they are isomorphic to $F_{i},$ and $G_{j}$ respectively, or zero. But since neither of $F_{i}$ or $G_{j}$ is zero, as in the proof of the previous lemma, we can conclude that the isomorphism $F_{i,h} \simeq G_{j,k}$ induces $F_{i} \simeq G_{j}$.

After erasing these two factors from $F$ and $G$, we end up with two polystable sheaves $F'$ and $G'$ with $\max\{\#\mbox{stable factors of }F',\#\mbox{stable factors of }G'\}=N-1$ and such that $\pi^{*}F'$ and $\pi^{*}G'$ are S-equivalent. By the induction hypothesis $F'$ and $G'$ are S-equivalent, and this concludes the proof.

The part about finiteness follows from Chevalley's theorem.
\end{proof}

\begin{prop}\label{prop.open.closed}
If the pullback of a stable sheaf is stable, then all maps $\iota_{n,m}, \iota^{o}_{n,m}, i_{n,m}, i^{o}_{n,m}$ are open and closed immersions.
\end{prop}

\begin{proof}
We already know that $\iota_{n,m}$ is an open immersion. The fact that stable sheaves go to stable sheaves implies that polystables go to polystables, and this says that $\iota_{n,m}$ sends closed points to closed points (recall that the closed points of $\mc{M}^{ss}_{n}$ correspond to polystable sheaves). Since we know that the induced map $i_{n,m}$ on the good moduli spaces is finite, by Proposition 6.4 of \cite{alper} we can conclude that $\iota_{n,m}$ is also finite, and in particular closed. This shows that $\iota_{n,m}$ is an open and closed immersion, and this implies the conclusion also for $i_{n,m}$.

Finally, the same conclusion for $\iota^{o}_{n,m}$ and $i^{o}_{n,m}$ holds because of Lemma \ref{stable.cartesian}.
\end{proof}

We have no examples in which $i_{n,m}$ is not an open and closed immersion, so one could conjecture that this is always the case. At the very least, the fact that on each connected component the inductive limit stabilizes (in the sense that $i_{n,m}$ is an isomorphism for $n,m$ divisible enough) seems very reasonable.
 
The following proposition gives sufficient conditions that ensure that stability is preserved under pullback.

\begin{prop}\label{pullback.stable}
The pullback of a stable sheaf is stable in each of the following cases:
\begin{itemize}
\item we are considering torsion-free sheaves and the log structure on $X$ is generically trivial (\ref{gen.trivial});
\item we look at components corresponding to a reduced Hilbert polynomial $h \in \bb{Q}[x]$, which is not the reduced Hilbert polynomial of a stable parabolic sheaf on one of the log stacks $X^{i}$.
\end{itemize}
\end{prop}

\begin{proof}
This is immediate from the previous discussion: a stable sheaf with non-stable pullback will have a lot of zeros, but the maps of a torsion-free parabolic sheaf on a log scheme with generically trivial log structure are injective (see \ref{gen.trivial.inj}), and this is it for the first part.

As for the second part, a stable sheaf with a non-stable pullback is of the form $I_{n,j}^{i}(F)$ for some $F \in \Coh(\radice{n}{X^{i}})_{s_{i}}$, and recall that $p_{n}(I_{n,j}^{i}(F))=p_{n}(F)$.
\end{proof}

\begin{rmk}
Note that the second condition will only be meaningful if the set of reduced Hilbert polynomials of stable parabolic sheaves on $X$ is not entirely contained in the set of reduced Hilbert polynomials of stable sheaves on one of the $X^{i}$.

We feel that this should be the case in general: the reduced Hilbert polynomial of a parabolic sheaf is in particular the reduced Hilbert polynomial of a sheaf on $X$ (the sum of its fundamental pieces), but this sheaf on $X$ is typically not even semi-stable. Moreover, adding generators to the log structure should give more freedom for stable sheaves, and thus for the set of their Hilbert polynomials. For example if the log structure has rank $1$, the pieces of a parabolic stable sheaf need not be stable on $X$.
\end{rmk}

\subsection{Limit moduli theory on the infinite root stack}\label{sec.4.3}

We will keep using the notations introduced in \ref{sec.4.2}, and moreover we will denote by $\pi_{n,m}\colon \radice{m}{X}\to \radice{n}{X}$ the natural projection for $n\mid m$, and by $\pi_{n} \colon \infroot{X}\to \radice{n}{X}$ the projection from the infinite root stack. Note that, being an inverse limit of flat morphisms, $\pi_{n}$ is flat.

Let us describe the moduli theory for finitely presented sheaves on $\infroot{X}$ that we get by taking a limit of the theories at finite levels. In our setting $\infroot{X}$ is actually coherent by \cite[Proposition 4.19]{infinite}, so that finitely presented sheaves are the same as coherent sheaves. Since the most important property for our arguments is finite presentation, we will continue to use this adjective instead of coherence.

Recall from \ref{fp.sheaves} that $\FP(\infroot{X})=\varinjlim_{n}\FP(\radice{n}{X})$, and this means that
\begin{itemize}
\item every finitely presented sheaf $F \in \FP(\infroot{X})$ is of the form $\pi_{n}^{*}F_{n}$ for some $n$ and $F_{n} \in \FP(\radice{n}{X})$,
\item for any $n,m$ and $F_{n}\in \FP(\radice{n}{X})$, $F_{m}\in \FP(\radice{m}{X})$ such that $F \simeq \pi_{n}^{*}F_{n} \simeq \pi_{m}^{*}F_{m}$, there exists $q\geq n,m$ such that $\pi_{n,q}^{*}F_{n} \simeq \pi_{m,q}^{*}F_{m}$ on $\radice{q}{X}$.
\end{itemize}

\begin{definition}
The \emph{reduced Hilbert polynomial} $p(F)$ of $F \in \FP(\infroot{X})$ is the reduced Hilbert polynomial $p_{n}(F_{n})$ of any finitely presented sheaf $F_{n} \in \FP(\radice{n}{X})$ such that $\pi_{n}^{*}F_{n} \simeq F$.
\end{definition}

Since $\pi_{n}^{*}$ is fully faithful and $p_{m}(\pi_{n,m}^{*}(F_{n}))=p_{n}(F_{n})$ by Proposition \ref{stability.preserved}, the reduced Hilbert polynomial of $F$ is well-defined.

\begin{definition}
A finitely presented sheaf $F \in \FP(\infroot{X})$ is \emph{pure} if it comes from a pure sheaf on one of the finite level root stacks $\radice{n}{X}$.

A finitely presented pure sheaf $F \in \FP(\infroot{X})$ is \emph{(semi-)stable} if for any finitely presented subsheaf $G\subseteq F$ we have
$$
p(G)\mbox{ } (\leq)\mbox{ } p(F).
$$ 
\end{definition}

\begin{prop}\label{semi.stable.finite}
Let $F \in \FP(\infroot{X})$, and assume $F_{n}\in \FP(\radice{n}{X})$ is such that $\pi_{n}^{*}F_{n} \simeq F$. Then $F$ is semi-stable if and only if $F_{n}$ is semi-stable on $\radice{n}{X}$. The ``only if'' part is true with ``semi-stable'' replaced by ``stable''.
\end{prop}

\begin{proof}
If $\pi_{n}^{*}F_{n}$ is (semi-)stable, then since $\pi_{n}^{*}$ is fully faithful and $\pi_{n}$ is flat, if $G\subseteq F_{n}$ is a non-zero proper subsheaf, then $\pi_{n}^{*}G \subseteq \pi_{n}^{*}F_{n}$ is also a non-zero proper subsheaf, and
$$
p_{n}(G)=p(\pi_{n}^{*}G) \mbox{ } (\leq) \mbox{ } p(\pi_{n}^{*}F_{n})=p_{n}(F_{n}).
$$

On the other hand, if $F_{n}$ is semi-stable, consider a finitely presented subsheaf $G \subseteq \pi^{*}F_{n}$. Since it is finitely presented, $G$ will come from some $G_{m} \in \FP(\radice{m}{X})$. By pushing forward the inclusion $$\pi_{m}^{*}G_{m}\subseteq \pi_{n}^{*}F_{n}$$to $\radice{q}{X}$ where $q=\lcm(m,n)$ and using the projection formula for $\pi_{q}$, we see that $G$ comes from $\pi_{m,q}^{*}G_{m}\subseteq \pi_{n,q}^{*}F_{n}$. Since by Proposition \ref{stability.preserved} $\pi_{n,q}^{*}F_{n}$ is semi-stable on $\radice{q}{X}$, we see that
$$
p(G)=p_{q}(\pi_{m,q}^{*}G_{q})\leq p_{q}(\pi_{n,q}^{*}F_{n})=p(\pi_{n}^{*}F_{n})
$$
so $\pi_{n}^{*}F_{n}$ is semi-stable.
\end{proof}

\begin{rmk}
The previous statement is false for stable sheaves in general, and there are stable sheaves $F_{n}\in \FP(\radice{n}{X})$ such that $\pi_{n}^{*}F_{n}$ is not stable. Indeed, this will happen if $\pi_{n,m}^{*}F_{n}$ is not stable for some $m=nq$.
\end{rmk}

We consider the stack $\mc{FP}_{\infroot{X}}$ over $\aff$ of finitely presented sheaves on $\infroot{X}$, defined as follows: an object over $A \in \aff$ is a finitely presented sheaf on $(\infroot{X})_{A}=\infroot{X}\times_{k}\spec(A)$, flat over $A$, and arrows are defined using pullback along $(\infroot{X})_{B}\to (\infroot{X})_{A}$ for a homomorphism $A\to B$. 

Inside $\mc{FP}_{\infroot{X}}$ there is a subcategory parametrizing families of semi-stable sheaves: define $\mc{M}^{ss}$ (resp. $\mc{M}^{s}$) as the stack over $\aff$ with objects over $A \in \aff$ finitely presented sheaves $F$ on $(\infroot{X})_{A}$, flat over $\spec(A)$, and such that for every geometric point $p\to \spec(A)$, the pullback of $F$ to $(\infroot{X})_{p}$ is semi-stable (resp. stable).

In the rest of this section we will prove the following theorem.

\begin{thm}\label{thm.parabolic.rational}
Let $X$ be a polarized simplicial log scheme over $k$ such that in the global chart $P\to \Div(X)$ the monoid $P$ is simplicial. The stack $\mc{M}^{ss}$ is an Artin stack, locally of finite type, and it has an open substack $\mc{M}^{s}\subseteq \mc{M}^{ss}$ parametrizing stable sheaves.

If in addition stability is preserved by pullback along the projections $\radice{m}{X}\to \radice{n}{X}$ between the root stacks of $X$ (for example if the log structure of $X$ is generically trivial, and we are considering torsion-free sheaves), then $\mc{M}^{ss}$ has a good moduli space $M^{ss}$, which is a disjoint union of projective schemes. Moreover there is an open subscheme $M^{s}\subseteq M^{ss}$ that is a coarse moduli space for the substack $\mc{M}^{s}$, and $\mc{M}^{s}\to M^{s}$ is a $\bb{G}_{m}$-gerbe.
\end{thm}

Let us start by relating the stack of (semi-)stable parabolic sheaves on $\infroot{X}$ with the ones at finite level.

\begin{prop}\mbox{}\label{direct.limit.ss}
\begin{itemize}
\item We have a natural isomorphism of stacks over $\aff$
$$
\varinjlim_{n} \mc{M}^{ss}_{n}\to \mc{M}^{ss}.
$$
\item If pullbacks preserve stability, then we also have an isomorphism
$$
\varinjlim_{n} \mc{M}^{s}_{n}\to \mc{M}^{s}
$$
which is compatible with the previous one. Moreover, in this last case the transition maps are open and closed immersions, so $\mc{M}^{ss}$ and $\mc{M}^{s}$ are in fact a union of connected components of the stacks at finite level.
\end{itemize}
\end{prop}

\begin{proof}
Let us spend a few words on the definition of the direct limit $\varinjlim_{n}\mc{M}^{ss}_{n}$.

Given in general a filtered directed system $\{\mc{C}_{i}\}_{i \in I}$ of fibered categories over some category $\mc{D}$, we can define the direct limit $\mc{C}=\varinjlim_{i\in I} \mc{C}_{i}$ as a fibered category over $\mc{D}$ as follows: objects are pairs $(d,c)$, where $d \in \mc{D}$ and $c \in \mc{C}_{i}(d)$ for some $i \in I$, and a morphism $(d,c)\to (d',c')$ is a pair $(f,g)$ where $f\colon d\to d'$ is a morphism in $\mc{D}$, and $g$ is an element of the direct limit $\varinjlim_{i \geq i_{0}} \Hom(\phi_{i_{0},i}(f^{*}c'),\phi_{i_{0},i}(c))$, where $i_{0}$ is an index where both $c$ and $c'$ are defined. In other words we are taking the disjoint union of the objects and the direct limit for morphisms, fiberwise. If $\mc{D}$ is a site and $\mc{C}_{i}$ are stacks, we can stackify $\mc{C}$ to get the direct limit as a stack.

In our particular case note that the direct limit is already a stack: this is because, since we are working on $(\Aff)^{op}$, every covering has a finite refinement, so we can reduce effectivity of descent data and the fact that Hom is a sheaf to some finite level. Moreover, since all the maps $\iota_{n,m}$ are fully faithful, in the direct limit we have $\Hom(F_{n},F_{m})=\Hom_{\mc{M}^{ss}_{h}}(\pi_{n,h}^{*}F_{n}, \pi_{m,h}^{*}F_{m})$, where $h=\lcm(n,m)$.

Now for every $n \in \bb{N}$ the pullback along $\pi_{n}\colon \infroot{X}\to \radice{n}{X}$ induces $\iota_{n}\colon \mc{M}^{ss}_{n}\to \mc{M}^{ss}$, and moreover these maps are compatible with the transition maps of the system $\iota_{n,m}$. Thus we have a morphism
$$
\iota\colon \varinjlim_{n} \mc{M}^{ss}_{n}\to \mc{M}^{ss}.
$$
We will check that this is fully faithful and essentially surjective.

Take a $k$-algebra $A$, and consider the map $(\varinjlim_{n} \mc{M}^{ss}_{n})(A)\to \mc{M}^{ss}(A)$. If $F$ is an object of $\mc{M}^{ss}(A)$, i.e. a finitely presented sheaf on $(\infroot{X})_{A}=\infroot{X}\times_{k}\spec(A)$, then since $(\infroot{X})_{A}=\varprojlim_{n}(\radice{n}{X})_{A}$, we have $\FP((\infroot{X})_{A})=\varinjlim \FP((\radice{n}{X})_{A})$, and $F$ comes from some $F_{n}\in \FP((\radice{n}{X})_{A})$. Moreover by possibly increasing $n$ we can assume that $F_{n}$ is flat over $A$, and its fibers over $(X_n)_{p}$ for geometric points $p\to \spec(A)$  will be semi-stable by Proposition \ref{semi.stable.finite}, since their pullback to $(\infroot{X})_{p}$ is. In other words $F_{n}$ is an object of $\mc{M}^{ss}_{n}(A)$, and its image in $(\varinjlim_{n}\mc{M}^{ss}_{n})(A)$ via $\iota$ will be isomorphic to $F$.

For full faithfulness, if $F_{n}$ and $F_{m}$ are two objects of $(\varinjlim_{n}\mc{M}^{ss}_{n})(A)$, as noted above we have $\Hom(F_{n},F_{m})=\Hom_{\mc{M}^{ss}_{h}}(\pi_{n,h}^{*}F_{n}, \pi_{m,h}^{*}F_{m})$ with $h=\lcm(n,m)$, and since pullback along $(\infroot{X})_{A}\to (X_{h})_{A}$ is fully faithful, the conclusion follows.

The same line of reasoning works for the statement about stable sheaves, and compatibility of the maps is immediate from the compatibility at finite level.
\end{proof}

\begin{rmk}\label{stable.open}
What perhaps is not clear enough, is that $j\colon \mc{M}^{s}\subseteq \mc{M}^{ss}$ is an open substack. This holds even if stability is not preserved by pullback.

In fact, take a morphism $f\colon T \to \mc{M}^{ss}$, and note that we can assume that $T$ is affine, say $T=\spec(A)$. The map $f$ corresponds to a sheaf $F \in \mc{M}^{ss}(A)$, and by the preceding proposition $F$ will come from some $F_{n}\in \mc{M}^{ss}_{n}(A)$.

From this, and the observation that $j^{-1}\mc{M}^{ss}_{n} \simeq \mc{M}^{s}_{n}$, we see that the fibered product $\mc{M}^{s}\times_{\mc{M}^{ss}}T$ coincides with $\mc{M}^{s}\times_{\mc{M}^{ss}_{n}}T=j^{-1}\mc{M}^{ss}_{n}\times_{\mc{M}^{ss}_{n}}T=\mc{M}^{s}_{n}\times_{\mc{M}^{ss}_{n}}T$, and this is open in $T$ because $\mc{M}^{s}_{n}\to \mc{M}^{ss}_{n}$ is an open immersion.
\end{rmk}

\begin{prop}\label{infty.algebraic}
The stack $\mc{M}^{ss}$ is an Artin stack, locally of finite presentation over $k$. Being an open substack, $\mc{M}^{s}$ has the same properties.
\end{prop}

\begin{lemma}
For any $n \in \bb{N}$ the morphism $\iota_{n}\colon \mc{M}_{n}^{ss}\to \mc{M}^{ss}$ induced by pullback is an open immersion.
\end{lemma}

\begin{proof}
This goes exactly as the proof of Lemma \ref{semistable.open}. The main point is that, by the projection formula for $\pi_{n}\colon \infroot{X}\to \radice{n}{X}$, a finitely presented sheaf $F \in \FP(\infroot{X})$ comes from $\radice{n}{X}$ if and only if the adjunction morphism $\pi_{n}^{*}{\pi_{n}}_{*}F\to F$ is an isomorphism.
\end{proof}

\begin{proof}[Proof of Proposition \ref{infty.algebraic}]
Let us fix a smooth presentation $A_{n}\to \mc{M}^{ss}_{n}$ for every $n \in \bb{N}$. We have a natural induced map $\varphi\colon A=\bigsqcup_{n} A_{n}\to \varinjlim_{n} \mc{M}^{ss}_{n}=\mc{M}^{ss}$, and this is a smooth presentation for $\mc{M}^{ss}$.

Indeed, $\varphi$ is surjective since $\mc{M}^{ss}$ is a union of the open substacks $\mc{M}^{ss}_{n}$, and $A_{n}\to \mc{M}^{ss}_{n}$ is surjective, and for a morphism $f\colon T\to \mc{M}^{ss}$ from a scheme $T$ we have $A \times_{\mc{M}^{ss}}T=\bigsqcup_{n}(A_{n}\times_{\mc{M}^{ss}}T)$, so we can consider a single piece $A_{n}\times_{\mc{M}^{ss}}T$. Now it suffices to note that the morphism $A_{n}\to \mc{M}^{ss}_{n}\subseteq \mc{M}^{ss}$ is a composition of two smooth representable morphisms.

The fact that the diagonal $\Delta\colon \mc{M}^{ss}\to \mc{M}^{ss}\times_{k}\mc{M}^{ss}$ is representable follows from a similar argument.
\end{proof}

\subsubsection{What invariants can we fix?}

Before we go further, let us briefly consider the following problem.

\begin{question}
Can we fix some invariants for finitely presented sheaves on $\infroot{X}$, in order to cut out a finite-type moduli stack inside $\mc{M}^{ss}$?
\end{question}

Ideally, since we are taking a limit of the theories at finite level, we would like to fix some invariant of coherent sheaves on the $\radice{n}{X}$'s, that is preserved by pullback along the maps $\pi_{n,m}\colon \radice{m}{X}\to \radice{n}{X}$.

The stacks $\mc{M}^{ss}_{n}$ we considered up to this point are not of finite type themselves, and the standard solution when one studies moduli of coherent sheaves is, as we did in the previous section, to fix the Hilbert polynomial $H \in \bb{Q}[x]$. This gives finite-type components, both of the corresponding moduli stack and of its good moduli space. There are other things one can fix, for example Chern classes up to numerical equivalence, but here we will focus mainly on Hilbert polynomials. 

It is clear that we cannot fix the Hilbert polynomial at the limit, since it is not preserved by pullback. Rather, we have $P_{m}(\pi^{*}_{n,m}(F))=q \cdot P_{n}(F)$, where $k$ is such that $m=nq$. On the other hand the \emph{reduced} Hilbert polynomial $h$ is preserved by pullback, i.e. $p_{m}(\pi^{*}_{n,m}F)=p_{n}(F)$ for any $F \in \FP(\radice{n}{X})$ (Proposition \ref{stability.preserved}).

\begin{notation}
We denote by $\mc{M}^{ss}_{h,n}$ and $\mc{M}^{s}_{h,n}$ the stacks of families of (semi-)stable parabolic sheaves on $\radice{n}{X}$ with reduced Hilbert polynomial $h \in \bb{Q}[x]$. They have good moduli spaces $M^{ss}_{h,n}$ and $M^{s}_{h,n}$, and since the reduced Hilbert polynomial is preserved by pullback, the morphisms $\iota_{n,m}\colon \mc{M}^{ss}_{n}\to \mc{M}^{ss}_{m}$ restricts to morphisms $\mc{M}^{ss}_{h,n}\to \mc{M}^{ss}_{h,m}$, which we will still denote by $\iota_{n,m}$. The same goes for the morphism $i_{n,m}$, and also for $\iota^{o}_{n,m}$, $i^{o}_{n,m}$ when they are defined.
\end{notation}

Exactly as in Proposition \ref{direct.limit.ss}, we have an isomorphism
$$
\varinjlim_{n}\mc{M}^{ss}_{h,n} \simeq \mc{M}^{ss}_{h}
$$
and the analogous one for stable sheaves if stability is preserved by pullback.

This all works well with the direct limit, but there is an issue: $\mc{M}^{ss}_{h,m}$ is not necessarily of finite type. In fact, fixing the reduced Hilbert polynomial $h$ does not fix the rank (say we are considering torsion-free sheaves), like it happens with the ordinary Hilbert polynomial, and the rank can become arbitrarily large, without changing $h$. In other words, we have $\mc{M}^{ss}_{h,n}=\bigsqcup_{H}\mc{M}^{ss}_{H,n}$ where the union ranges over $H \in \bb{Q}[x]$ of degree $d$ such that $H/\alpha=h$, where $\alpha$ is $d!$ times the leading coefficient of $H$. This disjoint union is usually not of finite type (for an explicit example, take $X$ to be the standard log point).

In the case where the logarithmic structure of $X$ is generically trivial, $X$ is integral and we are considering torsion-free sheaves, there is another thing that we can fix and that gives intermediate components of finite type, namely the rank of the pushforward of the sheaf to $X$. In fact, since $\pi\colon \radice{n}{X}\to X$ is generically an isomorphism, if $F \in \FP(\radice{n}{X})$ has rank $r$, the pushforward $\pi_{*}F$ will still have rank $r$. Moreover the ``parabolic'' rank of $\pi^{*}_{n,m}(F)$ (i.e. its rank as a torsion-free sheaf on the stack $\radice{m}{X}$) is easily seen to be $m\cdot r$, so fixing $h$ and $r$ is equivalent to fixing $H$, and thus will give a finite-type union of components $\mc{M}^{ss}_{h,r,n}$ of $\mc{M}^{ss}_{n}$.

With these assumptions, the moduli stacks $\mc{M}^{ss}_{h,r,n}$ and $\mc{M}^{s}_{h,r,n}$ are of finite type, and the good moduli space $M^{ss}_{h,r,n}$ (resp. $M^{s}_{h,r,n}$) is projective (resp. quasi-projective).

We remark that even in this case, the ``limit'' moduli stack $\mc{M}^{ss}_{h,r}$ is not necessarily of finite type, and it can have infinitely many connected components.

\begin{example}
Take $X=\bb{P}^{1}$, with the log structure induced by the divisor $0+\infty$, and let us fix the reduced Hilbert polynomial $h(x)=x+n$ for $n \in \bb{Z}$, and rank $r=1$.

For any $m \in \bb{N}$, the parabolic sheaf
$$
\xymatrix@C=18pt{
& -1 & -\frac{m-1}{m} & -\frac{m-2}{m} &  \cdots & -\frac{2}{m} & -\frac{1}{m} & 0 \\
F_{m}= & \mc{O}(n-2) \ar[r] & \mc{O}(n-1) \ar@{=}[r] & \mc{O}(n-1) \ar@{=}[r] & \cdots \ar@{=}[r] & \mc{O}(n-1)\ar[r] & \mc{O}(n) \ar@{=}[r] & \mc{O}(n)
}
$$
on $\radice{m}{X}$ has reduced Hilbert polynomial
$$
p_{m}(F_{m})(x)=\frac{x+n-1+(m-2)(x+n)+x+n+1}{m}=x+n=h(x)
$$
and rank $1$, so it gives is a point of $\mc{M}^{ss}_{h,1}$, and it sits in the substack $\mc{M}^{ss}_{h,1,m}$. Moreover, it is not in any of the $\mc{M}^{ss}_{h,1,j}$ with $j \mid m$ (otherwise the only two non-identity maps would need to be the identity), and so it not in a connected component coming from lower levels, since in this case the immersions are open and closed. This shows that there are infinitely many components in $\mc{M}^{ss}_{h,1}$, and so it is not of finite type.
\end{example}

\subsubsection{Taking the limit}

Let us finish the proof of Theorem \ref{thm.parabolic.rational}. All that is left to do is to take the direct limit $\varinjlim_{n} M^{ss}_{n}$ (and the analogous for the spaces of stable sheaves), and check that this is a good moduli space for $\mc{M}^{ss}$. We need the following lemma, whose proof is easy and left to the reader.

\begin{lemma}
Let $\{M_{i}\}_{i \in I}$ be a filtered directed system of locally noetherian schemes, where every transition map is an open and closed immersion. Then the ind-scheme $\varinjlim_{i \in I} M_{i}$ is isomorphic as an ind-scheme to a disjoint union of connected components of the $M_{i}$'s, and in particular it is a scheme.\qed
\end{lemma}

\begin{proof}[Proof of \ref{thm.parabolic.rational}]
The fact that $\mc{M}^{ss}$ is an Artin stack locally of finite presentation is in Proposition \ref{infty.algebraic}, and the fact that stable sheaves form an open substack is explained in Remark \ref{stable.open}.

As for the second part, assume that stability is preserved by pullback, so that all the maps between the stacks and moduli spaces are open and closed immersions. In particular $\mc{M}^{ss}$ and $\mc{M}^{s}$ will be ascending unions of (open and closed) substacks, where at each step we might add some new connected components. Let $M^{ss}$ be the direct limit of the system $\{M^{ss}_{n}\}_{n \in \bb{N}}$ of good moduli spaces at finite level. By the preceding lemma, this will be a disjoint union $M^{ss}=\bigsqcup_{i} M_{i}$, where each $M_{i}$ is a connected component of some $M^{ss}_{n}$.

We have a natural map $\mc{M}^{ss}\to M^{ss}$, and since being a good moduli space is a local property, we can restrict to a single component $M_{i}$, say it comes from $M^{ss}_{\overline{n}}$. Now the fibered product $\mc{M}^{ss}\times_{M^{ss}}M_{i}$ will clearly be the connected component of $\mc{M}^{ss}$ (coming from $\mc{M}^{ss}_{\overline{n}}$ and) corresponding to $M_{i}$, and so the projection $\mc{M}^{ss}\times_{M^{ss}}M_{i}\to M_{i}$ is a good moduli space, because it is a good moduli space at level $\overline{n}$.

The remaining statements about the substack of stable sheaves follow in the same way from the corresponding statement for the stacks $\mc{M}^{s}_{n}$ at finite level.
\end{proof}

\begin{rmk}
As we already mentioned, in the general case in which stability is not preserved by pullback it would be natural to conjecture that, although the morphisms $\mc{M}^{ss}_{n}\to \mc{M}^{ss}_{m}$ are not open and closed immersions, the maps between the good moduli spaces $M^{ss}_{n}\to M^{ss}_{m}$ are, and the direct limit still makes sense as a scheme.

Without knowing this, one can still wonder if formally the direct limit $\varinjlim_{n}M^{ss}_{n}$ (seen as an ind-scheme) can be considered a moduli space for semi-stable parabolic sheaves with rational weights.

One can show that this direct limit has some universal property with respect to morphism $\mc{M}^{ss}\to N$ where $N$ is an ind-algebraic space, and that if $\mc{M}^{ss}$ admits a good moduli space $M$, then $M$ (seen as an ind-scheme) also enjoys this universal property, so it has to be isomorphic to $\varinjlim_{n}M^{ss}_{n}$. For the details we refer to \cite[Section 4.3.2]{tesi}.

This suggests that even if it might not be a scheme, the direct limit $\varinjlim_{n}M^{ss}_{n}$ can be rightfully regarded as a moduli space for semi-stable parabolic sheaves with rational weights.
\end{rmk}

\bibliography{biblio}

\bibliographystyle{alpha}

\end{document}